\documentclass[reqno]{amsart}

\usepackage{amsmath}
\usepackage{amssymb}
\usepackage{verbatim}
\usepackage{dsfont}
\usepackage{enumerate}
\usepackage[all]{xy}
\usepackage[mathscr]{eucal}
\usepackage{amsthm}
\usepackage{stmaryrd}
\usepackage{amscd}
\usepackage{amsfonts}
\usepackage{latexsym}
\usepackage{amscd}
\usepackage{graphicx}
\usepackage{tikz}
\usepackage{txfonts}
\usepackage{setspace}
\usepackage{wasysym}
\usepackage{hyperref}
\usepackage{pdfpages}
\usetikzlibrary{decorations.pathmorphing}

\makeatletter

\newcommand{\arxiv}[1]{\href{http://arxiv.org/abs/#1}{\tt
    arXiv:\nolinkurl{#1}}}

\def\dot{{\color{white}\bullet}\!\!\!\circ}
\def\heart{{\color{white}\scriptstyle\varheartsuit}\hspace{-1.487mm}\scriptstyle\heartsuit}
\def\diamond{{\color{white}\scriptstyle\vardiamondsuit}\hspace{-1.308mm}\scriptstyle\diamondsuit}

\def\smallclock{\begin{tikzpicture}
\filldraw[white] (0,0) circle (1mm);
\draw[-,thin] (0,-0.1) to[out=180,in=-90] (-.1,0);
\draw[-,thin] (-0.1,0) to[out=90,in=180] (0,0.1);
\draw[->,thin] (0,0.1) to[out=0,in=90] (0.09,-0.04);
\draw[-,thin] (0.09,-0.02) to[out=-96,in=0] (0,-0.1);
\end{tikzpicture}}

\def\clockright{\begin{tikzpicture}[baseline=-.9mm]
\filldraw[white] (0,0) circle (1.72mm);
\draw[-,thin] (0,-0.18) to[out=180,in=-90] (-.18,0);
\draw[-,thin] (-0.18,0) to[out=90,in=180] (0,0.18);
\draw[->,thin] (0,0.18) to[out=0,in=90] (0.18,0);
\draw[-,thin] (0.178,0.02) to[out=-78,in=0] (0,-0.18);
\end{tikzpicture}}

\def\clockplus{\begin{tikzpicture}[baseline=-.9mm]
\filldraw[white] (0,0) circle (1.72mm);
\draw[-,thin] (0,-0.18) to[out=180,in=-90] (-.18,0);
\draw[-,thin] (-0.18,0) to[out=90,in=180] (0,0.18);
\draw[->,thin] (0,0.18) to[out=0,in=90] (0.18,0);
\draw[-,thin] (0.178,0.02) to[out=-78,in=0] (0,-0.18);
\node at (0,0.01) {$+$};
\end{tikzpicture}}

\def\clockminus{\begin{tikzpicture}[baseline=-.9mm]
\filldraw[white] (0,0) circle (1.72mm);
\draw[-,thin] (0,-0.18) to[out=180,in=-90] (-.18,0);
\draw[-,thin] (-0.18,0) to[out=90,in=180] (0,0.18);
\draw[->,thin] (0,0.18) to[out=0,in=90] (0.18,0);
\draw[-,thin] (0.178,0.02) to[out=-78,in=0] (0,-0.18);
\node at (0,0.01) {$-$};
\end{tikzpicture}}

\def\clockplusminus{\begin{tikzpicture}[baseline=-.9mm]
\filldraw[white] (0,0) circle (1.72mm);
\draw[-,thin] (0,-0.18) to[out=180,in=-90] (-.18,0);
\draw[-,thin] (-0.18,0) to[out=90,in=180] (0,0.18);
\draw[->,thin] (0,0.18) to[out=0,in=90] (0.18,0);
\draw[-,thin] (0.178,0.02) to[out=-78,in=0] (0,-0.18);
\node at (0,0.01) {$\pm$};
\end{tikzpicture}}

\def\clocktop{\begin{tikzpicture}[baseline=-.9mm]
\filldraw[white] (0,0) circle (1.72mm);
\draw[-,thin] (0,-0.18) to[out=180,in=-90] (-.18,0);
\draw[->,thin] (-0.18,0) to[out=90,in=180] (0,0.18);
\draw[-,thin] (-0.02,0.178) to[out=12,in=90] (0.18,0);
\draw[-,thin] (0.18,0) to[out=-90,in=0] (0,-0.18);
\end{tikzpicture}}

\def\anticlockright{\begin{tikzpicture}[baseline=-.9mm]
\filldraw[white] (0,0) circle (1.72mm);
\draw[-,thin] (0,-0.18) to[out=180,in=-90] (-.18,0);
\draw[-,thin] (-0.18,0) to[out=90,in=180] (0,0.18);
\draw[-,thin] (0,0.18) to[out=0,in=78] (0.178,-0.02);
\draw[<-,thin] (0.18,0) to[out=-90,in=0] (0,-0.18);
\end{tikzpicture}}

\def\anticlockleft{\begin{tikzpicture}[baseline=-.9mm]
\filldraw[white] (0,0) circle (1.72mm);
\draw[-,thin] (0,-0.18) to[out=180,in=-102] (-.178,0.02);
\draw[<-,thin] (-0.18,0) to[out=90,in=180] (0,0.18);
\draw[-,thin] (0.18,0) to[out=-90,in=0] (0,-0.18);
\draw[-,thin] (0,0.18) to[out=0,in=90] (0.18,0);
\end{tikzpicture}}

\def\anticlockplus{\begin{tikzpicture}[baseline=-.9mm]
\filldraw[white] (0,0) circle (1.72mm);
\draw[-,thin] (0,-0.18) to[out=180,in=-102] (-.178,0.02);
\draw[<-,thin] (-0.18,0) to[out=90,in=180] (0,0.18);
\draw[-,thin] (0.18,0) to[out=-90,in=0] (0,-0.18);
\draw[-,thin] (0,0.18) to[out=0,in=90] (0.18,0);
\node at (0,0.01) {$+$};
\end{tikzpicture}}

\def\anticlockminus{\begin{tikzpicture}[baseline=-.9mm]
\filldraw[white] (0,0) circle (1.72mm);
\draw[-,thin] (0,-0.18) to[out=180,in=-102] (-.178,0.02);
\draw[<-,thin] (-0.18,0) to[out=90,in=180] (0,0.18);
\draw[-,thin] (0.18,0) to[out=-90,in=0] (0,-0.18);
\draw[-,thin] (0,0.18) to[out=0,in=90] (0.18,0);
\node at (0,0.01) {$-$};
\end{tikzpicture}}

\def\anticlockplusminus{\begin{tikzpicture}[baseline=-.9mm]
\filldraw[white] (0,0) circle (1.72mm);
\draw[-,thin] (0,-0.18) to[out=180,in=-102] (-.178,0.02);
\draw[<-,thin] (-0.18,0) to[out=90,in=180] (0,0.18);
\draw[-,thin] (0.18,0) to[out=-90,in=0] (0,-0.18);
\draw[-,thin] (0,0.18) to[out=0,in=90] (0.18,0);
\node at (0,0.01) {$\pm$};
\end{tikzpicture}}

\def\posleft{\mathord{\begin{tikzpicture}[baseline = 0]
	\draw[thin,<-] (-0.18,-.1) to (0.18,.3);
	\draw[-,white,line width=4pt] (0.18,-.1) to (-0.18,.3);
	\draw[->,thin] (0.18,-.1) to (-0.18,.3);
\end{tikzpicture}}}
\def\negleft{\mathord{\begin{tikzpicture}[baseline = 0]
	\draw[->,thin] (0.18,-.1) to (-0.18,.3);
		\draw[-,white,line width=4pt] (-0.18,-.1) to (0.18,.3);
	\draw[thin,<-] (-0.18,-.1) to (0.18,.3);
\end{tikzpicture}}}
\def\rightcup{\:\begin{tikzpicture}[baseline = .5mm]
	\draw[<-,thin] (0.35,0.3) to[out=-90, in=0] (0.1,0);
	\draw[-,thin] (0.1,0) to[out = 180, in = -90] (-0.15,0.3);
\end{tikzpicture}\:}
\def\rightcap{\:\begin{tikzpicture}[baseline = .5mm]
	\draw[<-,thin] (0.35,0) to[out=90, in=0] (0.1,0.3);
	\draw[-,thin] (0.1,0.3) to[out = 180, in = 90] (-0.15,0);
\end{tikzpicture}\:}
\def\leftcup{\:\begin{tikzpicture}[baseline = .5mm]
	\draw[-,thin] (0.35,0.3) to[out=-90, in=0] (0.1,0);
	\draw[->,thin] (0.1,0) to[out = 180, in = -90] (-0.15,0.3);
\end{tikzpicture}\:}
\def\leftcap{\:\begin{tikzpicture}[baseline = .5mm]
	\draw[-,thin] (0.35,0) to[out=90, in=0] (0.1,0.3);
	\draw[->,thin] (0.1,0.3) to[out = 180, in = 90] (-0.15,0);
\end{tikzpicture}\:}

\def\dot{{\color{white}\bullet}\!\!\!\circ}
\definecolor{darkblue}{HTML}{000000}
\def\red#1{{\color{red} #1}}
\def\blue#1{{\color{blue} #1}}
\def\up{\uparrow}
\def\down{\downarrow}

\newtheorem{theorem}{Theorem}[section]
\newtheorem{lemma}[theorem]{Lemma}
\newtheorem{corollary}[theorem]{Corollary} 
\theoremstyle{definition}  
\newtheorem{definition}[theorem]{Definition}
\newtheorem{remark}[theorem]{Remark}
\def\Uq{U_{\!q}}

\@addtoreset{equation}{section}

\def\coind{\operatorname{coind}}
\def\ind{\operatorname{ind}}
\def\res{\operatorname{res}}
\def\e{e}
\def\h{h}
\def\tr{\operatorname{tr}}
\def\Sym{\operatorname{Sym}}

\def\SG{\mathfrak{S}}

\def\HEIS{\mathcal{H}eis}
\def\H{\mathcal{H}}
\def\A{\mathbb{O}}
\def\B{\widetilde{\mathbb{O}}}
\def\AH{\mathcal{AH}}
\def\OS{\mathcal{OS}}
\def\AOS{\mathcal{AOS}}
\newcommand{\End}{\operatorname{End}}
 
\newcommand{\unit}{\mathds{1}}
\def\pr{\operatorname{pr}}
\newcommand{\Add}{\operatorname{Add}}
\def\Kar{\operatorname{Kar}}
\newcommand{\Hom}{\operatorname{Hom}} 
\newcommand{\Mod}{\operatorname{\!-mod}}
\renewcommand{\mod}{\operatorname{\!-mod}}
\newcommand{\proj}{\operatorname{\!-pmod}}
\newcommand{\id}{\text{id}}

\newcommand{\Z}{\mathbb{Z}}
\newcommand{\N}{\mathbb{N}}

\newcommand{\Q}{\mathbb{Q}}
\newcommand{\K}{\mathbb{K}}
\newcommand{\eps}{\varepsilon}

\newcommand{\Cherednik}{\mathsf{H}}
\newcommand{\KZ}{\mathsf{KZ}}
\newcommand{\cO}{\mathcal{O}}
\renewcommand{\k}{\Bbbk}
\def\eta{\operatorname{flip}}

\allowdisplaybreaks

\begin{document}

\title[Quantum Heisenberg category]{\boldmath On the definition of 
  quantum Heisenberg category}

\author[J.~Brundan]{Jonathan Brundan}
\address{Department of Mathematics\\
University of Oregon\\ Eugene\\ OR 97403\\ USA}
\email{brundan@uoregon.edu}

\author[A.~Savage]{Alistair Savage}
\address{
  Department of Mathematics and Statistics \\
  University of Ottawa\\
  Ottawa, ON\\ Canada
}
\urladdr{\href{http://alistairsavage.ca}{alistairsavage.ca}, \textrm{\textit{ORCiD}:} \href{https://orcid.org/0000-0002-2859-0239}{orcid.org/0000-0002-2859-0239}}
\email{alistair.savage@uottawa.ca}

\author[B.~Webster]{Ben Webster}
\address{Department of Pure Mathematics, University of Waterloo \&
  Perimeter Institute for Theoretical Physics\\
Waterloo, ON\\ Canada}
\email{ben.webster@uwaterloo.ca}

\thanks{2010 {\it Mathematics Subject Classification}: 17B10, 18D10.}
\thanks{
This article is based upon work done
 while the first author was in
residence at the Mathematical Sciences Research Institute in Berkeley,
California during the Spring 2018 semester, which was supported 
by the National Science Foundation 
grant DMS-1440140.
The first author was also supported by DMS-1700905.
The second and third authors were supported by Discovery Grants from
the Natural Sciences and Engineering Research Council of Canada.  The
research of the third author was also supported in part by Perimeter Institute for Theoretical Physics. Research at Perimeter Institute is supported by the Government of Canada through the Department of Innovation, Science and Economic Development Canada and by the Province of Ontario through the Ministry of Research, Innovation and Science.
}

\begin{abstract}
We introduce a diagrammatic
monoidal category $\HEIS_k(z,t)$
which we call the \emph{quantum Heisenberg category};
here, $k \in \mathbb{Z}$ is ``central charge'' and $z$ and $t$ are invertible
parameters.
Special cases were known before: for central charge $k=-1$ and parameters
$z = q-q^{-1}$ and $t =
-z^{-1}$
our quantum Heisenberg category may be obtained from
the deformed version of Khovanov's Heisenberg category introduced by Licata and
Savage by inverting its polynomial generator, while
$\HEIS_0(z,t)$ is the affinization of the
HOMFLY-PT skein category.
We also prove a basis theorem for the morphism spaces in $\HEIS_k(z,t)$.
\end{abstract}

\maketitle  
\section{Introduction}

Fix a commutative ground ring $\k$ and parameters $z,t \in \k^\times$.
This paper introduces a family of pivotal monoidal categories 
$\HEIS_k(z,t)$, one for each {\em central charge} $k \in \Z$.
We refer to these categories as {\em quantum Heisenberg categories}.
The terminology is due to a connection to Khovanov's Heisenberg
category from \cite{K}: our category for central charge $k=-1$ is a
two parameter deformation of the category from {\em loc.\ cit.}, and is closely
related to the one parameter deformation introduced already 
by Licata and the second author in \cite{LS}.
The category $\HEIS_0(z,t)$ has also already appeared in the literature: it is the
{\em affine HOMFLY-PT skein category} from
\cite[$\S$4]{Bskein}.
For more general central charges, our categories are new. They were
discovered by mimicking the approach of \cite{Bheis}, where the definition of the degenerate Heisenberg
categories introduced in \cite{MS} was reformulated.

In fact, we will give three different monoidal presentations of $\HEIS_k(z,t)$. They all start from the affine Hecke
algebra $AH_n$ associated to the symmetric group $\SG_n$. It is convenient to
assemble these algebras for all $n\geq 0$ into a single monoidal category
$\AH(z)$. By definition, this is the strict $\k$-linear monoidal
category generated by one object $\up$ and two morphisms
$x:\up \rightarrow \up$ and
$\tau:\up\otimes \up \rightarrow \up\otimes \up$, subject to the
relations
\begin{align}
 \tau \circ (1_\up \otimes x) \circ \tau &= x \otimes 1_\up,\label{a}\\
 \tau\circ \tau &= z \tau + 1_{\up \otimes \up},\label{r1}\\
 (\tau \otimes 1_\up) \circ (1_\up \otimes \tau) \circ (\tau \otimes 1_\up) &= 
(1_\up \otimes \tau)  \circ (\tau \otimes 1_\up) \circ (1_\up \otimes \tau).\label{r2}
\end{align}
The second relation here implies that $\tau$ is invertible.
We also require that $x$ is invertible, i.e., there is another
generator $x^{-1}$ such that
\begin{equation}\label{dagger}
x \circ x^{-1} = x^{-1} \circ x =1_\up.
\end{equation}
Adopting the usual string calculus for strict monoidal
categories, we represent $\tau, \tau^{-1}, x$, and more generally $x^{\circ a}$ for any $a
\in \Z$,
by the
diagrams
\begin{align}
\tau &= \mathord{
\begin{tikzpicture}[baseline = -.5mm]
	\draw[->,thin,darkblue] (0.28,-.3) to (-0.28,.4);
	\draw[line width=4pt,white,-] (-0.28,-.3) to (0.28,.4);
	\draw[thin,darkblue,->] (-0.28,-.3) to (0.28,.4);
\end{tikzpicture}
}\:, &
\tau^{-1} &= 
\mathord{
\begin{tikzpicture}[baseline = -.5mm]
	\draw[thin,darkblue,->] (-0.28,-.3) to (0.28,.4);
	\draw[-,line width=4pt,white] (0.28,-.3) to (-0.28,.4);
	\draw[->,thin,darkblue] (0.28,-.3) to (-0.28,.4);
\end{tikzpicture}
}\:,
&
x &=
\mathord{
\begin{tikzpicture}[baseline = -1mm]
	\draw[<-,thin,darkblue] (0.08,.4) to (0.08,-.4);
      \node at (0.08,0) {$\dot$};
 \end{tikzpicture}
}\:,
&
x^{\circ a} &=
\mathord{
\begin{tikzpicture}[baseline = -1mm]
	\draw[<-,thin,darkblue] (0.08,.4) to (0.08,-.4);
      \node at (0.08,0) {$\dot$};
      \node at (0.28,0) {$\color{darkblue}\scriptstyle a$};
\end{tikzpicture}
}\:.
\end{align}
Then the relations (\ref{a})--(\ref{r2}) are equivalent to the following diagrammatic
relations:
\begin{align}
\mathord{
\begin{tikzpicture}[baseline = -.5mm]
	\draw[->,thin,darkblue] (0.28,-.3) to (-0.28,.4);
      \node at (0.165,-0.15) {$\dot$};
	\draw[line width=4pt,white,-] (-0.28,-.3) to (0.28,.4);
	\draw[thin,darkblue,->] (-0.28,-.3) to (0.28,.4);
\end{tikzpicture}
}&=\mathord{
\begin{tikzpicture}[baseline = -.5mm]
	\draw[thin,darkblue,->] (-0.28,-.3) to (0.28,.4);
	\draw[-,line width=4pt,white] (0.28,-.3) to (-0.28,.4);
	\draw[->,thin,darkblue] (0.28,-.3) to (-0.28,.4);
      \node at (-0.14,0.23) {$\dot$};
\end{tikzpicture}
}
\:,
&\mathord{
\begin{tikzpicture}[baseline = -.5mm]
	\draw[thin,darkblue,->] (-0.28,-.3) to (0.28,.4);
      \node at (-0.16,-0.15) {$\dot$};
	\draw[-,line width=4pt,white] (0.28,-.3) to (-0.28,.4);
	\draw[->,thin,darkblue] (0.28,-.3) to (-0.28,.4);
\end{tikzpicture}
}&= 
\mathord{
\begin{tikzpicture}[baseline = -.5mm]
	\draw[->,thin,darkblue] (0.28,-.3) to (-0.28,.4);
	\draw[line width=4pt,white,-] (-0.28,-.3) to (0.28,.4);
	\draw[thin,darkblue,->] (-0.28,-.3) to (0.28,.4);
      \node at (0.145,0.23) {$\dot$};
\end{tikzpicture}
}\:,
\label{rr3}\\
\mathord{
\begin{tikzpicture}[baseline = -.5mm]
	\draw[->,thin,darkblue] (0.28,-.3) to (-0.28,.4);
	\draw[line width=4pt,white,-] (-0.28,-.3) to (0.28,.4);
	\draw[thin,darkblue,->] (-0.28,-.3) to (0.28,.4);
\end{tikzpicture}
}&-\mathord{
\begin{tikzpicture}[baseline = -.5mm]
	\draw[thin,darkblue,->] (-0.28,-.3) to (0.28,.4);
	\draw[line width=4pt,white,-] (0.28,-.3) to (-0.28,.4);
	\draw[->,thin,darkblue] (0.28,-.3) to (-0.28,.4);
\end{tikzpicture}
}=
z\:\mathord{
\begin{tikzpicture}[baseline = -.5mm]
	\draw[->,thin,darkblue] (0.18,-.3) to (0.18,.4);
	\draw[->,thin,darkblue] (-0.18,-.3) to (-0.18,.4);
\end{tikzpicture}
}\:,\label{rr2}\\
\mathord{
\begin{tikzpicture}[baseline = -1mm]
	\draw[-,thin,darkblue] (0.28,-.6) to[out=90,in=-90] (-0.28,0);
	\draw[->,thin,darkblue] (-0.28,0) to[out=90,in=-90] (0.28,.6);
	\draw[-,line width=4pt,white] (-0.28,-.6) to[out=90,in=-90] (0.28,0);
	\draw[-,thin,darkblue] (-0.28,-.6) to[out=90,in=-90] (0.28,0);
	\draw[-,line width=4pt,white] (0.28,0) to[out=90,in=-90] (-0.28,.6);
	\draw[->,thin,darkblue] (0.28,0) to[out=90,in=-90] (-0.28,.6);
\end{tikzpicture}
}&=
\mathord{
\begin{tikzpicture}[baseline = -1mm]
	\draw[->,thin,darkblue] (0.18,-.6) to (0.18,.6);
	\draw[->,thin,darkblue] (-0.18,-.6) to (-0.18,.6);
\end{tikzpicture}
}
=
\mathord{
\begin{tikzpicture}[baseline = -1mm]
	\draw[->,thin,darkblue] (0.28,0) to[out=90,in=-90] (-0.28,.6);
	\draw[-,line width=4pt,white] (-0.28,0) to[out=90,in=-90] (0.28,.6);
	\draw[->,thin,darkblue] (-0.28,0) to[out=90,in=-90] (0.28,.6);
	\draw[-,thin,darkblue] (-0.28,-.6) to[out=90,in=-90] (0.28,0);
	\draw[-,line width=4pt,white] (0.28,-.6) to[out=90,in=-90] (-0.28,0);
	\draw[-,thin,darkblue] (0.28,-.6) to[out=90,in=-90] (-0.28,0);
\end{tikzpicture}
}\:,&
\mathord{
\begin{tikzpicture}[baseline = -1mm]
	\draw[->,thin,darkblue] (0.45,-.6) to (-0.45,.6);
        \draw[-,thin,darkblue] (0,-.6) to[out=90,in=-90] (-.45,0);
        \draw[-,line width=4pt,white] (-0.45,0) to[out=90,in=-90] (0,0.6);
        \draw[->,thin,darkblue] (-0.45,0) to[out=90,in=-90] (0,0.6);
	\draw[-,line width=4pt,white] (0.45,.6) to (-0.45,-.6);
	\draw[<-,thin,darkblue] (0.45,.6) to (-0.45,-.6);
\end{tikzpicture}
}
&=
\mathord{
\begin{tikzpicture}[baseline = -1mm]
	\draw[->,thin,darkblue] (0.45,-.6) to (-0.45,.6);
        \draw[-,line width=4pt,white] (0,-.6) to[out=90,in=-90] (.45,0);
        \draw[-,thin,darkblue] (0,-.6) to[out=90,in=-90] (.45,0);
        \draw[->,thin,darkblue] (0.45,0) to[out=90,in=-90] (0,0.6);
	\draw[-,line width=4pt,white] (0.45,.6) to (-0.45,-.6);
	\draw[<-,thin,darkblue] (0.45,.6) to (-0.45,-.6);
\end{tikzpicture}
}
\:.\:\:\:\label{rr0}
\end{align}
The affine Hecke algebra $AH_n$ itself may be identified with 
$\End_{\AH(z)}(\up^{\otimes n})$,
with its standard generators $x_i$ and $\tau_j$
coming from a dot on the $i$th string
and 
the positive crossing
of the $j$th and $(j+1)$th strings, 
respectively; our convention for this numbers strings $1,\dots,n$ from right
to left.  It is often convenient to assume (passing to a quadratic extension if necessary) that $\k$ contains a root $q$ of the quadratic equation $x^2-zx-1=0$, so that $z=q-q^{-1}$.
The quadratic relation in $AH_n$
may then be written as
$(\tau_j - q)(\tau_j+q^{-1}) = 0$. Such a choice of parameter $q$ is not needed in sections \ref{first}--\ref{third}, but is essential for the applications in sections~\ref{qgln}--\ref{sbasis}.

To obtain the quantum Heisenberg category $\HEIS_k(z,t)$ from $\AH(z)$,
we 
adjoin a right dual $\down$ to the object $\up$, i.e., we 
add an additional generating object $\down$ and additional generating
morphisms
\[
c = 
\mathord{
\begin{tikzpicture}[baseline = 1mm]
	\draw[<-,thin,darkblue] (0.4,0.4) to[out=-90, in=0] (0.1,0);
	\draw[-,thin,darkblue] (0.1,0) to[out = 180, in = -90] (-0.2,0.4);
\end{tikzpicture}
}\;:\unit \rightarrow \down \otimes \up
\quad \text{and} \quad
d=\mathord{
\begin{tikzpicture}[baseline = 1mm]
	\draw[<-,thin,darkblue] (0.4,0) to[out=90, in=0] (0.1,0.4);
	\draw[-,thin,darkblue] (0.1,0.4) to[out = 180, in = 90] (-0.2,0);
\end{tikzpicture}
}\;:\;\up\otimes \down \rightarrow \unit
\]
subject to the relations
\begin{align}
\mathord{
\begin{tikzpicture}[baseline = -.5]
  \draw[->,thin,darkblue] (0.3,0) to (0.3,.4);
	\draw[-,thin,darkblue] (0.3,0) to[out=-90, in=0] (0.1,-0.4);
	\draw[-,thin,darkblue] (0.1,-0.4) to[out = 180, in = -90] (-0.1,0);
	\draw[-,thin,darkblue] (-0.1,0) to[out=90, in=0] (-0.3,0.4);
	\draw[-,thin,darkblue] (-0.3,0.4) to[out = 180, in =90] (-0.5,0);
  \draw[-,thin,darkblue] (-0.5,0) to (-0.5,-.4);
\end{tikzpicture}
}
&=
\mathord{\begin{tikzpicture}[baseline=-.5mm]
  \draw[->,thin,darkblue] (0,-0.4) to (0,.4);
\end{tikzpicture}
}\:,
&\mathord{
\begin{tikzpicture}[baseline = -.5]
  \draw[->,thin,darkblue] (0.3,0) to (0.3,-.4);
	\draw[-,thin,darkblue] (0.3,0) to[out=90, in=0] (0.1,0.4);
	\draw[-,thin,darkblue] (0.1,0.4) to[out = 180, in = 90] (-0.1,0);
	\draw[-,thin,darkblue] (-0.1,0) to[out=-90, in=0] (-0.3,-0.4);
	\draw[-,thin,darkblue] (-0.3,-0.4) to[out = 180, in =-90] (-0.5,0);
  \draw[-,thin,darkblue] (-0.5,0) to (-0.5,.4);
\end{tikzpicture}
}
&=
\mathord{\begin{tikzpicture}[baseline=-.5]
  \draw[<-,thin,darkblue] (0,-0.4) to (0,.4);
\end{tikzpicture}
}\:.\label{rightadj}
\end{align}
Then we add several more 
generating morphisms
subject to relations which ensure that the resulting monoidal category is
strictly pivotal, and moreover that there is a distinguished isomorphism
$\up\otimes
\down\:\cong\:\down\otimes\up \:\oplus\: \unit^{\oplus k}$
if $k \geq 0$ or
$ \up\otimes \down \;\oplus\:
\unit^{\oplus(-k)}\cong\: \down\otimes \up $ if $k \leq 0$.
There are various equivalent
ways to accomplish this in practice; see sections~2--4.
In these sections, we establish the
equivalence of the three approaches, and record many
other useful 
relations which follow from the
defining ones, including the property already mentioned that
$\HEIS_k(z,t)$ admits a strictly pivotal structure.

In this paragraph, we explain the approach from section 4 
in the special case
$k=-1$.
According to Definition~\ref{def3} and (\ref{posalter}),
$\HEIS_{-1}(z,t)$ is the strict $\k$-linear monoidal category
generated by objects 
$\up,\down$ and morphisms
\[
    \begin{tikzpicture}[baseline = -.5mm]
	    \draw[->,thin,darkblue] (0.28,-.3) to (-0.28,.4);
	    \draw[line width=4pt,white,-] (-0.28,-.3) to (0.28,.4);
	    \draw[thin,darkblue,->] (-0.28,-.3) to (0.28,.4);
    \end{tikzpicture}
    ,\quad
    \begin{tikzpicture}[baseline = -.5mm]
	    \draw[thin,darkblue,->] (-0.28,-.3) to (0.28,.4);
	    \draw[-,line width=4pt,white] (0.28,-.3) to (-0.28,.4);
	    \draw[->,thin,darkblue] (0.28,-.3) to (-0.28,.4);
    \end{tikzpicture}
    ,\quad
    \begin{tikzpicture}[baseline = 1mm]
	    \draw[<-,thin,darkblue] (0.4,0) to[out=90, in=0] (0.1,0.4);
	    \draw[-,thin,darkblue] (0.1,0.4) to[out = 180, in = 90] (-0.2,0);
    \end{tikzpicture}
    \ ,\quad
    \begin{tikzpicture}[baseline = 1mm]
	    \draw[<-,thin,darkblue] (0.4,0.4) to[out=-90, in=0] (0.1,0);
	    \draw[-,thin,darkblue] (0.1,0) to[out = 180, in = -90] (-0.2,0.4);
    \end{tikzpicture}
    \ ,\quad
    \begin{tikzpicture}[baseline = 1mm]
	    \draw[-,thin,darkblue] (0.4,0) to[out=90, in=0] (0.1,0.4);
	    \draw[->,thin,darkblue] (0.1,0.4) to[out = 180, in = 90] (-0.2,0);
    \end{tikzpicture}
    \quad \text{and} \quad
    \begin{tikzpicture}[baseline = 1mm]
	    \draw[-,thin,darkblue] (0.4,0.4) to[out=-90, in=0] (0.1,0);
	    \draw[->,thin,darkblue] (0.1,0) to[out = 180, in = -90] (-0.2,0.4);
    \end{tikzpicture}
\]
subject to (\ref{rr2})--(\ref{rightadj}),
the relations
\begin{align*}
\qquad\quad
\mathord{
\begin{tikzpicture}[baseline=-.5mm]
\draw[-,thin,darkblue] (.5,.8) to[out=-90,in=90] (.5,.6);
\draw[-,thin,darkblue] (.5,.6) to[out=-90,in=0] (.3,.2);
\draw[-,thin,darkblue] (.3,.2) to[out=180,in=0] (-.3,.6);
\draw[-,thin,darkblue] (-.3,.6) to[out=180,in=90] (-.5,.4);
\draw[-,thin,darkblue] (-.5,.4) to[out=-90,in=90] (-.5,-.4);
\draw[-,thin,darkblue] (-.5,-.4) to[out=-90,in=180] (-.3,-.6);
\draw[-,thin,darkblue] (.3,-.2) to[out=0,in=90] (.5,-.4);
\draw[->,thin,darkblue] (.5,-.4) to[out=-90,in=90] (.5,-.8);
\draw[-,thin,darkblue] (.2,-.8) to[out=90,in=-90] (-.2,0);
\draw[-,line width=4pt,white] (-.2,0) to[out=90,in=-90] (.2,0.8);
\draw[->,thin,darkblue] (-.2,0) to[out=90,in=-90] (.2,0.8);
\draw[-,line width=4pt,white] (-.3,-.6) to[out=0,in=180] (.3,-.2);
\draw[-,thin,darkblue] (-.3,-.6) to[out=0,in=180] (.3,-.2);
\end{tikzpicture}
}
&=\mathord{
\begin{tikzpicture}[baseline = 0]
	\draw[<-,thin,darkblue] (0.08,-.6) to (0.08,.6);
	\draw[->,thin,darkblue] (-0.28,-.6) to (-0.28,.6);
\end{tikzpicture}
}\:,&
\mathord{
\begin{tikzpicture}[baseline=-.5mm]
\draw[-,thin,darkblue] (-.5,.8) to[out=-90,in=90] (-.5,.6);
\draw[-,thin,darkblue] (-.5,.6) to[out=-90,in=180] (-.3,.2);
\draw[-,thin,darkblue] (.3,.6) to[out=0,in=90] (.5,.4);
\draw[-,thin,darkblue] (.5,.4) to[out=-90,in=90] (.5,-.4);
\draw[-,thin,darkblue] (.5,-.4) to[out=-90,in=0] (.3,-.6);
\draw[-,thin,darkblue] (-.3,-.2) to[out=180,in=90] (-.5,-.4);
\draw[->,thin,darkblue] (-.5,-.4) to[out=-90,in=90] (-.5,-.8);
\draw[-,thin,darkblue] (.3,-.6) to[out=180,in=0] (-.3,-.2);
\draw[->,thin,darkblue] (.2,0) to[out=90,in=-90] (-.2,0.8);
\draw[-,line width=4pt,white] (-.2,-.8) to[out=90,in=-90] (.2,0);
\draw[-,thin,darkblue] (-.2,-.8) to[out=90,in=-90] (.2,0);
\draw[-,line width=4pt,white] (-.3,.2) to[out=0,in=180] (.3,.6);
\draw[-,thin,darkblue] (-.3,.2) to[out=0,in=180] (.3,.6);
\end{tikzpicture}
}
&=
\mathord{
\begin{tikzpicture}[baseline = 0]
	\draw[->,thin,darkblue] (0.08,-.6) to (0.08,.6);
	\draw[<-,thin,darkblue] (-0.28,-.6) to (-0.28,.6);
\end{tikzpicture}
}
+tz
\mathord{
\begin{tikzpicture}[baseline=-.5mm]
	\draw[<-,thin,darkblue] (0.3,0.6) to[out=-90, in=0] (0,.1);
	\draw[-,thin,darkblue] (0,.1) to[out = 180, in = -90] (-0.3,0.6);
	\draw[-,thin,darkblue] (0.3,-.6) to[out=90, in=0] (0,-0.1);
	\draw[->,thin,darkblue] (0,-0.1) to[out = 180, in = 90] (-0.3,-.6);
\end{tikzpicture}}\:,
&\mathord{
\begin{tikzpicture}[baseline = -0.5mm]
	\draw[<-,thin,darkblue] (0,0.6) to (0,0.3);
	\draw[-,thin,darkblue] (-0.3,-0.2) to [out=180,in=-90](-.5,0);
	\draw[-,thin,darkblue] (-0.5,0) to [out=90,in=180](-.3,0.2);
	\draw[-,thin,darkblue] (-0.3,.2) to [out=0,in=90](0,-0.3);
	\draw[-,thin,darkblue] (0,-0.3) to (0,-0.6);
	\draw[-,line width=4pt,white] (0,0.3) to [out=-90,in=0] (-.3,-0.2);
	\draw[-,thin,darkblue] (0,0.3) to [out=-90,in=0] (-.3,-0.2);
\end{tikzpicture}
}&=0,&
\mathord{
\begin{tikzpicture}[baseline = 1.25mm]
  \draw[->,thin,darkblue] (0.2,0.2) to[out=90,in=0] (0,.4);
  \draw[-,thin,darkblue] (0,0.4) to[out=180,in=90] (-.2,0.2);
\draw[-,thin,darkblue] (-.2,0.2) to[out=-90,in=180] (0,0);
  \draw[-,thin,darkblue] (0,0) to[out=0,in=-90] (0.2,0.2);
\end{tikzpicture}
}&= 
-t^{-1}z^{-1}1_\unit,
\end{align*}
and one more relation, which is equivalent to \eqref{dagger}.
We have {\em not} included the generating morphism $x$
since,
due to a special feature of the $k=-1$ case,
it can be recovered
from the other generators via the formula
\[
    x = 
    \begin{tikzpicture}[baseline = -.5mm]
	    \draw[->,thin,darkblue] (0.08,-.3) to (0.08,.4);
        \node at (0.08,0.05) {$\dot$};
    \end{tikzpicture}
    :=
    t\
    \begin{tikzpicture}[baseline = -0.5mm]
	    \draw[<-,thin,darkblue] (0,0.4) to (0,0.2);
    	\draw[-,thin,darkblue] (0,0.24) to [out=-90,in=180] (.2,-0.11);
    	\draw[-,thin,darkblue] (0.2,-0.11) to [out=0,in=-90](.35,0.04);
    	\draw[-,thin,darkblue] (0.35,0.04) to [out=90,in=0](.2,0.19);
    	\draw[-,thin,darkblue] (0,-0.16) to (0,-0.26);
    	\draw[-,line width=4pt,white] (0.2,.19) to [out=180,in=90](0,-0.16);
    	\draw[-,thin,darkblue] (0.2,.19) to [out=180,in=90](0,-0.16);
    \end{tikzpicture}
    -t^2\
    \begin{tikzpicture}[baseline = -0.5mm]
    	\draw[<-,thin,darkblue] (0,0.4) to (0,-0.3);
    \end{tikzpicture}
    \ .
\]
The relations in Definition~\ref{def3} which involve $x$ such as (\ref{rr3}) are consequences of the other
relations with one exception:
we must still impose 
that $x$ is
invertible, that is, relation \eqref{dagger}.

The deformed Heisenberg category ${\mathcal H}(q^2)$ introduced in \cite{LS} is
(the additive envelope of) the strict $\k$-linear monoidal category defined by
the same presentation as in the previous paragraph, with the parameters satisfying $tz=-1$,
but {\em without} the relation \eqref{dagger}. 
This follows easily on comparing our presentation with the one
in {\em loc.\ cit.}, using also the fact that our category is strictly
pivotal.
The generator $x$ denoted by a dot here is not the same as the morphism denoted by a dot in \cite{LS} (that is simply equal to the right curl); instead, our dot is the ``star dot" of \cite{CLLSS} (up to renormalization). The Hecke algebra generator $T=\mathord{
\begin{tikzpicture}[baseline = -.5mm]
	\draw[->,thin,darkblue] (0.2,-.2) to (-0.2,.3);
	\draw[thin,darkblue,->] (-0.2,-.2) to (0.2,.3);
\end{tikzpicture}
}$ from \cite[Definition 2.1]{LS} is related
to our $\tau$ by $T = q \tau$ (so that the quadratic relation becomes
$(T_j-q^2)(T_j+1) = 0$).
Also the generator $X$ appearing just before \cite[Lemma 3.8]{LS} is
our $-x$.
In fact, 
the category ${\mathcal H}(q^2)$
may be identified with the
monoidal subcategory of our category 
$\HEIS_{-1}(z,-z^{-1})$ consisting of all
objects and all
morphisms which do not 
involve negative powers of $x$.

For any $\k$-linear category $\mathcal C$, there is an associated strict $\k$-linear monoidal category $\mathcal{E}nd_\k(\mathcal C)$ 
consisting of $\k$-linear endofunctors and natural transformations.
Then one can consider ``representations'' of $\HEIS_k(z,t)$ by considering $\k$-linear monoidal functors into  $\mathcal{E}nd_\k(\mathcal C)$ for different choices of $\mathcal C$.  The motivation for the definition of $\HEIS_k(z,t)$ comes from the fact that
it acts in this way on other well-known categories appearing in
representation theory.
If $k = 0$ and $t=q^n$ then $\HEIS_k(z,t)$ acts on
representations of  $\Uq(\mathfrak{gl}_n)$, with the generating objects
$\up$ and $\down$ acting by tensoring
with the natural $\Uq(\mathfrak{gl}_n)$-module and its dual,
respectively;
see section~\ref{qgln}.
This action 
is an extension of the monoidal functor from 
the HOMFLY-PT skein category to the category of finite-dimensional
$\Uq(\mathfrak{gl}_n)$-modules constructed originally by Turaev \cite{Turaev1}.
If $k \neq 0$ then $\HEIS_k(z,t)$ acts on representations of the
cyclotomic Hecke algebras of level $|k|$ from \cite{AK}, with 
$\up$ and $\down$ acting by induction and restriction functors if $k <
0$, or vice versa if $k  > 0$;
see section~\ref{qcyclo}. 
When $k=-1$, this specializes to
the action of the deformed Heisenberg category
on modules over the usual (finite) Hecke algebras associated to
the symmetric groups
constructed already in \cite{LS}.
The action of $\HEIS_{-l}(z,t)$ on representations of cyclotomic Hecke algebras extends to an action on category $\mathcal O$ over the rational Cherednik algebras of type $\SG_n\wr \Z/l$ for all $n \geq 0$, with $\up$ and $\down$ acting by certain Bezrukavnikov-Etingof induction and restriction functors from \cite{BE}; see section~\ref{scherednik}.

We also prove a
basis theorem for the morphism spaces in
$\HEIS_k(z,t)$;
see section \ref{sbasis} for the precise statement.
In particular, our basis theorem implies that the {\em center} $\End_{\HEIS_k(z,t)}(\unit)$
of the quantum Heisenberg category
is the tensor product $\Sym \otimes \Sym$  
of {\em two} copies of the algebra of symmetric functions.
In the degenerate case studied in \cite{Bheis}, the basis theorem was proved by
treating the cases $k=0$ and $k \neq 0$ separately, appealing to results
from \cite{BCNR} and \cite{MS}; the proofs in {\em
  loc.\ cit.} ultimately exploited
analogs of the categorical actions mentioned above, on representations of degenerate cyclotomic Hecke algebras and representations of $\mathfrak{gl}_n(\mathbb{C})$, respectively.
In the quantum case, it is still possible
to prove the basis theorem when $k=0$ by such an argument, but for
non-zero $k$ the approach from \cite{MS} seems to be unmanageable due to
the larger center.
Instead, we prove the basis theorem here
by following
the technique developed in the
degenerate case in \cite[Theorem 6.4]{BSW1} (and earlier, 
in the context
of Kac-Moody 2-categories, in
\cite{Wunfurling}). It depends crucially on the
existence of an action of $\HEIS_k(z,t)$ on a ``sufficiently large''
module category, which is obtained
by choosing $l \gg 0$ then
taking the tensor product of actions of
$\HEIS_{-l}(z,t)$ and $\HEIS_{k+l}(z,1)$
on representations of suitably generic
cyclotomic Hecke algebras of levels $l$ and $k+l$, respectively.

The construction of this categorical tensor product involves a
remarkable monoidal functor from
$\HEIS_{k}(z,t)$ to a certain localization of the symmetric product 
\[
    \HEIS_{l}(z,u) \odot \HEIS_{m}(z,v)
\]
for $k=l+m$ and $t=uv$.
This functor is defined in section~\ref{scc} and is the quantum analog of the
categorical comultiplication from
\cite[Theorem 5.4]{BSW1}.
The particular tensor products exploited to prove the basis theorem are generic examples of {\em generalized cyclotomic quotients} of $\HEIS_k(z,t)$; see section~\ref{sgcq} for the general definition.
In fact, these $\k$-linear categories first appeared in 
 \cite[Proposition 5.6]{Wcanonical}, but in a rather different form;
the precise relationship between the categories of
{\em loc.\ cit.} and the ones here will be explained in  
\cite{BSW2}.

We have stopped short of proving any results about the {\em
  decategorification} of $\HEIS_k(z,t)$ here, but let us make some
remarks about this. There are two complementary points of view:
\begin{itemize}
\item
One can consider
the {\em Grothendieck ring} 
$K_0(\Kar(\HEIS_k(z,t)))$
of the additive Karoubi
envelope of $\HEIS_k(z,t)$.
For generic $z$ (i.e., when $q$ is not a root of unity), we expect that this is isomorphic to a
$\Z$-form
for a central reduction of the universal
enveloping algebra of the infinite-dimensional Heisenberg Lie
algebra, just as was established in the degenerate case in \cite[Theorem
1.1]{BSW1}.
However, there is a significant obstruction to proving this result in
the quantum case: we do not
know how to show that the split Grothendieck group $K_0(AH_n)$ of the
affine Hecke algebra is isomorphic to that of the finite Hecke
algebra.
\item
Alternatively, one can pass to the {\em trace} (or zeroth Hochschild homology).
In \cite{CLLSS}, this was computed already for the 
category ${\mathcal H}(q^2)$ of \cite{LS}, revealing an interesting connection to
the elliptic Hall algebra.
Using the basis theorem proved here, we expect it should be possible
to extend the calculations made in {\em loc. cit.} to give a description of the trace of the full category
$\HEIS_k(z,t)$ for all $k \in \Z$.
\end{itemize}

In the main body of the article,
proofs of all lemmas involving purely
diagrammatic manipulations 
have been omitted. However, we have attempted to give enough details
for the reader familiar with the analogous calculations
in the degenerate case from \cite[$\S$2]{Bheis} and \cite[$\S$5]{BSW1}
to be able to reconstruct the proofs. The authors are
currently preparing a sequel \cite{BSW3} in which we incorporate a (symmetric)
Frobenius algebra into the definition of $\HEIS_k(z,t)$, in a similar way
to the Frobenius Heisenberg categories defined in the degenerate case
in \cite{Savage}. We will include full proofs of all of the diagrammatic lemmas
in the more general Frobenius setting in this sequel.

\subsection*{Corrections to published version}  This version of the paper contains corrections of some errors present in the published version:
\begin{itemize}
    \item The second summation in Lemma~\ref{l2} was corrected.
    \item Some instances of ``left-hand" were changed to ``right-hand" in the proof of Theorem~\ref{comult}.
    \item Above \eqref{ind}, $\HEIS_{-l}(z,f_0^{-1})$ was changed to $\HEIS_{-l}(z,t)$.
    \item The phrase ``viewed as a module" was changed to ``viewed as module" above equation \ref{piglet}.
    \item Many occurrences of $\unit$ were changed to $1_\unit$ in the proofs of Lemma~\ref{clever} and Theorem~\ref{chemistry}.
    \item Once instance of the symbol $\otimes$ was changed to $\circ$ in the statement of Lemma~\ref{breakfast}.
\end{itemize}

\section{First approach}\label{first}

Before formulating our first definition of $\HEIS_k(z,t)$, let us make
some general remarks.
We refer to the relation (\ref{rr2}) as the {\em upward skein
  relation}.
Rotating it through $\pm 90^\circ$ or 180$^\circ$, one obtains three more
skein relations; for example, here is the {\em leftward skein
  relation}
\begin{equation}\label{leftwardsskein}
\mathord{

}\:.
\end{equation}
At present, this has no meaning since we have not defined the leftward
cups, caps or crossings which it involves!
However, already in the monoidal category obtained from $\AH(z)$ by
adjoining a right dual $\down$ to $\up$ as explained in the
introduction, we can introduce the rightward crossings:
\begin{align}\label{rotate}
\mathord{
%
}\:,
\end{align}
and then we see that
the rightward skein relation holds
from (\ref{rr2}). 
Rotating the two rightward crossings 
once more by a similar procedure, we obtain positive and negative
downward crossings satisfying the downward skein relation.
We also define the downward dot:
\begin{align}\label{downwarddot}
y = \mathord{
%
}\:.
\end{align}
It is immediate from these definitions and (\ref{rightadj}) that dots
and crossings slide past rightward cups and caps:
\begin{align}\label{cough1}
&&\mathord{
%
}\:.
\end{align}
Also, the following relations are easily deduced by attaching
rightward cups and caps to the relations in (\ref{rr0}), then rotating the pictures using the
definitions of the rightward/downward crossings:
\begin{align}\label{spit1}
\mathord{
%
}.
\end{align}
The following lemma will be used repeatedly (often without reference).
There are 
analogous dot slide relations for the rightward and
downward crossings (obtained by rotation).

\begin{lemma}\label{dotslide}
The following relations hold for $a \in \Z$:
\begin{align}\label{teaminus}
\mathord{

}\!
&= \left\{
\begin{array}{ll}
\!\mathord{
\begin{tikzpicture}[baseline = -.5mm]
	\draw[thin,darkblue,->] (-0.28,-.3) to (0.28,.4);
	\draw[-,line width=4pt,white] (0.28,-.3) to (-0.28,.4);
	\draw[->,thin,darkblue] (0.28,-.3) to (-0.28,.4);
      \node at (-0.14,0.23) {$\dot$};
      \node at (-.28,0.1) {$\color{darkblue}\scriptstyle a$};
\end{tikzpicture}
}- z \displaystyle\sum_{\substack{b+c=a\\b,c > 0}}
\mathord{
\begin{tikzpicture}[baseline = -.5mm]
	\draw[thin,darkblue,->] (-0.2,-.3) to (-0.2,.4);
	\draw[->,thin,darkblue] (0.2,-.3) to (0.2,.4);
      \node at (-0.34,0.05) {$\color{darkblue}\scriptstyle b$};
      \node at (-0.2,0.05) {$\dot$};
      \node at (0.2,0.05) {$\dot$};
      \node at (0.34,0.05) {$\color{darkblue}\scriptstyle c$};
\end{tikzpicture}
}
&\text{if $a > 0$,}\\
\!\mathord{
\begin{tikzpicture}[baseline = -.5mm]
	\draw[thin,darkblue,->] (-0.28,-.3) to (0.28,.4);
	\draw[-,line width=4pt,white] (0.28,-.3) to (-0.28,.4);
	\draw[->,thin,darkblue] (0.28,-.3) to (-0.28,.4);
      \node at (-0.14,0.23) {$\dot$};
      \node at (-.28,0.1) {$\color{darkblue}\scriptstyle a$};
\end{tikzpicture}
}
+ z \displaystyle\sum_{\substack{b+c=a\\b,c \leq 0}}
\mathord{
\begin{tikzpicture}[baseline = -.5mm]
	\draw[thin,darkblue,->] (-0.2,-.3) to (-0.2,.4);
	\draw[->,thin,darkblue] (0.2,-.3) to (0.2,.4);
      \node at (-0.34,0.05) {$\color{darkblue}\scriptstyle b$};
      \node at (-0.2,0.05) {$\dot$};
      \node at (0.2,0.05) {$\dot$};
      \node at (0.34,0.05) {$\color{darkblue}\scriptstyle c$};
\end{tikzpicture}
}
&\text{if $a \leq 0$;}
\end{array}%
\right.
\end{align}
\end{lemma}

Now we can explain the first way to complete the definition of the quantum 
Heisenberg category following the scheme outlined in the introduction.
The idea is to invert the morphism
\begin{align}
\label{invrel}
\left\{%
}
\right]
:\up \otimes \down \oplus 
\unit^{\oplus (-k)}
\rightarrow
 \down \otimes  \up&\text{if $k < 0$,}
\end{array}
\right.\end{align}
in $\Add(\HEIS_k(z,t))$ (where $\Add$ denotes the additive envelope).

\begin{definition}\label{def1}
The {\em quantum Heisenberg category} $\HEIS_k(z,t)$ is the strict
$\k$-linear monoidal category obtained from $\AH(z)$ by 
adjoining a right dual $\down$ to $\up$ as 
explained in the introduction,
together with the matrix entries of the following morphism
which we declare to be a two-sided inverse to 
the morphism (\ref{invrel}):
\begin{align}
\label{invrel2}
\left\{\begin{array}{rl}
\left[\:
\mathord{
\begin{tikzpicture}[baseline = 0]
	\draw[->,thin,darkblue] (0.28,-.3) to (-0.28,.4);
	\draw[-,line width=4pt,white] (-0.28,-.3) to (0.28,.4);
	\draw[<-,thin,darkblue] (-0.28,-.3) to (0.28,.4);
\end{tikzpicture}
}
\:\:\:
\mathord{
\begin{tikzpicture}[baseline = 1mm]
	\draw[-,thin,darkblue] (0.4,0.4) to[out=-90, in=0] (0.1,0);
	\draw[->,thin,darkblue] (0.1,0) to[out = 180, in = -90] (-0.2,0.4);
      \node at (0.12,-0.2) {$\color{darkblue}\scriptstyle{0}$};
      \node at (0.12,0.01) {$\diamond$};
\end{tikzpicture}
}
\:\:\:\cdots\:\:\:
\mathord{
\begin{tikzpicture}[baseline = 1mm]
	\draw[-,thin,darkblue] (0.4,0.4) to[out=-90, in=0] (0.1,0);
	\draw[->,thin,darkblue] (0.1,0) to[out = 180, in = -90] (-0.2,0.4);
      \node at (0.12,-0.2) {$\color{darkblue}\scriptstyle{k-1}$};
      \node at (0.12,0.01) {$\diamond$};
\end{tikzpicture}
}
\right]
:\down\otimes\up \oplus
\unit^{\oplus k}
\rightarrow
 \up \otimes  \down \hspace{4.5mm}
&\text{if $k \geq 0$},\\\\
\left[\!\!\!\!\!
\begin{array}{r}
\mathord{
\:\begin{tikzpicture}[baseline = 1mm]
	\draw[<-,thin,darkblue] (-0.28,-.3) to (0.28,.4);
	\draw[-,line width=4pt,white] (0.28,-.3) to (-0.28,.4);
	\draw[->,thin,darkblue] (0.28,-.3) to (-0.28,.4);
\end{tikzpicture}
}\\
\mathord{
\begin{tikzpicture}[baseline = 1mm]
	\draw[-,thin,darkblue] (0.4,0) to[out=90, in=0] (0.1,0.4);
	\draw[->,thin,darkblue] (0.1,0.4) to[out = 180, in = 90] (-0.2,0);
      \node at (0.12,0.6) {$\color{darkblue}\scriptstyle{0}$};
      \node at (0.12,0.37) {$\heart$};
\end{tikzpicture}
}
\\\vdots\:\:\:\\
\mathord{
\begin{tikzpicture}[baseline = 1mm]
	\draw[-,thin,darkblue] (0.4,0) to[out=90, in=0] (0.1,0.4);
	\draw[->,thin,darkblue] (0.1,0.4) to[out = 180, in = 90] (-0.2,0);
      \node at (0.12,0.6) {$\color{darkblue}\scriptstyle{-k-1}$};
      \node at (0.12,0.37) {$\heart$};
\end{tikzpicture}
}\!\!\!
\end{array}
\right]
:
\down \otimes \up \rightarrow
\up \otimes \down \oplus \unit^{\oplus(-k)}
&\text{if $k <
  0$}.
\end{array}
\right.\end{align}
We impose one more essential relation:
\begin{align}\label{impose}
\mathord{
\begin{tikzpicture}[baseline = .8mm]
  \draw[<-,thin,darkblue] (0.2,0.2) to[out=90,in=0] (0,.4);
  \draw[-,thin,darkblue] (0,0.4) to[out=180,in=90] (-.2,0.2);
\draw[-,thin,darkblue] (-.2,0.2) to[out=-90,in=180] (0,0);
  \draw[-,thin,darkblue] (0,0) to[out=0,in=-90] (0.2,0.2);
\end{tikzpicture}
}&= t z^{-1}1_\unit\:\text{if $k > 0$,}&
\clockright
&= (t z^{-1}-t^{-1}z^{-1}) 1_\unit\:\text{if $k = 0$,}
&\mathord{\begin{tikzpicture}[baseline = .8mm]
  \draw[-,thin,darkblue] (0.2,0.2) to[out=90,in=0] (0,.4);
  \draw[->,thin,darkblue] (0,0.4) to[out=180,in=90] (-.2,0.2);
\draw[-,thin,darkblue] (-.2,0.2) to[out=-90,in=180] (0,0);
  \draw[-,thin,darkblue] (0,0) to[out=0,in=-90] (0.2,0.2);
      \node at (0.2,0.2) {$\dot$};
      \node at (0.49,0.2) {$\color{darkblue}\scriptstyle{-k}$};
 \end{tikzpicture}
}
&= tz^{-1} 1_\unit\:\text{if $k < 0$,}
\end{align}
where the leftward cups and caps are defined by the formulas:
\begin{align}\label{leftwards}
\mathord{
\begin{tikzpicture}[baseline = 1mm]
	\draw[-,thin,darkblue] (0.4,0.4) to[out=-90, in=0] (0.1,0);
	\draw[->,thin,darkblue] (0.1,0) to[out = 180, in = -90] (-0.2,0.4);
 \end{tikzpicture}
}
&:=
\left\{
\begin{array}{cl}
-\displaystyle{t^{-1} z^{-1}}\:\mathord{
\begin{tikzpicture}[baseline = 0.5mm]
	\draw[-,thin,darkblue] (0.4,0.4) to[out=-90, in=0] (0.1,0);
	\draw[->,thin,darkblue] (0.1,0) to[out = 180, in = -90] (-0.2,0.4);
     \node at (-0.45,0.15) {$\color{darkblue}\scriptstyle{-1}$};
      \node at (-0.16,0.15) {$\dot$};
      \node at (0.16,-0.21) {$\color{darkblue}\scriptstyle{k-1}$};
      \node at (0.12,0.01) {$\diamond$};
\end{tikzpicture}
}
\hspace{3.5mm}&\!\!\!\text{if $k > 0$,}\\
t\mathord{
\begin{tikzpicture}[baseline = 0]
	\draw[<-,thin,darkblue] (-0.25,.6) to[out=300,in=90] (0.25,-0);
	\draw[-,thin,darkblue] (0.25,-0) to[out=-90, in=0] (0,-0.25);
	\draw[-,thin,darkblue] (0,-0.25) to[out = 180, in = -90] (-0.25,-0);
	\draw[-,line width=4pt,white] (0.25,.6) to[out=240,in=90] (-0.25,-0);
	\draw[-,thin,darkblue] (0.25,.6) to[out=240,in=90] (-0.25,-0);
\end{tikzpicture}
}
&\!\!\!\text{if $k = 0$,}\\
t^{-1}
\mathord{
\begin{tikzpicture}[baseline = 0]
	\draw[-,thin,darkblue] (0.25,.6) to[out=240,in=90] (-0.25,-0);
	\draw[-,thin,darkblue] (0.25,-0.) to[out=-90, in=0] (0,-0.25);
	\draw[-,thin,darkblue] (0,-0.25) to[out = 180, in = -90] (-0.25,-0);
	\draw[-,line width=4pt,white] (-0.25,.6) to[out=300,in=90] (0.25,-0);
	\draw[<-,thin,darkblue] (-0.25,.6) to[out=300,in=90] (0.25,-0);
      \node at (0.54,-0) {$\color{darkblue}\scriptstyle{-k}$};
      \node at (0.24,-0) {$\dot$};
\end{tikzpicture}
}
&\!\!\!\text{if $k <0$;}
\end{array}\right.
&\mathord{
\begin{tikzpicture}[baseline = 1mm]
	\draw[-,thin,darkblue] (0.4,0) to[out=90, in=0] (0.1,0.4);
	\draw[->,thin,darkblue] (0.1,0.4) to[out = 180, in = 90] (-0.2,0);
\end{tikzpicture}
}
&:=
\left\{
\begin{array}{cl}
t\mathord{
\begin{tikzpicture}[baseline = -1.5mm]
	\draw[-,thin,darkblue] (0.25,-.5) to[out=120,in=-90] (-0.25,0.1);
	\draw[-,line width=4pt,white] (-0.25,-.5) to[out=60,in=-90] (0.25,0.1);
	\draw[<-,thin,darkblue] (-0.25,-.5) to[out=60,in=-90] (0.25,0.1);
	\draw[-,thin,darkblue] (0.25,0.1) to[out=90, in=0] (0,0.35);
	\draw[-,thin,darkblue] (0,0.35) to[out = 180, in = 90] (-0.25,0.1);
      \node at (-0.42,0.05) {$\color{darkblue}\scriptstyle{k}$};
      \node at (-0.25,0.05) {$\dot$};
\end{tikzpicture}
}
&\text{if $k  \geq 0$,}\\
-\displaystyle{t^{-1}z^{-1}}\mathord{
\begin{tikzpicture}[baseline = 1mm]
	\draw[-,thin,darkblue] (0.4,0) to[out=90, in=0] (0.1,0.4);
	\draw[->,thin,darkblue] (0.1,0.4) to[out = 180, in = 90] (-0.2,0);
      \node at (0.12,0.6) {$\color{darkblue}\scriptstyle{0}$};
     \node at (0.12,0.37) {$\heart$};
\end{tikzpicture}}
&\text{if $k < 0$.}
\end{array}\right.
\end{align}
To complete the definition, we introduce a few more shorthands for
morphisms. We have already introduced one of the two leftward crossings;
define the other one so that the leftward skein relation (\ref{leftwardsskein})
holds.
Also set
\begin{align}
\mathord{
\begin{tikzpicture}[baseline = 1mm]
	\draw[-,thin,darkblue] (0.4,0.4) to[out=-90, in=0] (0.1,0);
	\draw[->,thin,darkblue] (0.1,0) to[out = 180, in = -90] (-0.2,0.4);
      \node at (0.12,-0.2) {$\color{darkblue}\scriptstyle{0}$};
      \node at (0.12,0.01) {$\heart$};
\end{tikzpicture}
}
&:=\mathord{
\begin{tikzpicture}[baseline = 1mm]
	\draw[-,thin,darkblue] (0.4,0.4) to[out=-90, in=0] (0.1,0);
	\draw[->,thin,darkblue] (0.1,0) to[out = 180, in = -90] (-0.2,0.4);
      \node at (0.12,-0.2) {$\color{darkblue}\scriptstyle{0}$};
      \node at (0.12,0.01) {$\diamond$};
\end{tikzpicture}
}+z
\mathord{
\begin{tikzpicture}[baseline = 0]
	\draw[<-,thin,darkblue] (-0.25,.6) to[out=300,in=90] (0.25,-0);
	\draw[-,thin,darkblue] (0.25,-0) to[out=-90, in=0] (0,-0.25);
	\draw[-,thin,darkblue] (0,-0.25) to[out = 180, in = -90] (-0.25,-0);
	\draw[-,line width=4pt,white] (0.25,.6) to[out=240,in=90] (-0.25,-0);
	\draw[-,thin,darkblue] (0.25,.6) to[out=240,in=90] (-0.25,-0);
\end{tikzpicture}
}\quad\text{if $k > 0$},
&\mathord{
\begin{tikzpicture}[baseline = 1mm]
	\draw[-,thin,darkblue] (0.4,0.4) to[out=-90, in=0] (0.1,0);
	\draw[->,thin,darkblue] (0.1,0) to[out = 180, in = -90] (-0.2,0.4);
      \node at (0.12,-0.2) {$\color{darkblue}\scriptstyle{a}$};
      \node at (0.12,0.01) {$\heart$};
\end{tikzpicture}
}
&:=\mathord{
\begin{tikzpicture}[baseline = 1mm]
	\draw[-,thin,darkblue] (0.4,0.4) to[out=-90, in=0] (0.1,0);
	\draw[->,thin,darkblue] (0.1,0) to[out = 180, in = -90] (-0.2,0.4);
      \node at (0.12,-0.2) {$\color{darkblue}\scriptstyle{a}$};
      \node at (0.12,0.01) {$\diamond$};
\end{tikzpicture}
}
\quad\text{if $0 < a < k$,}\label{nakano1}\\\label{nakano3}
\mathord{
\begin{tikzpicture}[baseline = 1mm]
	\draw[-,thin,darkblue] (0.4,0) to[out=90, in=0] (0.1,0.4);
	\draw[->,thin,darkblue] (0.1,0.4) to[out = 180, in = 90] (-0.2,0);
      \node at (0.12,0.6) {$\color{darkblue}\scriptstyle{0}$};
      \node at (0.12,0.4) {$\diamond$};
\end{tikzpicture}
}&:=
\mathord{
\begin{tikzpicture}[baseline = 1mm]
	\draw[-,thin,darkblue] (0.4,0) to[out=90, in=0] (0.1,0.4);
	\draw[->,thin,darkblue] (0.1,0.4) to[out = 180, in = 90] (-0.2,0);
      \node at (0.12,0.6) {$\color{darkblue}\scriptstyle{0}$};
      \node at (0.12,0.4) {$\heart$};
\end{tikzpicture}
}+z
\mathord{
\begin{tikzpicture}[baseline = -1.5mm]
	\draw[<-,thin,darkblue] (-0.25,-.5) to[out=60,in=-90] (0.25,0.1);
	\draw[-,line width=4pt,white] (0.25,-.5) to[out=120,in=-90] (-0.25,0.1);
	\draw[-,thin,darkblue] (0.25,-.5) to[out=120,in=-90] (-0.25,0.1);
	\draw[-,thin,darkblue] (0.25,0.1) to[out=90, in=0] (0,0.35);
	\draw[-,thin,darkblue] (0,0.35) to[out = 180, in = 90] (-0.25,0.1);
\end{tikzpicture}
}\quad\text{if $k < 0$,}
&
\mathord{
\begin{tikzpicture}[baseline = 1mm]
	\draw[-,thin,darkblue] (0.4,0) to[out=90, in=0] (0.1,0.4);
	\draw[->,thin,darkblue] (0.1,0.4) to[out = 180, in = 90] (-0.2,0);
      \node at (0.12,0.6) {$\color{darkblue}\scriptstyle{a}$};
      \node at (0.12,0.4) {$\diamond$};
\end{tikzpicture}
}&:=
\mathord{
\begin{tikzpicture}[baseline = 1mm]
	\draw[-,thin,darkblue] (0.4,0) to[out=90, in=0] (0.1,0.4);
	\draw[->,thin,darkblue] (0.1,0.4) to[out = 180, in = 90] (-0.2,0);
      \node at (0.12,0.6) {$\color{darkblue}\scriptstyle{a}$};
      \node at (0.12,0.4) {$\heart$};
\end{tikzpicture}
}\quad\text{if $0 < a < -k$.}
\end{align}
Next, introduce the following {\em $(+)$-bubbles}
assuming $a \leq 0$:
\begin{align}
\mathord{\begin{tikzpicture}[baseline = .8mm]
  \draw[-,thin,darkblue] (0.2,0.2) to[out=90,in=0] (0,.4);
  \draw[->,thin,darkblue] (0,0.4) to[out=180,in=90] (-.2,0.2);
\draw[-,thin,darkblue] (-.2,0.2) to[out=-90,in=180] (0,0);
  \draw[-,thin,darkblue] (0,0) to[out=0,in=-90] (0.2,0.2);
      \node at (0,0.2) {$\color{darkblue}{+}$};
      \node at (0.3,0.2) {$\color{darkblue}\scriptstyle{a}$};
 \end{tikzpicture}
}
&:=
\left\{
\begin{array}{ll}
-\displaystyle{tz^{-1}}\mathord{\begin{tikzpicture}[baseline = -1mm]
  \draw[<-,thin,darkblue] (0.2,0.2) to[out=90,in=0] (0,.4);
  \draw[-,thin,darkblue] (0,0.4) to[out=180,in=90] (-.2,0.2);
\draw[-,thin,darkblue] (-.2,0.2) to[out=-90,in=180] (0,0);
  \draw[-,thin,darkblue] (0,0) to[out=0,in=-90] (0.2,0.2);
     \node at (-0.2,0.2) {$\dot$};
      \node at (-0.4,0.2) {$\color{darkblue}\scriptstyle{k}$};
      \node at (0,-0.17) {$\color{darkblue}\scriptstyle{-a}$};
      \node at (0,0.01) {$\diamond$};
 \end{tikzpicture}
}
&\text{if $a > -k$,}\\
\displaystyle{tz^{-1}}1_\unit&\text{if $a=-k$,}\\
0&\text{if $a < -k$;}
\end{array}\right.
&
\mathord{\begin{tikzpicture}[baseline = .8mm]
  \draw[<-,thin,darkblue] (0.2,0.2) to[out=90,in=0] (0,.4);
  \draw[-,thin,darkblue] (0,0.4) to[out=180,in=90] (-.2,0.2);
\draw[-,thin,darkblue] (-.2,0.2) to[out=-90,in=180] (0,0);
  \draw[-,thin,darkblue] (0,0) to[out=0,in=-90] (0.2,0.2);
      \node at (0,0.2) {$\color{darkblue}{+}$};
      \node at (-0.3,0.2) {$\color{darkblue}\scriptstyle{a}$};
 \end{tikzpicture}
}
&:=
\left\{
\begin{array}{ll}
\displaystyle{t^{-1}z^{-1}}\mathord{\begin{tikzpicture}[baseline = 1.5mm]
  \draw[-,thin,darkblue] (0.2,0.2) to[out=90,in=0] (0,.4);
  \draw[->,thin,darkblue] (0,0.4) to[out=180,in=90] (-.2,0.2);
\draw[-,thin,darkblue] (-.2,0.2) to[out=-90,in=180] (0,0);
  \draw[-,thin,darkblue] (0,0) to[out=0,in=-90] (0.2,0.2);
     \node at (0.2,0.2) {$\dot$};
      \node at (0.43,0.2) {$\color{darkblue}\scriptstyle{-k}$};
      \node at (-.02,0.55) {$\color{darkblue}\scriptstyle{-a}$};
      \node at (0,0.4) {$\diamond$};
      \node at (0,-.02) {$\color{darkblue}\scriptstyle{\phantom.}$};
 \end{tikzpicture}
}
&\text{if $a > k$,}\\
- t^{-1}z^{-1} 1_\unit&\text{if $a=k$,}\\
0&\text{if $a < k$}.
\end{array}\right.\label{fake0}
\end{align}
Finally, define the $(+)$-bubbles with label $a > 0$ to be the usual bubbles
with $a$ dots:
\begin{align}
\mathord{\begin{tikzpicture}[baseline = -1mm]
  \draw[-,thin,darkblue] (0,0.2) to[out=180,in=90] (-.2,0);
  \draw[->,thin,darkblue] (0.2,0) to[out=90,in=0] (0,.2);
 \draw[-,thin,darkblue] (-.2,0) to[out=-90,in=180] (0,-0.2);
  \draw[-,thin,darkblue] (0,-0.2) to[out=0,in=-90] (0.2,0);
\node at (0,0) {$\color{darkblue}+$};
      \node at (0.3,0) {$\color{darkblue}\scriptstyle a$};
\end{tikzpicture}
}&:=
\mathord{\begin{tikzpicture}[baseline = -1mm]
  \draw[-,thin,darkblue] (0,0.2) to[out=180,in=90] (-.2,0);
  \draw[->,thin,darkblue] (0.2,0) to[out=90,in=0] (0,.2);
 \draw[-,thin,darkblue] (-.2,0) to[out=-90,in=180] (0,-0.2);
  \draw[-,thin,darkblue] (0,-0.2) to[out=0,in=-90] (0.2,0);
      \node at (0.2,0) {$\dot$};
      \node at (0.4,0) {$\color{darkblue}\scriptstyle a$};
\end{tikzpicture}
}\:,&
\mathord{\begin{tikzpicture}[baseline = -1mm]
  \draw[<-,thin,darkblue] (0,0.2) to[out=180,in=90] (-.2,0);
  \draw[-,thin,darkblue] (0.2,0) to[out=90,in=0] (0,.2);
 \draw[-,thin,darkblue] (-.2,0) to[out=-90,in=180] (0,-0.2);
  \draw[-,thin,darkblue] (0,-0.2) to[out=0,in=-90] (0.2,0);
\node at (0,0) {$\color{darkblue}+$};
      \node at (-0.3,0) {$\color{darkblue}\scriptstyle a$};
\end{tikzpicture}
}&:=
\mathord{\begin{tikzpicture}[baseline = -1mm]
  \draw[<-,thin,darkblue] (0,0.2) to[out=180,in=90] (-.2,0);
  \draw[-,thin,darkblue] (0.2,0) to[out=90,in=0] (0,.2);
 \draw[-,thin,darkblue] (-.2,0) to[out=-90,in=180] (0,-0.2);
  \draw[-,thin,darkblue] (0,-0.2) to[out=0,in=-90] (0.2,0);
      \node at (-0.2,0) {$\dot$};
      \node at (-0.4,0) {$\color{darkblue}\scriptstyle a$};
\end{tikzpicture}
}\:.
\label{fake2}\\\intertext{Then define {\em $(-)$-bubbles}
for all
$a \in \Z$ by setting}
\label{fake1}
\mathord{\begin{tikzpicture}[baseline = -1mm]
  \draw[-,thin,darkblue] (0,0.2) to[out=180,in=90] (-.2,0);
  \draw[->,thin,darkblue] (0.2,0) to[out=90,in=0] (0,.2);
 \draw[-,thin,darkblue] (-.2,0) to[out=-90,in=180] (0,-0.2);
  \draw[-,thin,darkblue] (0,-0.2) to[out=0,in=-90] (0.2,0);
\node at (0,0) {$\color{darkblue}-$};
      \node at (0.3,0) {$\color{darkblue}\scriptstyle a$};
\end{tikzpicture}
}&:=
\mathord{\begin{tikzpicture}[baseline = -1mm]
  \draw[-,thin,darkblue] (0,0.2) to[out=180,in=90] (-.2,0);
  \draw[->,thin,darkblue] (0.2,0) to[out=90,in=0] (0,.2);
 \draw[-,thin,darkblue] (-.2,0) to[out=-90,in=180] (0,-0.2);
  \draw[-,thin,darkblue] (0,-0.2) to[out=0,in=-90] (0.2,0);
      \node at (0.2,0) {$\dot$};
      \node at (0.4,0) {$\color{darkblue}\scriptstyle a$};
\end{tikzpicture}
}
-
\mathord{\begin{tikzpicture}[baseline = -1mm]
  \draw[-,thin,darkblue] (0,0.2) to[out=180,in=90] (-.2,0);
  \draw[->,thin,darkblue] (0.2,0) to[out=90,in=0] (0,.2);
 \draw[-,thin,darkblue] (-.2,0) to[out=-90,in=180] (0,-0.2);
  \draw[-,thin,darkblue] (0,-0.2) to[out=0,in=-90] (0.2,0);
\node at (0,0) {$\color{darkblue}+$};
      \node at (0.3,0) {$\color{darkblue}\scriptstyle a$};
\end{tikzpicture}
}\:,
&
\mathord{\begin{tikzpicture}[baseline = -1mm]
  \draw[<-,thin,darkblue] (0,0.2) to[out=180,in=90] (-.2,0);
  \draw[-,thin,darkblue] (0.2,0) to[out=90,in=0] (0,.2);
 \draw[-,thin,darkblue] (-.2,0) to[out=-90,in=180] (0,-0.2);
  \draw[-,thin,darkblue] (0,-0.2) to[out=0,in=-90] (0.2,0);
\node at (0,0) {$\color{darkblue}-$};
      \node at (-0.3,0) {$\color{darkblue}\scriptstyle a$};
\end{tikzpicture}
}&:=
\mathord{\begin{tikzpicture}[baseline = -1mm]
  \draw[<-,thin,darkblue] (0,0.2) to[out=180,in=90] (-.2,0);
  \draw[-,thin,darkblue] (0.2,0) to[out=90,in=0] (0,.2);
 \draw[-,thin,darkblue] (-.2,0) to[out=-90,in=180] (0,-0.2);
  \draw[-,thin,darkblue] (0,-0.2) to[out=0,in=-90] (0.2,0);
      \node at (-0.2,0) {$\dot$};
      \node at (-0.4,0) {$\color{darkblue}\scriptstyle a$};
\end{tikzpicture}
}
-
\mathord{\begin{tikzpicture}[baseline = -1mm]
  \draw[<-,thin,darkblue] (0,0.2) to[out=180,in=90] (-.2,0);
  \draw[-,thin,darkblue] (0.2,0) to[out=90,in=0] (0,.2);
 \draw[-,thin,darkblue] (-.2,0) to[out=-90,in=180] (0,-0.2);
  \draw[-,thin,darkblue] (0,-0.2) to[out=0,in=-90] (0.2,0);
\node at (0,0) {$\color{darkblue}+$};
      \node at (-0.3,0) {$\color{darkblue}\scriptstyle a$};
\end{tikzpicture}
}
\:.
\end{align}
\end{definition} 

\vspace{2mm}

In the case $k=0$, the assertion that (\ref{invrel}) and
(\ref{invrel2}) are two-sided inverses means that
\begin{align}\label{lunch}
\mathord{
\begin{tikzpicture}[baseline = 0]
	\draw[->,thin,darkblue] (-0.28,0) to[out=90,in=-90] (0.28,.6);
	\draw[-,thin,darkblue] (0.28,-.6) to[out=90,in=-90] (-0.28,0);
	\draw[-,line width=4pt,white] (0.28,0) to[out=90,in=-90] (-0.28,.6);
	\draw[-,line width=4pt,white] (-0.28,-.6) to[out=90,in=-90] (0.28,0);
	\draw[-,thin,darkblue] (0.28,0) to[out=90,in=-90] (-0.28,.6);
	\draw[<-,thin,darkblue] (-0.28,-.6) to[out=90,in=-90] (0.28,0);
\end{tikzpicture}
}
&=\mathord{
\begin{tikzpicture}[baseline = 0]
	\draw[->,thin,darkblue] (0.08,-.6) to (0.08,.6);
	\draw[<-,thin,darkblue] (-0.28,-.6) to (-0.28,.6);
 \end{tikzpicture}}
\quad\text{if $k=0$,}&
\mathord{
\begin{tikzpicture}[baseline = 0]
	\draw[-,thin,darkblue] (-0.28,-.6) to[out=90,in=-90] (0.28,0);
	\draw[->,thin,darkblue] (0.28,0) to[out=90,in=-90] (-0.28,.6);
	\draw[-,line width=4pt,white] (-0.28,0) to[out=90,in=-90] (0.28,.6);
	\draw[-,line width=4pt,white] (0.28,-.6) to[out=90,in=-90] (-0.28,0);
	\draw[-,thin,darkblue] (-0.28,0) to[out=90,in=-90] (0.28,.6);
	\draw[<-,thin,darkblue] (0.28,-.6) to[out=90,in=-90] (-0.28,0);
\end{tikzpicture}
}
&
=\mathord{
\begin{tikzpicture}[baseline = 0]
	\draw[<-,thin,darkblue] (0.08,-.6) to (0.08,.6);
	\draw[->,thin,darkblue] (-0.28,-.6) to (-0.28,.6);
 \end{tikzpicture}}\quad\text{if $k=0$.}
\end{align}
In fact, the defining relations for $\HEIS_0(z,t)$ from Definition~\ref{def1} are
exactly the same as the ones for the affine HOMFLY-PT
skein category $\AOS(z,t)$ from \cite[Theorem 1.1 and $\S$4]{Bskein}.
Thus,
\[
    \HEIS_0(z,t) = \AOS(z,t).
\]
In this case, most 
of the other relations that we
need have already been proved in {\em
  loc.\ cit.}. 
However, the arguments there exploit a theorem of Turaev
\cite[Lemma I.3.3]{Turaev1} to 
establish all of the relations that
do not involve dots; the approach described below reproves all of these relations
in a way that is indendent of Turaev's work.

When $k > 0$, the assertion that the morphisms (\ref{invrel}) and
(\ref{invrel2}) are two-sided inverses implies the following relations:
\begin{align}\label{tea1}
\mathord{ 
\begin{tikzpicture}[baseline = 0]
	\draw[->,thin,darkblue] (-0.28,0) to[out=90,in=-90] (0.28,.6);
	\draw[-,thin,darkblue] (0.28,-.6) to[out=90,in=-90] (-0.28,0);
	\draw[-,line width=4pt,white] (0.28,0) to[out=90,in=-90] (-0.28,.6);
	\draw[-,line width=4pt,white] (-0.28,-.6) to[out=90,in=-90] (0.28,0);
	\draw[-,thin,darkblue] (0.28,0) to[out=90,in=-90] (-0.28,.6);
	\draw[<-,thin,darkblue] (-0.28,-.6) to[out=90,in=-90] (0.28,0);
\end{tikzpicture}
}
&=\mathord{
\begin{tikzpicture}[baseline = 0]
	\draw[->,thin,darkblue] (0.08,-.6) to (0.08,.6);
	\draw[<-,thin,darkblue] (-0.28,-.6) to (-0.28,.6);
 \end{tikzpicture}}
\quad\text{if $k>0$,}&
\mathord{
\begin{tikzpicture}[baseline = 0]
	\draw[-,thin,darkblue] (-0.28,-.6) to[out=90,in=-90] (0.28,0);
	\draw[->,thin,darkblue] (0.28,0) to[out=90,in=-90] (-0.28,.6);
	\draw[-,line width=4pt,white] (-0.28,0) to[out=90,in=-90] (0.28,.6);
	\draw[-,line width=4pt,white] (0.28,-.6) to[out=90,in=-90] (-0.28,0);
	\draw[-,thin,darkblue] (-0.28,0) to[out=90,in=-90] (0.28,.6);
	\draw[<-,thin,darkblue] (0.28,-.6) to[out=90,in=-90] (-0.28,0);
\end{tikzpicture}
}
&
=\mathord{
\begin{tikzpicture}[baseline = 0]
	\draw[<-,thin,darkblue] (0.08,-.6) to (0.08,.6);
	\draw[->,thin,darkblue] (-0.28,-.6) to (-0.28,.6);
\end{tikzpicture}}
-\sum_{a=0}^{k-1}
\mathord{
\begin{tikzpicture}[baseline=0mm]
	\draw[-,thin,darkblue] (0.3,0.6) to[out=-90, in=0] (0,0.1);
	\draw[->,thin,darkblue] (0,0.1) to[out = 180, in = -90] (-0.3,0.6);
      \node at (0.02,0.28) {$\color{darkblue}\scriptstyle{a}$};
	\draw[<-,thin,darkblue] (0.3,-.6) to[out=90, in=0] (0,-0.1);
	\draw[-,thin,darkblue] (0,-0.1) to[out = 180, in = 90] (-0.3,-.6);
   \node at (-0.25,-0.4) {$\dot$};
   \node at (-.4,-.3) {$\color{darkblue}\scriptstyle{a}$};
      \node at (0.02,0.11) {$\diamond$};
\end{tikzpicture}}\quad\text{if $k>0$,}
\end{align}\begin{align}
\mathord{
\begin{tikzpicture}[baseline = 0]
	\draw[-,thin,darkblue] (0.25,-0.) to[out=-90, in=0] (0,-0.25);
	\draw[-,thin,darkblue] (0,-0.25) to[out = 180, in = -90] (-0.25,-0);
	\draw[-,thin,darkblue] (-0.25,.6) to[out=300,in=90] (0.25,-0);
	\draw[-,line width=4pt,white] (0.25,.6) to[out=240,in=90] (-0.25,-0);
	\draw[<-,thin,darkblue] (0.25,.6) to[out=240,in=90] (-0.25,-0);
\end{tikzpicture}
}&=0
\:\text{if $k>0$,}&
\mathord{
\begin{tikzpicture}[baseline = -1.5mm]
	\draw[-,thin,darkblue] (0.25,-.5) to[out=120,in=-90] (-0.25,0.1);
	\draw[-,line width=4pt,white] (-0.25,-.5) to[out=60,in=-90] (0.25,0.1);
	\draw[<-,thin,darkblue] (-0.25,-.5) to[out=60,in=-90] (0.25,0.1);
	\draw[-,thin,darkblue] (0.25,0.1) to[out=90, in=0] (0,0.35);
	\draw[-,thin,darkblue] (0,0.35) to[out = 180, in = 90] (-0.25,0.1);
      \node at (-0.45,0.05) {$\color{darkblue}\scriptstyle{a}$};
      \node at (-0.25,0.05) {$\dot$};
\end{tikzpicture}
}
&=0 \:\text{if $0 \leq a < k$,}
&\mathord{
\begin{tikzpicture}[baseline = 2mm]
  \draw[<-,thin,darkblue] (0,0.4) to[out=180,in=90] (-.2,0.2);
  \draw[-,thin,darkblue] (0.2,0.2) to[out=90,in=0] (0,.4);
 \draw[-,thin,darkblue] (-.2,0.2) to[out=-90,in=180] (0,0);
  \draw[-,thin,darkblue] (0,0) to[out=0,in=-90] (0.2,0.2);
   \node at (-0.2,0.2) {$\dot$};
   \node at (-0.4,0.2) {$\color{darkblue}\scriptstyle{a}$};
\end{tikzpicture}
}&= -\delta_{a,k} \displaystyle{t^{-1}z^{-1}} 1_\unit
\:\text{if $0 < a \leq k$.}\label{tea0}
\end{align}
To derive these relations, we multiplied the matrices (\ref{invrel}) and (\ref{invrel2})
in both orders, then equated the result with the appropriate identity
matrix.
The following useful relation is an easy exercise at this point; one
needs to use (\ref{teaminus}), (\ref{impose}), (\ref{leftwards}) and (\ref{tea0}):
\begin{align}
\mathord{
\begin{tikzpicture}[baseline = 0]
	\draw[<-,thin,darkblue] (0.25,.6) to[out=240,in=90] (-0.25,-0);
	\draw[-,thin,darkblue] (0.25,-0.) to[out=-90, in=0] (0,-0.25);
	\draw[-,thin,darkblue] (0,-0.25) to[out = 180, in = -90] (-0.25,-0);
	\draw[-,line width=4pt,white] (-0.25,.6) to[out=300,in=90] (0.25,-0);
	\draw[-,thin,darkblue] (-0.25,.6) to[out=300,in=90] (0.25,-0);
      \node at (-0.45,-0) {$\color{darkblue}\scriptstyle{a}$};
      \node at (-0.24,-0) {$\dot$};
\end{tikzpicture}
}&= \delta_{a,0}\,
t\:\mathord{
\begin{tikzpicture}[baseline = 1mm]
	\draw[<-,thin,darkblue] (0.4,0.4) to[out=-90, in=0] (0.1,0);
	\draw[-,thin,darkblue] (0.1,0) to[out = 180, in = -90] (-0.2,0.4);
 \end{tikzpicture}
}\qquad\quad\text{for $0 \leq a \leq k$.}\label{exercise}
\end{align}

Finally, when $k < 0$, we will need the following relations which are
deduced from (\ref{invrel}) and (\ref{invrel2}) by the same argument
as explained in the previous paragraph:
\begin{align}\label{tea3}
\mathord{
\begin{tikzpicture}[baseline = 0]
	\draw[->,thin,darkblue] (-0.28,0) to[out=90,in=-90] (0.28,.6);
	\draw[-,line width=4pt,white] (0.28,0) to[out=90,in=-90] (-0.28,.6);
	\draw[-,thin,darkblue] (0.28,0) to[out=90,in=-90] (-0.28,.6);
	\draw[<-,thin,darkblue] (-0.28,-.6) to[out=90,in=-90] (0.28,0);
	\draw[-,line width=4pt,white] (0.28,-.6) to[out=90,in=-90] (-0.28,0);
	\draw[-,thin,darkblue] (0.28,-.6) to[out=90,in=-90] (-0.28,0);
\end{tikzpicture}
}
&=\mathord{
\begin{tikzpicture}[baseline = 0]
	\draw[->,thin,darkblue] (0.08,-.6) to (0.08,.6);
	\draw[<-,thin,darkblue] (-0.28,-.6) to (-0.28,.6);
 \end{tikzpicture}}
-\sum_{a=0}^{k-1}\:
\mathord{
\begin{tikzpicture}[baseline=0mm]
	\draw[<-,thin,darkblue] (0.3,0.6) to[out=-90, in=0] (0,0.1);
	\draw[-,thin,darkblue] (0,0.1) to[out = 180, in = -90] (-0.3,0.6);
      \node at (0.02,-0.28) {$\color{darkblue}\scriptstyle{a}$};
	\draw[-,thin,darkblue] (0.3,-.6) to[out=90, in=0] (0,-0.1);
	\draw[->,thin,darkblue] (0,-0.1) to[out = 180, in = 90] (-0.3,-.6);
   \node at (0.28,0.35) {$\dot$};
   \node at (.43,.33) {$\color{darkblue}\scriptstyle{a}$};
      \node at (0.02,-0.11) {$\heart$};
\end{tikzpicture}}\quad\text{if $k<0$,}&
\mathord{
\begin{tikzpicture}[baseline = 0]
	\draw[-,thin,darkblue] (-0.28,-.6) to[out=90,in=-90] (0.28,0);
	\draw[-,line width=4pt,white] (0.28,-.6) to[out=90,in=-90] (-0.28,0);
	\draw[-,thin,darkblue] (-0.28,0) to[out=90,in=-90] (0.28,.6);
	\draw[<-,thin,darkblue] (0.28,-.6) to[out=90,in=-90] (-0.28,0);
	\draw[-,line width=4pt,white] (0.28,0) to[out=90,in=-90] (-0.28,.6);
	\draw[->,thin,darkblue] (0.28,0) to[out=90,in=-90] (-0.28,.6);
\end{tikzpicture}
}
&
=\mathord{
\begin{tikzpicture}[baseline = 0]
	\draw[<-,thin,darkblue] (0.08,-.6) to (0.08,.6);
	\draw[->,thin,darkblue] (-0.28,-.6) to (-0.28,.6);
\end{tikzpicture}}
\quad\text{if $k<0$,}
\end{align}\begin{align}\label{tea2}
\mathord{
\begin{tikzpicture}[baseline = -1.5mm]
	\draw[-,thin,darkblue] (-0.25,-.5) to[out=60,in=-90] (0.25,0.1);
	\draw[-,line width=4pt,white] (0.25,-.5) to[out=120,in=-90] (-0.25,0.1);
	\draw[<-,thin,darkblue] (0.25,-.5) to[out=120,in=-90] (-0.25,0.1);
	\draw[-,thin,darkblue] (0.25,0.1) to[out=90, in=0] (0,0.35);
	\draw[-,thin,darkblue] (0,0.35) to[out = 180, in = 90] (-0.25,0.1);
\end{tikzpicture}
}&=0
\:\text{if $k<0$,}
&\mathord{
\begin{tikzpicture}[baseline = 0]
	\draw[-,thin,darkblue] (0.25,.6) to[out=240,in=90] (-0.25,-0);
	\draw[-,thin,darkblue] (0.25,-0.) to[out=-90, in=0] (0,-0.25);
	\draw[-,thin,darkblue] (0,-0.25) to[out = 180, in = -90] (-0.25,-0);
	\draw[-,line width=4pt,white] (-0.25,.6) to[out=300,in=90] (0.25,-0);
	\draw[<-,thin,darkblue] (-0.25,.6) to[out=300,in=90] (0.25,-0);
      \node at (0.45,-0) {$\color{darkblue}\scriptstyle{a}$};
      \node at (0.24,-0) {$\dot$};
\end{tikzpicture}
}
&=0 \:\text{if $0 \leq a < -k$,}
&
\mathord{
\begin{tikzpicture}[baseline = 2mm]
  \draw[->,thin,darkblue] (0.2,0.2) to[out=90,in=0] (0,.4);
  \draw[-,thin,darkblue] (0,0.4) to[out=180,in=90] (-.2,0.2);
\draw[-,thin,darkblue] (-.2,0.2) to[out=-90,in=180] (0,0);
  \draw[-,thin,darkblue] (0,0) to[out=0,in=-90] (0.2,0.2);
   \node at (0.2,0.2) {$\dot$};
   \node at (0.4,0.2) {$\color{darkblue}\scriptstyle{a}$};
\end{tikzpicture}
}&= 
-\delta_{a,0}
\displaystyle{t^{-1}z^{-1}}1_\unit
\:\text{if $0 \leq a < -k$.}
\end{align}

Now we are going to consider the counterpart of 
the morphism (\ref{invrel}) defined
using the negative instead of positive rightward crossing:
\begin{align}
\label{invrel1}
\left\{\begin{array}{rl}
\left[\!\!\!\!\!
\begin{array}{r}
\mathord{
\begin{tikzpicture}[baseline = 0]
	\draw[<-,thin,darkblue] (0.28,-.3) to (-0.28,.4);
	\draw[-,line width=4pt,white] (-0.28,-.3) to (0.28,.4);
	\draw[->,thin,darkblue] (-0.28,-.3) to (0.28,.4);
   \end{tikzpicture}
}\\
\mathord{
\begin{tikzpicture}[baseline = 1mm]
	\draw[<-,thin,darkblue] (0.4,0) to[out=90, in=0] (0.1,0.4);
      \node at (-0.15,0.45) {$\phantom\bullet$};
	\draw[-,thin,darkblue] (0.1,0.4) to[out = 180, in = 90] (-0.2,0);
\end{tikzpicture}
}\\
\mathord{
\begin{tikzpicture}[baseline = 1mm]
	\draw[<-,thin,darkblue] (0.4,0) to[out=90, in=0] (0.1,0.4);
	\draw[-,thin,darkblue] (0.1,0.4) to[out = 180, in = 90] (-0.2,0);
      \node at (-0.15,0.45) {$\phantom\bullet$};
      \node at (-0.15,0.2) {$\dot$};
\end{tikzpicture}
}\\\vdots\:\:\;\\
\mathord{
\begin{tikzpicture}[baseline = 1mm]
	\draw[<-,thin,darkblue] (0.4,0) to[out=90, in=0] (0.1,0.4);
	\draw[-,thin,darkblue] (0.1,0.4) to[out = 180, in = 90] (-0.2,0);
     \node at (-0.52,0.2) {$\color{darkblue}\scriptstyle{k-1}$};
      \node at (-0.15,0.42) {$\phantom\bullet$};
      \node at (-0.15,0.2) {$\dot$};
\end{tikzpicture}
}
\end{array}
\right]
:
\up \otimes \down \rightarrow
\down \otimes \up \oplus \unit^{\oplus k}\hspace{4.5mm}
&\text{if $k >
  0$,}\\\\
\left[\:
\mathord{
\begin{tikzpicture}[baseline = 0]
	\draw[<-,thin,darkblue] (0.28,-.3) to (-0.28,.4);
	\draw[-,line width=4pt,white] (-0.28,-.3) to (0.28,.4);
	\draw[->,thin,darkblue] (-0.28,-.3) to (0.28,.4);
\end{tikzpicture}
}\:\:\:
\mathord{
\begin{tikzpicture}[baseline = -0.9mm]
	\draw[<-,thin,darkblue] (0.4,0.2) to[out=-90, in=0] (0.1,-.2);
	\draw[-,thin,darkblue] (0.1,-.2) to[out = 180, in = -90] (-0.2,0.2);
\end{tikzpicture}
}
\:\:\:
\mathord{
\begin{tikzpicture}[baseline = -0.9mm]
	\draw[<-,thin,darkblue] (0.4,0.2) to[out=-90, in=0] (0.1,-.2);
	\draw[-,thin,darkblue] (0.1,-.2) to[out = 180, in = -90] (-0.2,0.2);
      \node at (0.38,0) {$\dot$};
\end{tikzpicture}
}
\:\:\:\cdots
\:\:\:
\mathord{
\begin{tikzpicture}[baseline = -0.9mm]
	\draw[<-,thin,darkblue] (0.4,0.2) to[out=-90, in=0] (0.1,-.2);
	\draw[-,thin,darkblue] (0.1,-.2) to[out = 180, in = -90] (-0.2,0.2);
     \node at (0.83,0) {$\color{darkblue}\scriptstyle{-k-1}$};
      \node at (0.38,0) {$\dot$};
\end{tikzpicture}
}
\right]
:\up \otimes \down \oplus 
\unit^{\oplus (-k)}
\rightarrow
 \down \otimes  \up&\text{if $k \leq 0$.}
\end{array}
\right.\end{align}

\begin{lemma}\label{irrelevant}
The morphism (\ref{invrel1}) is invertible with two-sided inverse
\begin{align}
\label{invrel3}
\left\{\begin{array}{rl}
\left[\:
\mathord{
\begin{tikzpicture}[baseline = 0]
	\draw[->,thin,darkblue] (0.28,-.3) to (-0.28,.4);
	\draw[-,line width=4pt,white] (-0.28,-.3) to (0.28,.4);
	\draw[<-,thin,darkblue] (-0.28,-.3) to (0.28,.4);
\end{tikzpicture}
}
\:\:\:
\mathord{
\begin{tikzpicture}[baseline = 1mm]
	\draw[-,thin,darkblue] (0.4,0.4) to[out=-90, in=0] (0.1,0);
	\draw[->,thin,darkblue] (0.1,0) to[out = 180, in = -90] (-0.2,0.4);
      \node at (0.12,-0.2) {$\color{darkblue}\scriptstyle{0}$};
      \node at (0.12,0.01) {$\heart$};
\end{tikzpicture}
}
\:\:\:\cdots\:\:\:
\mathord{
\begin{tikzpicture}[baseline = 1mm]
	\draw[-,thin,darkblue] (0.4,0.4) to[out=-90, in=0] (0.1,0);
	\draw[->,thin,darkblue] (0.1,0) to[out = 180, in = -90] (-0.2,0.4);
      \node at (0.12,-0.2) {$\color{darkblue}\scriptstyle{k-1}$};
      \node at (0.12,0.01) {$\heart$};
\end{tikzpicture}
}
\right]
:\down\otimes\up \oplus
\unit^{\oplus k}
\rightarrow
 \up \otimes  \down \hspace{4.5mm}
&\text{if $k > 0$},\\\\
\left[\!\!\!\!\!
\begin{array}{r}
\mathord{
\:\begin{tikzpicture}[baseline = 1mm]
	\draw[<-,thin,darkblue] (-0.28,-.3) to (0.28,.4);
	\draw[-,line width=4pt,white] (0.28,-.3) to (-0.28,.4);
	\draw[->,thin,darkblue] (0.28,-.3) to (-0.28,.4);
\end{tikzpicture}
}\\
\mathord{
\begin{tikzpicture}[baseline = 1mm]
	\draw[-,thin,darkblue] (0.4,0) to[out=90, in=0] (0.1,0.4);
	\draw[->,thin,darkblue] (0.1,0.4) to[out = 180, in = 90] (-0.2,0);
      \node at (0.12,0.62) {$\color{darkblue}\scriptstyle{0}$};
      \node at (0.12,0.4) {$\diamond$};
\end{tikzpicture}
}
\\\vdots\:\:\:\\
\mathord{
\begin{tikzpicture}[baseline = 1mm]
	\draw[-,thin,darkblue] (0.4,0) to[out=90, in=0] (0.1,0.4);
	\draw[->,thin,darkblue] (0.1,0.4) to[out = 180, in = 90] (-0.2,0);
      \node at (0.12,0.6) {$\color{darkblue}\scriptstyle{-k-1}$};
      \node at (0.12,0.4) {$\diamond$};
\end{tikzpicture}
}\!\!\!
\end{array}
\right]
:
\down \otimes \up \rightarrow
\up \otimes \down \oplus \unit^{\oplus(-k)}
&\text{if $k \leq
  0$.}
\end{array}
\right.\end{align}
Moreover, we have that
\begin{align}\label{septimus}
\mathord{
\begin{tikzpicture}[baseline = .8mm]
  \draw[<-,thin,darkblue] (0.2,0.2) to[out=90,in=0] (0,.4);
  \draw[-,thin,darkblue] (0,0.4) to[out=180,in=90] (-.2,0.2);
\draw[-,thin,darkblue] (-.2,0.2) to[out=-90,in=180] (0,0);
  \draw[-,thin,darkblue] (0,0) to[out=0,in=-90] (0.2,0.2);
      \node at (-0.2,0.2) {$\dot$};
      \node at (-0.35,0.2) {$\color{darkblue}\scriptstyle{k}$};
\end{tikzpicture}
}&=-{t^{-1}z^{-1}}1_\unit\:\text{if $k > 0$,}&
\anticlockleft
&= {(t z^{-1}-t^{-1}z^{-1})}1_\unit\:\text{if $k = 0$,}
&\mathord{\begin{tikzpicture}[baseline = .8mm]
  \draw[-,thin,darkblue] (0.2,0.2) to[out=90,in=0] (0,.4);
  \draw[->,thin,darkblue] (0,0.4) to[out=180,in=90] (-.2,0.2);
\draw[-,thin,darkblue] (-.2,0.2) to[out=-90,in=180] (0,0);
  \draw[-,thin,darkblue] (0,0) to[out=0,in=-90] (0.2,0.2);
 \end{tikzpicture}
}
&= -{t^{-1}z^{-1}}1_\unit\:\text{if $k < 0$,}
\end{align}
\begin{align}\label{gloop}
\mathord{
\begin{tikzpicture}[baseline = 1mm]
	\draw[-,thin,darkblue] (0.4,0.4) to[out=-90, in=0] (0.1,0);
	\draw[->,thin,darkblue] (0.1,0) to[out = 180, in = -90] (-0.2,0.4);
 \end{tikzpicture}
}
&=
\left\{
\begin{array}{cl}
\displaystyle{tz^{-1}}\:\mathord{
\begin{tikzpicture}[baseline = 0.5mm]
	\draw[-,thin,darkblue] (0.4,0.4) to[out=-90, in=0] (0.1,0);
	\draw[->,thin,darkblue] (0.1,0) to[out = 180, in = -90] (-0.2,0.4);
      \node at (0.1,-0.21) {$\color{darkblue}\scriptstyle{0}$};
      \node at (0.12,0.01) {$\heart$};
\end{tikzpicture}
}
\hspace{3.5mm}&\!\!\!\text{if $k > 0$,}\\
t^{-1}
\mathord{
\begin{tikzpicture}[baseline = 0]
	\draw[-,thin,darkblue] (0.25,.6) to[out=240,in=90] (-0.25,-0);
	\draw[-,thin,darkblue] (0.25,-0.) to[out=-90, in=0] (0,-0.25);
	\draw[-,thin,darkblue] (0,-0.25) to[out = 180, in = -90] (-0.25,-0);
	\draw[-,line width=4pt,white] (-0.25,.6) to[out=300,in=90] (0.25,-0);
	\draw[<-,thin,darkblue] (-0.25,.6) to[out=300,in=90] (0.25,-0);
      \node at (0.54,-0) {$\color{darkblue}\scriptstyle{-k}$};
      \node at (0.24,-0) {$\dot$};
\end{tikzpicture}
}
&\!\!\!\text{if $k \leq 0$,}
\end{array}\right.
&\mathord{
\begin{tikzpicture}[baseline = 1mm]
	\draw[-,thin,darkblue] (0.4,0) to[out=90, in=0] (0.1,0.4);
	\draw[->,thin,darkblue] (0.1,0.4) to[out = 180, in = 90] (-0.2,0);
\end{tikzpicture}
}
&=
\left\{
\begin{array}{cl}
t\mathord{
\begin{tikzpicture}[baseline = -1.5mm]
	\draw[-,thin,darkblue] (0.25,-.5) to[out=120,in=-90] (-0.25,0.1);
	\draw[-,line width=4pt,white] (-0.25,-.5) to[out=60,in=-90] (0.25,0.1);
	\draw[<-,thin,darkblue] (-0.25,-.5) to[out=60,in=-90] (0.25,0.1);
	\draw[-,thin,darkblue] (0.25,0.1) to[out=90, in=0] (0,0.35);
	\draw[-,thin,darkblue] (0,0.35) to[out = 180, in = 90] (-0.25,0.1);
      \node at (-0.42,0.05) {$\color{darkblue}\scriptstyle{k}$};
      \node at (-0.25,0.05) {$\dot$};
\end{tikzpicture}
}
&\text{if $k  > 0$,}\\
t^{-1}\mathord{
\begin{tikzpicture}[baseline = -1.5mm]
	\draw[<-,thin,darkblue] (-0.25,-.5) to[out=60,in=-90] (0.25,0.1);
	\draw[-,thin,darkblue] (0.25,0.1) to[out=90, in=0] (0,0.35);
	\draw[-,thin,darkblue] (0,0.35) to[out = 180, in = 90] (-0.25,0.1);
	\draw[-,line width=4pt,white] (0.25,-.5) to[out=120,in=-90] (-0.25,0.1);
	\draw[-,thin,darkblue] (0.25,-.5) to[out=120,in=-90] (-0.25,0.1);
\end{tikzpicture}
}
&\text{if $k  = 0$,}\\
\displaystyle{tz^{-1}}\mathord{
\begin{tikzpicture}[baseline = 1mm]
	\draw[-,thin,darkblue] (0.4,0) to[out=90, in=0] (0.1,0.4);
	\draw[->,thin,darkblue] (0.1,0.4) to[out = 180, in = 90] (-0.2,0);
      \node at (0.12,0.6) {$\color{darkblue}\scriptstyle{-k-1}$};
     \node at (0.12,0.4) {$\diamond$};
     \node at (-0.4,0.23) {$\color{darkblue}\scriptstyle{-1}$};
      \node at (-0.15,0.21) {$\dot$};
\end{tikzpicture}}
&\text{if $k < 0$.}\end{array}\right.
\end{align}
\end{lemma}

\section{Second approach}\label{second}

Our second presentation for $\HEIS_k(z,t)$ is very similar to the first
presentation, but we invert the morphism (\ref{invrel1})
instead of (\ref{invrel}).

\begin{definition}\label{def2}
The {\em quantum Heisenberg category} $\HEIS_k(z,t)$ is the strict
$\k$-linear monoidal category obtained from $\AH(z)$ by 
adjoining a right dual $\down$ to $\up$ as 
explained in the introduction,
together with the matrix entries of the 
morphism (\ref{invrel3}),
which we declare to be a two-sided inverse to 
(\ref{invrel1}).
In addition, we impose the relation (\ref{septimus})
for the leftward cups and caps which are defined in this approach from (\ref{gloop}).
Define the other leftward crossing, i.e., the one which does not
appear in (\ref{invrel3}), 
so the leftward skein
relation (\ref{leftwardsskein}) holds.
Also set
\begin{align}
\mathord{
\begin{tikzpicture}[baseline = 1mm]
	\draw[-,thin,darkblue] (0.4,0.4) to[out=-90, in=0] (0.1,0);
	\draw[->,thin,darkblue] (0.1,0) to[out = 180, in = -90] (-0.2,0.4);
      \node at (0.12,-0.2) {$\color{darkblue}\scriptstyle{0}$};
      \node at (0.12,0.01) {$\diamond$};
\end{tikzpicture}
}
&:=\mathord{
\begin{tikzpicture}[baseline = 1mm]
	\draw[-,thin,darkblue] (0.4,0.4) to[out=-90, in=0] (0.1,0);
	\draw[->,thin,darkblue] (0.1,0) to[out = 180, in = -90] (-0.2,0.4);
      \node at (0.12,-0.2) {$\color{darkblue}\scriptstyle{0}$};
      \node at (0.12,0.01) {$\heart$};
\end{tikzpicture}
}-z
\mathord{
\begin{tikzpicture}[baseline = 0]
	\draw[<-,thin,darkblue] (-0.25,.6) to[out=300,in=90] (0.25,-0);
	\draw[-,thin,darkblue] (0.25,-0) to[out=-90, in=0] (0,-0.25);
	\draw[-,thin,darkblue] (0,-0.25) to[out = 180, in = -90] (-0.25,-0);
	\draw[-,line width=4pt,white] (0.25,.6) to[out=240,in=90] (-0.25,-0);
	\draw[-,thin,darkblue] (0.25,.6) to[out=240,in=90] (-0.25,-0);
\end{tikzpicture}
}\quad\text{if $k > 0$},
&\mathord{
\begin{tikzpicture}[baseline = 1mm]
	\draw[-,thin,darkblue] (0.4,0.4) to[out=-90, in=0] (0.1,0);
	\draw[->,thin,darkblue] (0.1,0) to[out = 180, in = -90] (-0.2,0.4);
      \node at (0.12,-0.2) {$\color{darkblue}\scriptstyle{a}$};
      \node at (0.12,0.01) {$\diamond$};
\end{tikzpicture}
}
&:=\mathord{
\begin{tikzpicture}[baseline = 1mm]
	\draw[-,thin,darkblue] (0.4,0.4) to[out=-90, in=0] (0.1,0);
	\draw[->,thin,darkblue] (0.1,0) to[out = 180, in = -90] (-0.2,0.4);
      \node at (0.12,-0.2) {$\color{darkblue}\scriptstyle{a}$};
      \node at (0.12,0.01) {$\heart$};
\end{tikzpicture}
}
\quad\text{if $0 < a < k$,}\\
\mathord{
\begin{tikzpicture}[baseline = 1mm]
	\draw[-,thin,darkblue] (0.4,0) to[out=90, in=0] (0.1,0.4);
	\draw[->,thin,darkblue] (0.1,0.4) to[out = 180, in = 90] (-0.2,0);
      \node at (0.12,0.6) {$\color{darkblue}\scriptstyle{0}$};
      \node at (0.12,0.4) {$\heart$};
\end{tikzpicture}
}&:=
\mathord{
\begin{tikzpicture}[baseline = 1mm]
	\draw[-,thin,darkblue] (0.4,0) to[out=90, in=0] (0.1,0.4);
	\draw[->,thin,darkblue] (0.1,0.4) to[out = 180, in = 90] (-0.2,0);
      \node at (0.12,0.6) {$\color{darkblue}\scriptstyle{0}$};
      \node at (0.12,0.4) {$\diamond$};
\end{tikzpicture}
}-z
\mathord{
\begin{tikzpicture}[baseline = -1.5mm]
	\draw[<-,thin,darkblue] (-0.25,-.5) to[out=60,in=-90] (0.25,0.1);
	\draw[-,line width=4pt,white] (0.25,-.5) to[out=120,in=-90] (-0.25,0.1);
	\draw[-,thin,darkblue] (0.25,-.5) to[out=120,in=-90] (-0.25,0.1);
	\draw[-,thin,darkblue] (0.25,0.1) to[out=90, in=0] (0,0.35);
	\draw[-,thin,darkblue] (0,0.35) to[out = 180, in = 90] (-0.25,0.1);
\end{tikzpicture}
}\quad\text{if $k < 0$,}
&
\mathord{
\begin{tikzpicture}[baseline = 1mm]
	\draw[-,thin,darkblue] (0.4,0) to[out=90, in=0] (0.1,0.4);
	\draw[->,thin,darkblue] (0.1,0.4) to[out = 180, in = 90] (-0.2,0);
      \node at (0.12,0.6) {$\color{darkblue}\scriptstyle{a}$};
      \node at (0.12,0.4) {$\heart$};
\end{tikzpicture}
}&:=
\mathord{
\begin{tikzpicture}[baseline = 1mm]
	\draw[-,thin,darkblue] (0.4,0) to[out=90, in=0] (0.1,0.4);
	\draw[->,thin,darkblue] (0.1,0.4) to[out = 180, in = 90] (-0.2,0);
      \node at (0.12,0.6) {$\color{darkblue}\scriptstyle{a}$};
      \node at (0.12,0.4) {$\diamond$};
\end{tikzpicture}
}\quad\text{if $0 < a < -k$.}\label{nakano2}
\end{align}
Finally define the $(+)$- and $(-)$-bubbles from (\ref{fake0})--(\ref{fake1})
as before.
\end{definition} 

\begin{theorem}
Definitions~\ref{def1} and \ref{def2} 
give two different presentations for the same monoidal category, with
all of the named morphisms introduced in the two definitions being the same.
Moreover, there is a unique 
isomorphism of $\k$-linear monoidal categories
\begin{equation}\label{om}
\Omega_k:\HEIS_k(z,t)\rightarrow \HEIS_{-k}(z,t^{-1})^{\operatorname{op}}
\end{equation}
sending
\begin{align*}
\mathord{
\begin{tikzpicture}[baseline = -.5mm]
	\draw[->,thin,darkblue] (0,-.3) to (0,.4);
      \node at (0,0.05) {$\dot$};
\end{tikzpicture}
}&\mapsto
\mathord{
\begin{tikzpicture}[baseline = -.5mm]
	\draw[<-,thin,darkblue] (0,-.3) to (0,.4);
      \node at (0,0.05) {$\dot$};
\end{tikzpicture}
},&
\mathord{
\begin{tikzpicture}[baseline = -.5mm]
	\draw[->,thin,darkblue] (0.28,-.3) to (-0.28,.4);
	\draw[line width=4pt,white,-] (-0.28,-.3) to (0.28,.4);
	\draw[thin,darkblue,->] (-0.28,-.3) to (0.28,.4);
\end{tikzpicture}
}&\mapsto 
-\mathord{
\begin{tikzpicture}[baseline = -.5mm]
	\draw[thin,darkblue,<-] (-0.28,-.3) to (0.28,.4);
	\draw[-,line width=4pt,white] (0.28,-.3) to (-0.28,.4);
	\draw[<-,thin,darkblue] (0.28,-.3) to (-0.28,.4);
\end{tikzpicture}
}\:,&
\mathord{
\begin{tikzpicture}[baseline = 1mm]
	\draw[<-,thin,darkblue] (0.4,0) to[out=90, in=0] (0.1,0.4);
	\draw[-,thin,darkblue] (0.1,0.4) to[out = 180, in = 90] (-0.2,0);
\end{tikzpicture}
}\:&\mapsto \mathord{
\begin{tikzpicture}[baseline = 1mm]
	\draw[<-,thin,darkblue] (0.4,0.4) to[out=-90, in=0] (0.1,0);
	\draw[-,thin,darkblue] (0.1,0) to[out = 180, in = -90] (-0.2,0.4);
\end{tikzpicture}
},\:&
\mathord{
\begin{tikzpicture}[baseline = 1mm]
	\draw[<-,thin,darkblue] (0.4,0.4) to[out=-90, in=0] (0.1,0);
	\draw[-,thin,darkblue] (0.1,0) to[out = 180, in = -90] (-0.2,0.4);
\end{tikzpicture}
}\:&\mapsto
\mathord{
\begin{tikzpicture}[baseline = 1mm]
	\draw[<-,thin,darkblue] (0.4,0) to[out=90, in=0] (0.1,0.4);
	\draw[-,thin,darkblue] (0.1,0.4) to[out = 180, in = 90] (-0.2,0);
\end{tikzpicture}
}\:.
\end{align*}
The effect of $\Omega_k$ on the other morphisms
is as follows:
\begin{align*}
\mathord{
\begin{tikzpicture}[baseline = -.5mm]
	\draw[<-,thin,darkblue] (0,-.3) to (0,.4);
      \node at (0,0.05) {$\dot$};
\end{tikzpicture}
}&\mapsto
\mathord{
\begin{tikzpicture}[baseline = -.5mm]
	\draw[->,thin,darkblue] (0,-.3) to (0,.4);
      \node at (0,0.05) {$\dot$};
\end{tikzpicture}
},&
\mathord{
\begin{tikzpicture}[baseline = -.5mm]
	\draw[thin,darkblue,->] (-0.28,-.3) to (0.28,.4);
	\draw[-,line width=4pt,white] (0.28,-.3) to (-0.28,.4);
	\draw[<-,thin,darkblue] (0.28,-.3) to (-0.28,.4);
\end{tikzpicture}
}&\mapsto 
-\mathord{
\begin{tikzpicture}[baseline = -.5mm]
	\draw[<-,thin,darkblue] (0.28,-.3) to (-0.28,.4);
	\draw[line width=4pt,white,->] (-0.28,-.3) to (0.28,.4);
	\draw[thin,darkblue,->] (-0.28,-.3) to (0.28,.4);
\end{tikzpicture}
}\:,
&
\mathord{
\begin{tikzpicture}[baseline = -.5mm]
	\draw[<-,thin,darkblue] (0.28,-.3) to (-0.28,.4);
	\draw[line width=4pt,white,<-] (-0.28,-.3) to (0.28,.4);
	\draw[thin,darkblue,<-] (-0.28,-.3) to (0.28,.4);
\end{tikzpicture}
}&\mapsto 
-\mathord{
\begin{tikzpicture}[baseline = -.5mm]
	\draw[thin,darkblue,->] (-0.28,-.3) to (0.28,.4);
	\draw[-,line width=4pt,white] (0.28,-.3) to (-0.28,.4);
	\draw[->,thin,darkblue] (0.28,-.3) to (-0.28,.4);
\end{tikzpicture}
}\:,&
\mathord{
\begin{tikzpicture}[baseline = -.5mm]
	\draw[thin,darkblue,<-] (-0.28,-.3) to (0.28,.4);
	\draw[-,line width=4pt,white] (0.28,-.3) to (-0.28,.4);
	\draw[->,thin,darkblue] (0.28,-.3) to (-0.28,.4);
\end{tikzpicture}
}&\mapsto 
-\mathord{
\begin{tikzpicture}[baseline = -.5mm]
	\draw[->,thin,darkblue] (0.28,-.3) to (-0.28,.4);
	\draw[line width=4pt,white,-] (-0.28,-.3) to (0.28,.4);
	\draw[thin,darkblue,<-] (-0.28,-.3) to (0.28,.4);
\end{tikzpicture}
}\:,\\
\mathord{
\begin{tikzpicture}[baseline = -.5mm]
	\draw[thin,darkblue,->] (-0.28,-.3) to (0.28,.4);
	\draw[-,line width=4pt,white] (0.28,-.3) to (-0.28,.4);
	\draw[->,thin,darkblue] (0.28,-.3) to (-0.28,.4);
\end{tikzpicture}
}&\mapsto 
-\mathord{
\begin{tikzpicture}[baseline = -.5mm]
	\draw[<-,thin,darkblue] (0.28,-.3) to (-0.28,.4);
	\draw[line width=4pt,white,-] (-0.28,-.3) to (0.28,.4);
	\draw[thin,darkblue,<-] (-0.28,-.3) to (0.28,.4);
\end{tikzpicture}
}\:,&
\mathord{
\begin{tikzpicture}[baseline = -.5mm]
	\draw[<-,thin,darkblue] (0.28,-.3) to (-0.28,.4);
	\draw[line width=4pt,white,-] (-0.28,-.3) to (0.28,.4);
	\draw[thin,darkblue,->] (-0.28,-.3) to (0.28,.4);
\end{tikzpicture}
}&\mapsto 
-\mathord{
\begin{tikzpicture}[baseline = -.5mm]
	\draw[thin,darkblue,->] (-0.28,-.3) to (0.28,.4);
	\draw[-,line width=4pt,white] (0.28,-.3) to (-0.28,.4);
	\draw[<-,thin,darkblue] (0.28,-.3) to (-0.28,.4);
\end{tikzpicture}
}\:,
&
\mathord{
\begin{tikzpicture}[baseline = -.5mm]
	\draw[<-,thin,darkblue] (0.28,-.3) to (-0.28,.4);
	\draw[line width=4pt,white,-] (-0.28,-.3) to (0.28,.4);
	\draw[thin,darkblue,<-] (-0.28,-.3) to (0.28,.4);
\end{tikzpicture}
}&\mapsto 
-\mathord{
\begin{tikzpicture}[baseline = -.5mm]
	\draw[->,thin,darkblue] (0.28,-.3) to (-0.28,.4);
	\draw[line width=4pt,white,->] (-0.28,-.3) to (0.28,.4);
	\draw[thin,darkblue,->] (-0.28,-.3) to (0.28,.4);
\end{tikzpicture}
}\:,&
\mathord{
\begin{tikzpicture}[baseline = -.5mm]
	\draw[->,thin,darkblue] (0.28,-.3) to (-0.28,.4);
	\draw[line width=4pt,white,-] (-0.28,-.3) to (0.28,.4);
	\draw[thin,darkblue,<-] (-0.28,-.3) to (0.28,.4);
\end{tikzpicture}
}&\mapsto 
-\mathord{
\begin{tikzpicture}[baseline = -.5mm]
	\draw[thin,darkblue,<-] (-0.28,-.3) to (0.28,.4);
	\draw[-,line width=4pt,white] (0.28,-.3) to (-0.28,.4);
	\draw[->,thin,darkblue] (0.28,-.3) to (-0.28,.4);
\end{tikzpicture}
}\:,
\\
\mathord{
\begin{tikzpicture}[baseline = 1mm]
	\draw[-,thin,darkblue] (0.4,0.4) to[out=-90, in=0] (0.1,0);
	\draw[->,thin,darkblue] (0.1,0) to[out = 180, in = -90] (-0.2,0.4);
      \node at (0.12,-0.2) {$\color{darkblue}\scriptstyle{a}$};
      \node at (0.12,0.01) {$\diamond$};
\end{tikzpicture}
}
&\mapsto
\mathord{
\begin{tikzpicture}[baseline = 1mm]
	\draw[-,thin,darkblue] (0.4,0) to[out=90, in=0] (0.1,0.4);
	\draw[->,thin,darkblue] (0.1,0.4) to[out = 180, in = 90] (-0.2,0);
      \node at (0.12,0.6) {$\color{darkblue}\scriptstyle{a}$};
      \node at (0.12,0.4) {$\diamond$};
\end{tikzpicture}
}\:,&
\mathord{
\begin{tikzpicture}[baseline = 1mm]
	\draw[-,thin,darkblue] (0.4,0) to[out=90, in=0] (0.1,0.4);
	\draw[->,thin,darkblue] (0.1,0.4) to[out = 180, in = 90] (-0.2,0);
      \node at (0.12,0.6) {$\color{darkblue}\scriptstyle{a}$};
      \node at (0.12,0.4) {$\diamond$};
\end{tikzpicture}
}
&\mapsto
\mathord{
\begin{tikzpicture}[baseline = 1mm]
	\draw[-,thin,darkblue] (0.4,0.4) to[out=-90, in=0] (0.1,0);
	\draw[->,thin,darkblue] (0.1,0) to[out = 180, in = -90] (-0.2,0.4);
      \node at (0.12,-0.2) {$\color{darkblue}\scriptstyle{a}$};
      \node at (0.12,0.01) {$\diamond$};
\end{tikzpicture}
}\:,&
\mathord{
\begin{tikzpicture}[baseline = 1mm]
	\draw[-,thin,darkblue] (0.4,0.4) to[out=-90, in=0] (0.1,0);
	\draw[->,thin,darkblue] (0.1,0) to[out = 180, in = -90] (-0.2,0.4);
      \node at (0.12,-0.2) {$\color{darkblue}\scriptstyle{a}$};
      \node at (0.12,0.01) {$\heart$};
\end{tikzpicture}
}&\mapsto
\mathord{
\begin{tikzpicture}[baseline = 1mm]
	\draw[-,thin,darkblue] (0.4,0) to[out=90, in=0] (0.1,0.4);
	\draw[->,thin,darkblue] (0.1,0.4) to[out = 180, in = 90] (-0.2,0);
      \node at (0.12,0.6) {$\color{darkblue}\scriptstyle{a}$};
      \node at (0.12,0.4) {$\heart$};
\end{tikzpicture}
}\:,
&
\mathord{
\begin{tikzpicture}[baseline = 1mm]
	\draw[-,thin,darkblue] (0.4,0) to[out=90, in=0] (0.1,0.4);
	\draw[->,thin,darkblue] (0.1,0.4) to[out = 180, in = 90] (-0.2,0);
      \node at (0.12,0.6) {$\color{darkblue}\scriptstyle{a}$};
      \node at (0.12,0.4) {$\heart$};
\end{tikzpicture}
}&\mapsto
\mathord{
\begin{tikzpicture}[baseline = 1mm]
	\draw[-,thin,darkblue] (0.4,0.4) to[out=-90, in=0] (0.1,0);
	\draw[->,thin,darkblue] (0.1,0) to[out = 180, in = -90] (-0.2,0.4);
      \node at (0.12,-0.2) {$\color{darkblue}\scriptstyle{a}$};
      \node at (0.12,0.01) {$\heart$};
\end{tikzpicture}
}\:,
\\
\mathord{
\begin{tikzpicture}[baseline = 1mm]
	\draw[-,thin,darkblue] (0.4,0) to[out=90, in=0] (0.1,0.4);
	\draw[->,thin,darkblue] (0.1,0.4) to[out = 180, in = 90] (-0.2,0);
\end{tikzpicture}
}\:&\mapsto \mathord{
-\begin{tikzpicture}[baseline = 1mm]
	\draw[-,thin,darkblue] (0.4,0.4) to[out=-90, in=0] (0.1,0);
	\draw[->,thin,darkblue] (0.1,0) to[out = 180, in = -90] (-0.2,0.4);
\end{tikzpicture}
},\:&
\mathord{
\begin{tikzpicture}[baseline = 1mm]
	\draw[-,thin,darkblue] (0.4,0.4) to[out=-90, in=0] (0.1,0);
	\draw[->,thin,darkblue] (0.1,0) to[out = 180, in = -90] (-0.2,0.4);
\end{tikzpicture}
}\:&\mapsto
-\mathord{
\begin{tikzpicture}[baseline = 1mm]
	\draw[-,thin,darkblue] (0.4,0) to[out=90, in=0] (0.1,0.4);
	\draw[->,thin,darkblue] (0.1,0.4) to[out = 180, in = 90] (-0.2,0);
\end{tikzpicture}
}\:,&
\mathord{\begin{tikzpicture}[baseline = .8mm]
  \draw[-,thin,darkblue] (0.2,0.2) to[out=90,in=0] (0,.4);
  \draw[->,thin,darkblue] (0,0.4) to[out=180,in=90] (-.2,0.2);
\draw[-,thin,darkblue] (-.2,0.2) to[out=-90,in=180] (0,0);
  \draw[-,thin,darkblue] (0,0) to[out=0,in=-90] (0.2,0.2);
      \node at (0,0.2) {$\color{darkblue}{\pm}$};
      \node at (0.3,0.2) {$\color{darkblue}\scriptstyle{a}$};
 \end{tikzpicture}
}
&\mapsto
-\:\mathord{\begin{tikzpicture}[baseline = .8mm]
  \draw[<-,thin,darkblue] (0.2,0.2) to[out=90,in=0] (0,.4);
  \draw[-,thin,darkblue] (0,0.4) to[out=180,in=90] (-.2,0.2);
\draw[-,thin,darkblue] (-.2,0.2) to[out=-90,in=180] (0,0);
  \draw[-,thin,darkblue] (0,0) to[out=0,in=-90] (0.2,0.2);
      \node at (0,0.2) {$\color{darkblue}{\pm}$};
      \node at (-0.3,0.2) {$\color{darkblue}\scriptstyle{a}$};
 \end{tikzpicture}
}\:,&
\mathord{\begin{tikzpicture}[baseline = .8mm]
  \draw[<-,thin,darkblue] (0.2,0.2) to[out=90,in=0] (0,.4);
  \draw[-,thin,darkblue] (0,0.4) to[out=180,in=90] (-.2,0.2);
\draw[-,thin,darkblue] (-.2,0.2) to[out=-90,in=180] (0,0);
  \draw[-,thin,darkblue] (0,0) to[out=0,in=-90] (0.2,0.2);
      \node at (0,0.2) {$\color{darkblue}{\pm}$};
      \node at (-0.3,0.2) {$\color{darkblue}\scriptstyle{a}$};
 \end{tikzpicture}
}
&\mapsto
-\:\mathord{\begin{tikzpicture}[baseline = .8mm]
  \draw[-,thin,darkblue] (0.2,0.2) to[out=90,in=0] (0,.4);
  \draw[->,thin,darkblue] (0,0.4) to[out=180,in=90] (-.2,0.2);
\draw[-,thin,darkblue] (-.2,0.2) to[out=-90,in=180] (0,0);
  \draw[-,thin,darkblue] (0,0) to[out=0,in=-90] (0.2,0.2);
      \node at (0,0.2) {$\color{darkblue}{\pm}$};
      \node at (0.3,0.2) {$\color{darkblue}\scriptstyle{a}$};
 \end{tikzpicture}
}
\:.
\end{align*}
\end{theorem}

\begin{proof}
To avoid confusion, denote the category $\HEIS_k(z,t)$ from
Definition~\ref{def1} by $\HEIS_k^{\operatorname{old}}(z,t)$
and the one from Definition~\ref{def2} by
$\HEIS_k^{\operatorname{new}}(z,t)$.
The relations and other definitions for the category
$\HEIS^{\operatorname{new}}_k(z,t)$
in Definition~\ref{def2} 
and the ones for
$\HEIS^{\operatorname{old}}_{-k}(z,t^{-1})$
from Definition~\ref{def1} are related 
by reflecting all diagrams in a
horizontal plane and multiplying by $(-1)^{x+y}$, where $x$ is the
number of crossings and $y$ is the number of leftward cups and caps
(including leftward cups and caps in $(+)$- and $(-)$-bubbles but
not ones labelled by $\diamond$ or $\heart$).
It follows that there are mutually inverse isomorphisms
$$
\HEIS_{-k}^{\operatorname{old}}(z,t^{-1})
\:\substack{\Omega_-\\{\textstyle\rightleftarrows}\\\Omega_+}\:
\HEIS_{k}^{\operatorname{new}}(z,t)^{\operatorname{op}} 
$$
both defined in the same way as the functor $\Omega_k$ in the statement of the theorem.
Now we apply
Lemma~\ref{irrelevant}  and
Definition~\ref{def2} to construct a strict $\k$-linear monoidal
functor $$
\Theta_k:\HEIS_k^{\operatorname{new}}(z,t) \rightarrow
\HEIS_k^{\operatorname{old}}(z,t)
$$ which is the identity on diagrams.
This functor is an isomorphism because it has a two-sided inverse, namely,
$\Omega_+ \circ \Theta_{-k} \circ \Omega_-$.
Thus, using $\Theta_k$, we may identify $\HEIS_k^{\operatorname{new}}(z,t)$
and $\HEIS_k^{\operatorname{old}}(z,t)$. Finally,  $\Omega_k := \Omega_+$ gives
the required symmetry.
\end{proof}

In the remainder of the section, we record some further consequences
of the defining relations, thereby showing that $\HEIS_k(z,t)$ is
strictly pivotal.
The first lemma explains how dots slide past leftward cups, caps
and crossings. Its generalization to 
dots with arbitrary multiplicities $n \in
\Z$ may also be deduced using induction and the leftward skein
relation like in Lemma~\ref{dotslide}.

\begin{lemma}
The following relations hold:
\begin{align}\label{piv1}
\mathord{
\begin{tikzpicture}[baseline = 0]
	\draw[-,thin,darkblue] (0.4,0.4) to[out=-90, in=0] (0.1,-0.1);
	\draw[->,thin,darkblue] (0.1,-0.1) to[out = 180, in = -90] (-0.2,0.4);
      \node at (0.37,0.15) {$\dot$};
\end{tikzpicture}
}
&=
\mathord{
\begin{tikzpicture}[baseline = 0]
	\draw[-,thin,darkblue] (0.4,0.4) to[out=-90, in=0] (0.1,-0.1);
	\draw[->,thin,darkblue] (0.1,-0.1) to[out = 180, in = -90] (-0.2,0.4);
      \node at (-0.16,0.15) {$\dot$};
\end{tikzpicture}
}\:,
&
\mathord{
\begin{tikzpicture}[baseline = 1mm]
	\draw[-,thin,darkblue] (0.4,-0.1) to[out=90, in=0] (0.1,0.4);
	\draw[->,thin,darkblue] (0.1,0.4) to[out = 180, in = 90] (-0.2,-0.1);
      \node at (-0.14,0.25) {$\dot$};
\end{tikzpicture}
}&
=
\mathord{
\begin{tikzpicture}[baseline = 1mm]
	\draw[-,thin,darkblue] (0.4,-0.1) to[out=90, in=0] (0.1,0.4);
	\draw[->,thin,darkblue] (0.1,0.4) to[out = 180, in = 90] (-0.2,-0.1);
      \node at (0.32,0.25) {$\dot$};
\end{tikzpicture}
}\:,\\
\mathord{
\begin{tikzpicture}[baseline = -.5mm]
	\draw[->,thin,darkblue] (0.28,-.3) to (-0.28,.4);
      \node at (0.165,-0.15) {$\dot$};
	\draw[line width=4pt,white,-] (-0.28,-.3) to (0.28,.4);
	\draw[thin,darkblue,<-] (-0.28,-.3) to (0.28,.4);
\end{tikzpicture}
}&=\mathord{
\begin{tikzpicture}[baseline = -.5mm]
	\draw[thin,darkblue,<-] (-0.28,-.3) to (0.28,.4);
	\draw[-,line width=4pt,white] (0.28,-.3) to (-0.28,.4);
	\draw[->,thin,darkblue] (0.28,-.3) to (-0.28,.4);
      \node at (-0.14,0.23) {$\dot$};
\end{tikzpicture}
}\
\:,
&\mathord{
\begin{tikzpicture}[baseline = -.5mm]
	\draw[thin,darkblue,<-] (-0.28,-.3) to (0.28,.4);
	\draw[-,line width=4pt,white] (0.28,-.3) to (-0.28,.4);
	\draw[->,thin,darkblue] (0.28,-.3) to (-0.28,.4);
      \node at (0.145,0.23) {$\dot$};
\end{tikzpicture}
}&= 
\mathord{
\begin{tikzpicture}[baseline = -.5mm]
	\draw[->,thin,darkblue] (0.28,-.3) to (-0.28,.4);
	\draw[line width=4pt,white,-] (-0.28,-.3) to (0.28,.4);
	\draw[thin,darkblue,<-] (-0.28,-.3) to (0.28,.4);
      \node at (-0.16,-0.15) {$\dot$};
\end{tikzpicture}
}\:.\label{rr4}
\end{align}
\end{lemma}

Let $\Sym$ be the algebra of symmetric functions over $\k$. This is an
infinite rank polynomial
algebra with two sets of algebraically independent generators,
namely, the {\em elementary symmetric functions} $\e_1,\e_2,\dots$ and the 
{\em complete symmetric functions} $\h_1,\h_2,\dots$.
Adopting the convention that  
$\e_n = \h_n := \delta_{n,0}$ for $n \leq
0$, the elementary and complete symmetric functions are related by the
following well-known
identity \cite[(I.2.6)]{Mac}:
\begin{equation}\label{symid}
\sum_{r+s=n} (-1)^s \e_r \h_{s} = \delta_{n,0}.
\end{equation}
The following lemma, which we may refer to as the 
{\em infinite Grassmannian relation} (following Lauda), shows that
there is a well-defined homomorphism
\begin{equation}\label{beta}
\beta:\Sym\otimes \Sym \rightarrow \End_{\HEIS_k(z,t)}(\unit)
\end{equation}
such that
\begin{align}\label{new1}
\h_n \otimes 1 &\mapsto (-1)^{n-1}t z\:
\mathord{\begin{tikzpicture}[baseline = -1mm]
  \draw[<-,thin,darkblue] (0,0.2) to[out=180,in=90] (-.2,0);
  \draw[-,thin,darkblue] (0.2,0) to[out=90,in=0] (0,.2);
 \draw[-,thin,darkblue] (-.2,0) to[out=-90,in=180] (0,-0.2);
  \draw[-,thin,darkblue] (0,-0.2) to[out=0,in=-90] (0.2,0);
\node at (0,0) {$\color{darkblue}+$};
      \node at (-0.47,0) {$\color{darkblue}\scriptstyle n+k$};
\end{tikzpicture}
}\:,&
1\otimes \h_n &\mapsto (-1)^n t^{-1}z \:
\mathord{\begin{tikzpicture}[baseline = -1mm]
  \draw[<-,thin,darkblue] (0,0.2) to[out=180,in=90] (-.2,0);
  \draw[-,thin,darkblue] (0.2,0) to[out=90,in=0] (0,.2);
 \draw[-,thin,darkblue] (-.2,0) to[out=-90,in=180] (0,-0.2);
  \draw[-,thin,darkblue] (0,-0.2) to[out=0,in=-90] (0.2,0);
\node at (0,0) {$\color{darkblue}-$};
      \node at (-0.4,0) {$\color{darkblue}\scriptstyle -n$};
\end{tikzpicture}
}\:,\\\label{new2}
\e_n\otimes 1 &\mapsto t^{-1} z\:
\mathord{\begin{tikzpicture}[baseline = -1mm]
  \draw[-,thin,darkblue] (0,0.2) to[out=180,in=90] (-.2,0);
  \draw[->,thin,darkblue] (0.2,0) to[out=90,in=0] (0,.2);
 \draw[-,thin,darkblue] (-.2,0) to[out=-90,in=180] (0,-0.2);
  \draw[-,thin,darkblue] (0,-0.2) to[out=0,in=-90] (0.2,0);
\node at (0,0) {$\color{darkblue}+$};
      \node at (0.45,0) {$\color{darkblue}\scriptstyle n-k$};
\end{tikzpicture}
}\:,&
1\otimes \e_n &\mapsto -tz \:
\mathord{\begin{tikzpicture}[baseline = -1mm]
  \draw[-,thin,darkblue] (0,0.2) to[out=180,in=90] (-.2,0);
  \draw[->,thin,darkblue] (0.2,0) to[out=90,in=0] (0,.2);
 \draw[-,thin,darkblue] (-.2,0) to[out=-90,in=180] (0,-0.2);
  \draw[-,thin,darkblue] (0,-0.2) to[out=0,in=-90] (0.2,0);
\node at (0,0) {$\color{darkblue}-$};
      \node at (0.4,0) {$\color{darkblue}\scriptstyle -n$};
\end{tikzpicture}
}\:.
\end{align}
We will prove in Corollary \ref{cor:beta-iso} that $\beta$ is actually an {\em isomorphism}.

\begin{lemma}\label{infgrass}
For any $a \in \Z$, we have that
\begin{align}
\sum_{\substack{b,c\in\Z\\b+c = a}}
\mathord{
\begin{tikzpicture}[baseline = 1mm]
  \draw[-,thin,darkblue] (0.2,0.2) to[out=90,in=0] (0,.4);
  \draw[->,thin,darkblue] (0,0.4) to[out=180,in=90] (-.2,0.2);
\draw[-,thin,darkblue] (-.2,0.2) to[out=-90,in=180] (0,0);
  \draw[-,thin,darkblue] (0,0) to[out=0,in=-90] (0.2,0.2);
\node at (0,.2) {$\color{darkblue}+$};
\node at (.3,.2) {$\color{darkblue}\scriptstyle{b}$};
\end{tikzpicture}
}
\mathord{
\begin{tikzpicture}[baseline = 1mm]
  \draw[<-,thin,darkblue] (0.2,0.2) to[out=90,in=0] (0,.4);
  \draw[-,thin,darkblue] (0,0.4) to[out=180,in=90] (-.2,0.2);
\draw[-,thin,darkblue] (-.2,0.2) to[out=-90,in=180] (0,0);
  \draw[-,thin,darkblue] (0,0) to[out=0,in=-90] (0.2,0.2);
\node at (0,.2) {$\color{darkblue}+$};
\node at (-.3,.2) {$\color{darkblue}\scriptstyle{c}$};
\end{tikzpicture}
}
=\sum_{\substack{b,c\in\Z\\b+c = a}}
\mathord{
\begin{tikzpicture}[baseline = 1mm]
  \draw[-,thin,darkblue] (0.2,0.2) to[out=90,in=0] (0,.4);
  \draw[->,thin,darkblue] (0,0.4) to[out=180,in=90] (-.2,0.2);
\draw[-,thin,darkblue] (-.2,0.2) to[out=-90,in=180] (0,0);
  \draw[-,thin,darkblue] (0,0) to[out=0,in=-90] (0.2,0.2);
\node at (0,.2) {$\color{darkblue}-$};
\node at (.3,.2) {$\color{darkblue}\scriptstyle{b}$};
\end{tikzpicture}
}
\mathord{
\begin{tikzpicture}[baseline = 1mm]
  \draw[<-,thin,darkblue] (0.2,0.2) to[out=90,in=0] (0,.4);
  \draw[-,thin,darkblue] (0,0.4) to[out=180,in=90] (-.2,0.2);
\draw[-,thin,darkblue] (-.2,0.2) to[out=-90,in=180] (0,0);
  \draw[-,thin,darkblue] (0,0) to[out=0,in=-90] (0.2,0.2);
\node at (0,.2) {$\color{darkblue}-$};
\node at (-.3,.2) {$\color{darkblue}\scriptstyle{c}$};
\end{tikzpicture}
}&= -\delta_{a,0}\: z^{-2} 1_\unit.
\end{align}
Moreover:
\begin{align}
\label{tanks}
\mathord{\begin{tikzpicture}[baseline = .8mm]
  \draw[-,thin,darkblue] (0.2,0.2) to[out=90,in=0] (0,.4);
  \draw[->,thin,darkblue] (0,0.4) to[out=180,in=90] (-.2,0.2);
\draw[-,thin,darkblue] (-.2,0.2) to[out=-90,in=180] (0,0);
  \draw[-,thin,darkblue] (0,0) to[out=0,in=-90] (0.2,0.2);
      \node at (0,0.2) {$\color{darkblue}{+}$};
      \node at (0.3,0.2) {$\color{darkblue}\scriptstyle{a}$};
 \end{tikzpicture}
}
&=\delta_{a,-k} tz^{-1} 1_\unit\quad\text{if $a\leq -k$,}&
\mathord{\begin{tikzpicture}[baseline = .8mm]
  \draw[<-,thin,darkblue] (0.2,0.2) to[out=90,in=0] (0,.4);
  \draw[-,thin,darkblue] (0,0.4) to[out=180,in=90] (-.2,0.2);
\draw[-,thin,darkblue] (-.2,0.2) to[out=-90,in=180] (0,0);
  \draw[-,thin,darkblue] (0,0) to[out=0,in=-90] (0.2,0.2);
      \node at (0,0.2) {$\color{darkblue}{+}$};
      \node at (-0.3,0.2) {$\color{darkblue}\scriptstyle{a}$};
 \end{tikzpicture}
}
&=-\delta_{a,k} t^{-1}z^{-1} 1_\unit\quad\text{if $a \leq k$,}\\\label{tanks2}
\mathord{\begin{tikzpicture}[baseline = .8mm]
  \draw[<-,thin,darkblue] (0.2,0.2) to[out=90,in=0] (0,.4);
  \draw[-,thin,darkblue] (0,0.4) to[out=180,in=90] (-.2,0.2);
\draw[-,thin,darkblue] (-.2,0.2) to[out=-90,in=180] (0,0);
  \draw[-,thin,darkblue] (0,0) to[out=0,in=-90] (0.2,0.2);
      \node at (0,0.2) {$\color{darkblue}{-}$};
      \node at (-0.3,0.2) {$\color{darkblue}\scriptstyle{a}$};
 \end{tikzpicture}
}
&=\delta_{a,0} tz^{-1} 1_\unit\:\,\quad\text{if $a \geq 0$,}&
\mathord{\begin{tikzpicture}[baseline = .8mm]
  \draw[-,thin,darkblue] (0.2,0.2) to[out=90,in=0] (0,.4);
  \draw[->,thin,darkblue] (0,0.4) to[out=180,in=90] (-.2,0.2);
\draw[-,thin,darkblue] (-.2,0.2) to[out=-90,in=180] (0,0);
  \draw[-,thin,darkblue] (0,0) to[out=0,in=-90] (0.2,0.2);
      \node at (0,0.2) {$\color{darkblue}{-}$};
      \node at (0.3,0.2) {$\color{darkblue}\scriptstyle{a}$};
 \end{tikzpicture}
}
&=
-\delta_{a,0} t^{-1}z^{-1} 1_\unit\quad\text{if $a \geq 0$.}
\end{align}
\end{lemma}

\begin{corollary}
For an indeterminate $w$, we have that 
\begin{equation}\label{igproper}
\anticlockplus\,(w)\; \clockplus\,(w)=
\anticlockminus\,(w)\; \clockminus\,(w) = 
1_\unit,
\end{equation}
where
\begin{align}\label{fourofem}
\anticlockplus\,(w) &:= t^{-1}z\sum_{n\in\Z}
\mathord{
\begin{tikzpicture}[baseline = 1.25mm]
  \draw[-,thin] (0,0.4) to[out=180,in=90] (-.2,0.2);
  \draw[->,thin] (0.2,0.2) to[out=90,in=0] (0,.4);
 \draw[-,thin] (-.2,0.2) to[out=-90,in=180] (0,0);
  \draw[-,thin] (0,0) to[out=0,in=-90] (0.2,0.2);
   \node at (0,.21) {$+$};
   \node at (0.33,0.2) {$\scriptstyle{n}$};
\end{tikzpicture}
}\: w^{-n}
\in w^k 1_\unit + w^{k-1} \End_{\HEIS_k(z,t)}(\unit)\llbracket
w^{-1} \rrbracket
,\\
\clockplus\,(w)&:= -tz\sum_{n\in\Z}
\mathord{
\begin{tikzpicture}[baseline = 1.25mm]
  \draw[<-,thin] (0,0.4) to[out=180,in=90] (-.2,0.2);
  \draw[-,thin] (0.2,0.2) to[out=90,in=0] (0,.4);
 \draw[-,thin] (-.2,0.2) to[out=-90,in=180] (0,0);
  \draw[-,thin] (0,0) to[out=0,in=-90] (0.2,0.2);
   \node at (0,.21) {$+$};
   \node at (-0.33,0.2) {$\scriptstyle{n}$};
\end{tikzpicture}
} \:w^{-n} \in w^{-k}1_\unit + 
w^{-k-1}\End_{\HEIS_k(z,t)}(\unit)\llbracket
w^{-1} \rrbracket, \label{igp}
\\
\anticlockminus\,(w) &:= -tz\sum_{n\in\Z}
\mathord{
\begin{tikzpicture}[baseline = 1.25mm]
  \draw[-,thin] (0,0.4) to[out=180,in=90] (-.2,0.2);
  \draw[->,thin] (0.2,0.2) to[out=90,in=0] (0,.4);
 \draw[-,thin] (-.2,0.2) to[out=-90,in=180] (0,0);
  \draw[-,thin] (0,0) to[out=0,in=-90] (0.2,0.2);
   \node at (0,.21) {$-$};
   \node at (0.33,0.2) {$\scriptstyle{n}$};
\end{tikzpicture}
}\: w^{-n}\in 1_\unit+ w\End_{\HEIS_k(z,t)}(\unit)\llbracket
w \rrbracket,\\
\clockminus\,(w) &:= t^{-1}z\sum_{n\in\Z}
\mathord{
\begin{tikzpicture}[baseline = 1.25mm]
  \draw[<-,thin] (0,0.4) to[out=180,in=90] (-.2,0.2);
  \draw[-,thin] (0.2,0.2) to[out=90,in=0] (0,.4);
 \draw[-,thin] (-.2,0.2) to[out=-90,in=180] (0,0);
  \draw[-,thin] (0,0) to[out=0,in=-90] (0.2,0.2);
   \node at (0,.21) {$-$};
   \node at (-0.33,0.2) {$\scriptstyle{n}$};
\end{tikzpicture}
} \:w^{-n}\in 1_\unit + w\End_{\HEIS_k(z,t)}(\unit)\llbracket
w \rrbracket.\label{igm}
\end{align}
\end{corollary}

Using the next relations plus (\ref{nakano1}) and (\ref{nakano2}), the leftward cups and caps decorated by
$\diamond$ or $\heart$ can be eliminated from any diagram.

\begin{lemma}\label{redundant}
The following relations hold:
\begin{align}
\mathord{
\begin{tikzpicture}[baseline = 1mm]
	\draw[-,thin,darkblue] (0.4,0.4) to[out=-90, in=0] (0.1,0);
	\draw[->,thin,darkblue] (0.1,0) to[out = 180, in = -90] (-0.2,0.4);
      \node at (0.12,-0.2) {$\color{darkblue}\scriptstyle{a}$};
      \node at (0.12,0.01) {$\diamond$};
\end{tikzpicture}
}
&=-z^2 \sum_{b \geq 1}
\mathord{
\begin{tikzpicture}[baseline = 1mm]
	\draw[-,thin,darkblue] (0.4,0.4) to[out=-90, in=0] (0.1,0);
	\draw[->,thin,darkblue] (0.1,0) to[out = 180, in = -90] (-0.2,0.4);
      \node at (-0.33,0.18) {$\color{darkblue}\scriptstyle{b}$};
      \node at (-0.17,0.18) {$\dot$};
\end{tikzpicture}
}\quad\mathord{\begin{tikzpicture}[baseline = .8mm]
  \draw[-,thin,darkblue] (0.2,0.2) to[out=90,in=0] (0,.4);
  \draw[->,thin,darkblue] (0,0.4) to[out=180,in=90] (-.2,0.2);
\draw[-,thin,darkblue] (-.2,0.2) to[out=-90,in=180] (0,0);
  \draw[-,thin,darkblue] (0,0) to[out=0,in=-90] (0.2,0.2);
      \node at (0,0.2) {$\color{darkblue}{+}$};
      \node at (0.51,0.2) {$\color{darkblue}\scriptstyle{-a-b}$};
 \end{tikzpicture}
}&&\text{if $0 \leq a < k$,}\\\label{nakano5}
\mathord{
\begin{tikzpicture}[baseline = 1mm]
	\draw[-,thin,darkblue] (0.4,0) to[out=90, in=0] (0.1,0.4);
	\draw[->,thin,darkblue] (0.1,0.4) to[out = 180, in = 90] (-0.2,0);
      \node at (0.12,0.6) {$\color{darkblue}\scriptstyle{a}$};
      \node at (0.12,0.4) {$\diamond$};
\end{tikzpicture}
}
&=-z^2 \sum_{b \geq 1}
\:\mathord{
\begin{tikzpicture}[baseline = 1mm]
	\draw[-,thin,darkblue] (0.4,0) to[out=90, in=0] (0.1,0.4);
	\draw[->,thin,darkblue] (0.1,0.4) to[out = 180, in = 90] (-0.2,0);
      \node at (0.48,0.25) {$\color{darkblue}\scriptstyle{b}$};
      \node at (0.32,0.25) {$\dot$};
\end{tikzpicture}}
\,\mathord{\begin{tikzpicture}[baseline = .8mm]
  \draw[<-,thin,darkblue] (0.2,0.2) to[out=90,in=0] (0,.4);
  \draw[-,thin,darkblue] (0,0.4) to[out=180,in=90] (-.2,0.2);
\draw[-,thin,darkblue] (-.2,0.2) to[out=-90,in=180] (0,0);
  \draw[-,thin,darkblue] (0,0) to[out=0,in=-90] (0.2,0.2);
      \node at (0,0.2) {$\color{darkblue}{+}$};
      \node at (-0.51,0.2) {$\color{darkblue}\scriptstyle{-a-b}$};
 \end{tikzpicture}
}&&\text{if $0 \leq a < -k$.}
\end{align}
\end{lemma}

The next lemma shows that $\down$ is left dual to $\up$ (as well as being
right dual by the original construction).
Thus, the monoidal category $\HEIS_k(z,t)$ is rigid.

\begin{lemma}\label{rigiditylemma}
The following relations hold:
\begin{align}\label{adjfinal}
\mathord{
\begin{tikzpicture}[baseline = 0]
  \draw[-,thin,darkblue] (0.3,0) to (0.3,-.4);
	\draw[-,thin,darkblue] (0.3,0) to[out=90, in=0] (0.1,0.4);
	\draw[-,thin,darkblue] (0.1,0.4) to[out = 180, in = 90] (-0.1,0);
	\draw[-,thin,darkblue] (-0.1,0) to[out=-90, in=0] (-0.3,-0.4);
	\draw[-,thin,darkblue] (-0.3,-0.4) to[out = 180, in =-90] (-0.5,0);
  \draw[->,thin,darkblue] (-0.5,0) to (-0.5,.4);
\end{tikzpicture}
}
&=
\mathord{\begin{tikzpicture}[baseline=0]
  \draw[->,thin,darkblue] (0,-0.4) to (0,.4);
\end{tikzpicture}
}\:,
&\mathord{
\begin{tikzpicture}[baseline = 0]
  \draw[-,thin,darkblue] (0.3,0) to (0.3,.4);
	\draw[-,thin,darkblue] (0.3,0) to[out=-90, in=0] (0.1,-0.4);
	\draw[-,thin,darkblue] (0.1,-0.4) to[out = 180, in = -90] (-0.1,0);
	\draw[-,thin,darkblue] (-0.1,0) to[out=90, in=0] (-0.3,0.4);
	\draw[-,thin,darkblue] (-0.3,0.4) to[out = 180, in =90] (-0.5,0);
  \draw[->,thin,darkblue] (-0.5,0) to (-0.5,-.4);
\end{tikzpicture}
}
&=
\mathord{\begin{tikzpicture}[baseline=0]
  \draw[<-,thin,darkblue] (0,-0.4) to (0,.4);
\end{tikzpicture}
}\:.
\end{align}
\end{lemma}

The final lemma together with (\ref{piv1}) implies that
$\HEIS_k(z,t)$ is strictly pivotal, with duality functor
\begin{equation}\label{spiv}
*:\HEIS_k(z,t)\stackrel{\sim}{\rightarrow}
\left(\HEIS_k(z,t)^{\operatorname{op}}\right)^{\operatorname{rev}}
\end{equation}
defined on morphisms by rotating diagrams through $180^\circ$.

\begin{lemma}
The following relations hold:
\begin{align}\label{cough3}
\mathord{
\begin{tikzpicture}[baseline = 0]
\draw[->,thin,darkblue](.6,-.3) to (.1,.4);
	\draw[-,line width=4pt,white] (0.6,0.4) to[out=-140, in=0] (0.1,-0.1);
	\draw[-,thin,darkblue] (0.6,0.4) to[out=-140, in=0] (0.1,-0.1);
	\draw[->,thin,darkblue] (0.1,-0.1) to[out = 180, in = -90] (-0.2,0.4);
\end{tikzpicture}
}&=
\mathord{
\begin{tikzpicture}[baseline = 0]
\draw[->,thin,darkblue](-.5,-.3) to (0,.4);
	\draw[-,thin,darkblue] (0.3,0.4) to[out=-90, in=0] (0,-0.1);
	\draw[-,line width=4pt,white] (0,-0.1) to[out = 180, in = -40] (-0.5,0.4);
	\draw[->,thin,darkblue] (0,-0.1) to[out = 180, in = -40] (-0.5,0.4);
\end{tikzpicture}
}\:,&
\mathord{
\begin{tikzpicture}[baseline = 0]
	\draw[-,thin,darkblue] (0.6,0.4) to[out=-140, in=0] (0.1,-0.1);
	\draw[->,thin,darkblue] (0.1,-0.1) to[out = 180, in = -90] (-0.2,0.4);
\draw[-,line width=4pt,white](.6,-.3) to (.1,.4);
\draw[<-,thin,darkblue](.6,-.3) to (.1,.4);
\end{tikzpicture}
}&=
\mathord{
\begin{tikzpicture}[baseline = 0]
	\draw[-,thin,darkblue] (0.3,0.4) to[out=-90, in=0] (0,-0.1);
	\draw[->,thin,darkblue] (0,-0.1) to[out = 180, in = -40] (-0.5,0.4);
\draw[-,line width=4pt,white](-.5,-.3) to (0,.4);
\draw[<-,thin,darkblue](-.5,-.3) to (0,.4);
\end{tikzpicture}
}\:,&
\mathord{
\begin{tikzpicture}[baseline = 0]
	\draw[-,thin,darkblue] (0.6,0.4) to[out=-140, in=0] (0.1,-0.1);
	\draw[->,thin,darkblue] (0.1,-0.1) to[out = 180, in = -90] (-0.2,0.4);
\draw[-,line width=4pt,white](.6,-.3) to (.1,.4);
\draw[->,thin,darkblue](.6,-.3) to (.1,.4);
\end{tikzpicture}
}&=
\mathord{
\begin{tikzpicture}[baseline = 0]
	\draw[-,thin,darkblue] (0.3,0.4) to[out=-90, in=0] (0,-0.1);
	\draw[->,thin,darkblue] (0,-0.1) to[out = 180, in = -40] (-0.5,0.4);
\draw[-,line width=4pt,white](-.5,-.3) to (0,.4);
\draw[->,thin,darkblue](-.5,-.3) to (0,.4);
\end{tikzpicture}
}\:,&
\mathord{
\begin{tikzpicture}[baseline = 0]
\draw[<-,thin,darkblue](.6,-.3) to (.1,.4);
	\draw[-,line width=4pt,white] (0.6,0.4) to[out=-140, in=0] (0.1,-0.1);
	\draw[-,thin,darkblue] (0.6,0.4) to[out=-140, in=0] (0.1,-0.1);
	\draw[->,thin,darkblue] (0.1,-0.1) to[out = 180, in = -90] (-0.2,0.4);
\end{tikzpicture}
}&=
\mathord{
\begin{tikzpicture}[baseline = 0]
\draw[<-,thin,darkblue](-.5,-.3) to (0,.4);
	\draw[-,thin,darkblue] (0.3,0.4) to[out=-90, in=0] (0,-0.1);
	\draw[-,line width=4pt,white] (0,-0.1) to[out = 180, in = -40] (-0.5,0.4);
	\draw[->,thin,darkblue] (0,-0.1) to[out = 180, in = -40] (-0.5,0.4);
\end{tikzpicture}
}\:,
\\\label{cough4}
\mathord{
\begin{tikzpicture}[baseline = 0]
\draw[->,thin,darkblue](.6,.4) to (.1,-.3);
	\draw[-,line width=4pt,white] (0.6,-0.3) to[out=140, in=0] (0.1,0.2);
	\draw[-,thin,darkblue] (0.6,-0.3) to[out=140, in=0] (0.1,0.2);
	\draw[->,thin,darkblue] (0.1,0.2) to[out = -180, in = 90] (-0.2,-0.3);
\end{tikzpicture}
}&=\mathord{
\begin{tikzpicture}[baseline = 0]
\draw[->,thin,darkblue](-.5,.4) to (0,-.3);
	\draw[-,thin,darkblue] (0.3,-0.3) to[out=90, in=0] (0,0.2);
	\draw[-,line width=4pt,white] (0,0.2) to[out = -180, in = 40] (-0.5,-0.3);
	\draw[->,thin,darkblue] (0,0.2) to[out = -180, in = 40] (-0.5,-0.3);
\end{tikzpicture}
}\:,&
\mathord{
\begin{tikzpicture}[baseline = 0]
	\draw[-,thin,darkblue] (0.6,-0.3) to[out=140, in=0] (0.1,0.2);
	\draw[->,thin,darkblue] (0.1,0.2) to[out = -180, in = 90] (-0.2,-0.3);
\draw[-,line width=4pt,white](.6,.4) to (.1,-.3);
\draw[<-,thin,darkblue](.6,.4) to (.1,-.3);
\end{tikzpicture}
}&=
\mathord{
\begin{tikzpicture}[baseline = 0]
	\draw[-,thin,darkblue] (0.3,-0.3) to[out=90, in=0] (0,0.2);
	\draw[->,thin,darkblue] (0,0.2) to[out = -180, in = 40] (-0.5,-0.3);
\draw[-,line width=4pt,white](-.5,.4) to (0,-.3);
\draw[<-,thin,darkblue](-.5,.4) to (0,-.3);
\end{tikzpicture}
}\:,&
\mathord{
\begin{tikzpicture}[baseline = 0]
	\draw[-,thin,darkblue] (0.3,-0.3) to[out=90, in=0] (0,0.2);
	\draw[->,thin,darkblue] (0,0.2) to[out = -180, in = 40] (-0.5,-0.3);
\draw[-,line width=4pt,white](-.5,.4) to (0,-.3);
\draw[->,thin,darkblue](-.5,.4) to (0,-.3);
\end{tikzpicture}
}&=
\mathord{
\begin{tikzpicture}[baseline = 0]
	\draw[-,thin,darkblue] (0.6,-0.3) to[out=140, in=0] (0.1,0.2);
	\draw[->,thin,darkblue] (0.1,0.2) to[out = -180, in = 90] (-0.2,-0.3);
\draw[-,line width=4pt,white](.6,.4) to (.1,-.3);
\draw[->,thin,darkblue](.6,.4) to (.1,-.3);
\end{tikzpicture}
}\:,&
\mathord{
\begin{tikzpicture}[baseline = 0]
\draw[<-,thin,darkblue](-.5,.4) to (0,-.3);
	\draw[-,thin,darkblue] (0.3,-0.3) to[out=90, in=0] (0,0.2);
	\draw[-,line width=4pt,white] (0,0.2) to[out = -180, in = 40] (-0.5,-0.3);
	\draw[->,thin,darkblue] (0,0.2) to[out = -180, in = 40] (-0.5,-0.3);
\end{tikzpicture}
}&=\mathord{
\begin{tikzpicture}[baseline = 0]
\draw[<-,thin,darkblue](.6,.4) to (.1,-.3);
	\draw[-,line width=4pt,white] (0.6,-0.3) to[out=140, in=0] (0.1,0.2);
	\draw[-,thin,darkblue] (0.6,-0.3) to[out=140, in=0] (0.1,0.2);
	\draw[->,thin,darkblue] (0.1,0.2) to[out = -180, in = 90] (-0.2,-0.3);
\end{tikzpicture}
}\:.
\end{align}
\end{lemma}

\section{Third approach}\label{third}

Now we have enough relations in hand to formulate our third
presentation for $\HEIS_k(z,t)$.
This presentation does not involve any leftward cups or caps
decorated by $\diamond$ or $\heart$; Lemma~\ref{redundant} showed
already that these are redundant as generators.

\begin{definition}\label{def3}
The {\em quantum Heisenberg category} $\HEIS_k(z,t)$
is the strict $\k$-linear monoidal category obtained from $\AH(z)$
by adjoining a right dual $\down$ to $\up$ as explained in the
introduction, plus two more generating morphisms
 $\:\mathord{

}\:$ subject to the following additional relations:
\begin{align}\label{pos}
\mathord{
%
}= (\delta_{a,0} t z^{-1} - \delta_{a,k} t^{-1}z^{-1}) 1_\unit
\quad\:\text{if $k \leq a \leq 0$.}
\end{align}
Here, we have used the leftward crossings which are defined in this
approach by
\begin{align}\label{cold}
\mathord{
%
}\:,
\end{align}
and the $(+)$-bubbles which are defined for $a \leq k$ or $a \leq -k$, respectively, by
\begin{align}\label{d1}
\mathord{
%
}\right)_{i,j=1,\dots,a},\label{d2}
\end{align}
interpreting the determinants as $\delta_{a,0}$ in case $a \leq 0$.
Finally, define the $(+)$-bubbles with label $a > 0$ to be the usual bubbles
with $a$ dots as in (\ref{fake2}), then define the $(-)$-bubbles for all
$a \in \Z$ so that
(\ref{fake1}) holds. 
\end{definition}

Before proving the equivalence of this definition with the earlier
ones, we make some remarks about the relations (\ref{pos})--(\ref{d2}).
If $k \leq 1$, the relation (\ref{pos}) is equivalent to
\begin{align}\label{posalt}
\mathord{
%
}.
\end{align}
This follows immediately from the definition of the $(+)$-bubbles from
(\ref{d1}).
Similarly, when $k \geq -1$, the relation (\ref{neg}) is equivalent to 
\begin{align}
\mathord{
%
}.\label{negalt}
\end{align}
Here are some other useful consequences of these relations:
\begin{align}
\mathord{
%
}&\text{if $k \leq 0$.}
\label{r1alt}
\end{align}
These follow from (\ref{curls})--(\ref{morecurls})
on expanding the definitions of the sideways crossings.
Then, using (\ref{r1alt}) and the leftward skein relation to
convert the negative crossings in (\ref{posalt}) to positive ones, 
relation (\ref{posalt}) can be further simplified 
in case that 
$k < 0$: it is equivalent to
\begin{equation}\label{posalter}
\mathord{
%
}\:.
\end{equation}
Similarly, (\ref{negalt}) 
is equivalent to the following when $k > 0$:
\begin{equation}\label{negalter}
\mathord{
%
}\:.
\end{equation}
Finally, when $k=0$, the relations (\ref{posalt})--(\ref{negalt})
together are equivalent to the single assertion 
\begin{equation}\label{bothalt}
\mathord{
%
}\right)^{-1},
\end{equation}
i.e., both of the relations from (\ref{lunch}).

\begin{theorem}\label{toronto}
The category $\HEIS_k(z,t)$ defined by Definition~\ref{def3} is the same
as the one from Definitions~\ref{def1} and \ref{def2},
with all morphisms introduced in the third definition being the
same as the ones from before.
\end{theorem}

\begin{proof}
To avoid confusion in the proof, we denote the category from
the equivalent Definitions~\ref{def1} and \ref{def2} by
$\HEIS_k^{\operatorname{old}}(z,t)$, and the one from
Definition~\ref{def3}
by $\HEIS_k^{\operatorname{new}}(z,t)$. 
From the evident symmetry in the relations (\ref{pos})--(\ref{d2}), 
it follows that there is an isomorphism
$$
\Omega_k:\HEIS^{\operatorname{new}}_k(z,t) \rightarrow \HEIS^{\operatorname{new}}_{-k}(z,t^{-1})^{\operatorname{op}}
$$
which reflects diagrams in a horizontal plane and multiplies by
$(-1)^{x+y}$
where $x$ is the number of crossings and $y$ is the number of
leftward cups and caps. Combining this with (\ref{om}), we are reduced to
proving the theorem under the assumption that $k \leq 0$.

We first check that all of
the defining relations (\ref{pos})--(\ref{d2}) of
$\HEIS_k^{\operatorname{new}}(z,t)$ are
satisfied in $\HEIS_k^{\operatorname{old}}(z,t)$, so that there is a
strict $\k$-linear monoidal functor
$$
\Theta:\HEIS_k^{\operatorname{new}}(z,t) \rightarrow
\HEIS_k^{\operatorname{old}}(z,t)
$$
which is the identity on diagrams.
For this, note to start with that (\ref{cold}) holds in
$\HEIS_k^{\operatorname{old}}(z,t)$ as we have shown that the latter
category is strictly pivotal.
The relation (\ref{d1}) is almost trivial when $k \leq 0$ and holds
thanks to (\ref{tanks}).
For (\ref{d2}), the identity holds if $a-k \leq 0$ due again to (\ref{tanks}), so assume that  $a-k > 0$.
Then the desired identity is the image under the
homomorphism $\beta$ from (\ref{beta})
of the identity
$$
(-1)^{a-k-1} t^{-1}z^{-1} \h_{a-k} \otimes
1
= -z^{a-k-1} t^{-a+k-1} \det \left(-tz^{-1}\e_{i-j+1}\otimes 1\right)_{i,j=1,\dots,a-k}
$$
in $\Sym\otimes\Sym$. This follows from the well-known identity
$\h_{n} = \det \left(\e_{i-j+1}\right)_{i,j=1,\dots,n}$; see \cite[Exercise
I.2.8]{Mac}.
It remains to check the relations
(\ref{pos})--(\ref{morecurls}). 
For (\ref{pos})--(\ref{neg}) when $k=0$, we just need to check the
equivalent form (\ref{bothalt}), which follows by (\ref{lunch}).
For (\ref{pos}) when $k < 0$, we check the equivalent form
(\ref{posalter}),
which holds due to the second
relation from (\ref{tea3}).
For (\ref{neg}) when $k < 0$,
 we use the first relation from
(\ref{tea3}), expanding the leftward caps decorated by $\heart$ 
using (\ref{leftwards}) when $a=0$ or
(\ref{nakano3}) and (\ref{nakano5}) when $a > 0$.
Finally, the relations (\ref{curls})--(\ref{morecurls}) follow easily from
(\ref{tea2}), (\ref{impose})--(\ref{leftwards}) and (\ref{septimus})--(\ref{gloop}).

Now we want to show that $\Theta$ is an isomorphism. We do this by
using the presentation from Definition~\ref{def1} to
construct a two-sided inverse
$$
\Phi:\HEIS_k^{\operatorname{old}}(z,t) \rightarrow
\HEIS_k^{\operatorname{new}}(z,t),
$$
still assuming that $k \leq 0$.
We define $\Phi$ on morphisms by declaring that it 
takes the rightward cup, the rightward cap, and all dots and crossings
(with any
orientation) to the
corresponding morphisms in $\HEIS_k^{\operatorname{new}}(z,t)$, 
and also
\begin{align*}
\Phi\Big(\:\mathord{
\begin{tikzpicture}[baseline = 1mm]
	\draw[-,thin,darkblue] (0.4,0) to[out=90, in=0] (0.1,0.4);
	\draw[->,thin,darkblue] (0.1,0.4) to[out = 180, in = 90] (-0.2,0);
      \node at (0.12,0.6) {$\color{darkblue}\scriptstyle{0}$};
      \node at (0.12,0.37) {$\heart$};
\end{tikzpicture}
}\:\Big) &:= -tz\: \mathord{
\begin{tikzpicture}[baseline = 1mm]
	\draw[-,thin,darkblue] (0.4,0) to[out=90, in=0] (0.1,0.4);
	\draw[->,thin,darkblue] (0.1,0.4) to[out = 180, in = 90] (-0.2,0);
 \end{tikzpicture}
}\quad\text{if $k < 0$},
&\Phi\Big(\:\mathord{
\begin{tikzpicture}[baseline = 1mm]
	\draw[-,thin,darkblue] (0.4,0) to[out=90, in=0] (0.1,0.4);
	\draw[->,thin,darkblue] (0.1,0.4) to[out = 180, in = 90] (-0.2,0);
      \node at (0.12,0.55) {$\color{darkblue}\scriptstyle{a}$};
      \node at (0.12,0.37) {$\heart$};
\end{tikzpicture}
}\:\Big) &:=
-z^2 \sum_{b \geq 1}\:
\mathord{
\begin{tikzpicture}[baseline = 1mm]
	\draw[-,thin,darkblue] (0.4,0) to[out=90, in=0] (0.1,0.4);
	\draw[->,thin,darkblue] (0.1,0.4) to[out = 180, in = 90] (-0.2,0);
      \node at (0.48,0.25) {$\color{darkblue}\scriptstyle{b}$};
      \node at (0.32,0.25) {$\dot$};
\end{tikzpicture}}
\,\mathord{\begin{tikzpicture}[baseline = .8mm]
  \draw[<-,thin,darkblue] (0.2,0.2) to[out=90,in=0] (0,.4);
  \draw[-,thin,darkblue] (0,0.4) to[out=180,in=90] (-.2,0.2);
\draw[-,thin,darkblue] (-.2,0.2) to[out=-90,in=180] (0,0);
  \draw[-,thin,darkblue] (0,0) to[out=0,in=-90] (0.2,0.2);
      \node at (0,0.2) {$\color{darkblue}{+}$};
      \node at (-0.51,0.2) {$\color{darkblue}\scriptstyle{-a-b}$};
 \end{tikzpicture}
}\quad\text{if $0 <a < -k$.}
\end{align*}
To see that $\Phi$ is well defined, we must verify the relations
from Definition~\ref{def1}. 
For (\ref{impose}), we must check the following in
$\HEIS_k^{\operatorname{new}}(z,t)$:
\begin{align*}
t\mathord{
\begin{tikzpicture}[baseline = 2mm]
	\draw[<-,thin,darkblue] (-0.25,.6) to[out=-90,in=90] (0.25,-0);
	\draw[-,thin,darkblue] (0.25,-0) to[out=-90, in=0] (0,-0.25);
	\draw[-,thin,darkblue] (0,-0.25) to[out = 180, in = -90] (-0.25,-0);
	\draw[-,line width=4pt,white] (0.25,.6) to[out=-90,in=90] (-0.25,-0);
	\draw[-,thin,darkblue] (0.25,.6) to[out=-90,in=90]    (-0.25,-0);
        \draw[-,thin,darkblue] (.25,.6) to [out=90,in=0] (0,.85);
        \draw[-,thin,darkblue] (-.25,.6) to [out=90,in=180] (0,.85);
\end{tikzpicture}
}
&= (t z^{-1}-t^{-1}z^{-1}) 1_\unit\quad\text{if $k=0$,}
&\mathord{\begin{tikzpicture}[baseline = .8mm]
  \draw[-,thin,darkblue] (0.2,0.2) to[out=90,in=0] (0,.4);
  \draw[->,thin,darkblue] (0,0.4) to[out=180,in=90] (-.2,0.2);
\draw[-,thin,darkblue] (-.2,0.2) to[out=-90,in=180] (0,0);
  \draw[-,thin,darkblue] (0,0) to[out=0,in=-90] (0.2,0.2);
      \node at (0.2,0.2) {$\dot$};
      \node at (0.49,0.2) {$\color{darkblue}\scriptstyle{-k}$};
 \end{tikzpicture}
}
&= tz^{-1} 1_\unit\quad\text{if $k < 0$.}
\end{align*}
These follow from (\ref{morecurls}) and (\ref{r2alt}).
Then the main work is to show that the images under $\Phi$ of the
morphisms
(\ref{invrel}) and (\ref{invrel2}) are two-sided inverses in
$\HEIS_k^{\operatorname{new}}(z,t)$.
When $k=0$, this is immediate from (\ref{bothalt}), so suppose that $k
< 0$.
The images under $\Phi$ of the two equations in (\ref{tea3}) are precisely the
known relations (\ref{neg}) and (\ref{posalter}). We are left with
checking that the images under $\Phi$ of the 
relations
\begin{align*}
\mathord{
\begin{tikzpicture}[baseline = 0]
	\draw[-,thin,darkblue] (0.25,.6) to[out=240,in=90] (-0.25,-0);
	\draw[-,thin,darkblue] (0.25,-0.) to[out=-90, in=0] (0,-0.25);
	\draw[-,thin,darkblue] (0,-0.25) to[out = 180, in = -90] (-0.25,-0);
	\draw[-,line width=4pt,white] (-0.25,.6) to[out=300,in=90] (0.25,-0);
	\draw[<-,thin,darkblue] (-0.25,.6) to[out=300,in=90] (0.25,-0);
      \node at (0.45,-0) {$\color{darkblue}\scriptstyle{a}$};
      \node at (0.24,-0) {$\dot$};
\end{tikzpicture}
}
&=0,
&\mathord{
\begin{tikzpicture}[baseline = -1.5mm]
	\draw[-,thin,darkblue] (-0.25,-.5) to[out=60,in=-90] (0.25,0.1);
	\draw[-,line width=4pt,white] (0.25,-.5) to[out=120,in=-90] (-0.25,0.1);
	\draw[<-,thin,darkblue] (0.25,-.5) to[out=120,in=-90] (-0.25,0.1);
	\draw[-,thin,darkblue] (0.25,0.1) to[out=90, in=0] (0,0.35);
	\draw[-,thin,darkblue] (0,0.35) to[out = 180, in = 90] (-0.25,0.1);
      \node at (0,.53) {$\color{darkblue}\scriptstyle{b}$};
      \node at (0,.33) {$\heart$};
\end{tikzpicture}
}&=0,
&
\mathord{
\begin{tikzpicture}[baseline = 2mm]
  \draw[-,thin,darkblue] (0.2,0.2) to[out=90,in=0] (0,.4);
  \draw[->,thin,darkblue] (0,0.4) to[out=180,in=90] (-.2,0.2);
\draw[-,thin,darkblue] (-.2,0.2) to[out=-90,in=180] (0,0);
  \draw[-,thin,darkblue] (0,0) to[out=0,in=-90] (0.2,0.2);
      \node at (0,.58) {$\color{darkblue}\scriptstyle{b}$};
      \node at (0,.38) {$\heart$};
   \node at (0.2,0.2) {$\dot$};
   \node at (0.4,0.2) {$\color{darkblue}\scriptstyle{a}$};
\end{tikzpicture}
}&= 
\delta_{a,b}1_\unit
\end{align*}
 hold in $\HEIS_k^{\operatorname{new}}(z,t)$
 for $0 \leq a,b < -k$.
The first of these when $a=0$ follows by (\ref{r1alt}).
To see it for $0 < a < -k$, we first apply the leftward skein relation, then
slide the dots past the crossing using the leftward analog of (\ref{teaplus}) which
may be deduced from the definition (\ref{cold}), and finally appeal to
(\ref{morecurls}).
The second and third relations follow from
(\ref{l2alt}) and (\ref{morecurls}) in the case that $b=0$.
To prove them when $0 < b < -k$, we must show that
\begin{align*}
\sum_{c \geq 1}
\mathord{
\begin{tikzpicture}[baseline = -1.5mm]
	\draw[-,thin,darkblue] (-0.25,-.5) to[out=60,in=-90] (0.25,0.1);
	\draw[-,line width=4pt,white] (0.25,-.5) to[out=120,in=-90] (-0.25,0.1);
	\draw[<-,thin,darkblue] (0.25,-.5) to[out=120,in=-90] (-0.25,0.1);
	\draw[-,thin,darkblue] (0.25,0.1) to[out=90, in=0] (0,0.35);
	\draw[-,thin,darkblue] (0,0.35) to[out = 180, in = 90] (-0.25,0.1);
      \node at (.4,.13) {$\color{darkblue}\scriptstyle{c}$};
      \node at (.25,.13) {$\dot$};
\end{tikzpicture}
}\!\!\!\mathord{\begin{tikzpicture}[baseline = .8mm]
  \draw[<-,thin,darkblue] (0.2,0.2) to[out=90,in=0] (0,.4);
  \draw[-,thin,darkblue] (0,0.4) to[out=180,in=90] (-.2,0.2);
\draw[-,thin,darkblue] (-.2,0.2) to[out=-90,in=180] (0,0);
  \draw[-,thin,darkblue] (0,0) to[out=0,in=-90] (0.2,0.2);
      \node at (0,0.2) {$\color{darkblue}{+}$};
      \node at (-0.51,0.2) {$\color{darkblue}\scriptstyle{-b-c}$};
 \end{tikzpicture}
}
&=0,
&
\sum_{c \geq 1}
\mathord{
\begin{tikzpicture}[baseline =.6mm]
  \draw[-,thin,darkblue] (0.2,0.2) to[out=90,in=0] (0,.4);
  \draw[->,thin,darkblue] (0,0.4) to[out=180,in=90] (-.2,0.2);
\draw[-,thin,darkblue] (-.2,0.2) to[out=-90,in=180] (0,0);
  \draw[-,thin,darkblue] (0,0) to[out=0,in=-90] (0.2,0.2);
   \node at (0.2,0.2) {$\dot$};
   \node at (0.55,0.2) {$\color{darkblue}\scriptstyle{a+c}$};
\end{tikzpicture}
}
\mathord{\begin{tikzpicture}[baseline = .6mm]
  \draw[<-,thin,darkblue] (0.2,0.2) to[out=90,in=0] (0,.4);
  \draw[-,thin,darkblue] (0,0.4) to[out=180,in=90] (-.2,0.2);
\draw[-,thin,darkblue] (-.2,0.2) to[out=-90,in=180] (0,0);
  \draw[-,thin,darkblue] (0,0) to[out=0,in=-90] (0.2,0.2);
      \node at (0,0.2) {$\color{darkblue}{+}$};
      \node at (-0.51,0.2) {$\color{darkblue}\scriptstyle{-b-c}$};
 \end{tikzpicture}
}&= 
-\delta_{a,b}\:z^{-2} 1_\unit
\end{align*}
in $\HEIS_k^{\operatorname{new}}(z,t)$.
For the first identity, it is zero if $b \geq -k$ as the
$(+)$-bubble vanishes by (\ref{r2}). To see it for $0 < b < -k$, use the
skein relation, commute the dots past the crossing, then appeal to
(\ref{morecurls}) and (\ref{l2alt}).
For the second identity, define a homomorphism
$\gamma:\Sym \rightarrow \End_{\HEIS_k^{\operatorname{new}}(z,t)}(\unit)$
by sending $\e_n \mapsto t^{-1} z \:
\mathord{\begin{tikzpicture}[baseline = -1mm]
  \draw[-,thin,darkblue] (0,0.2) to[out=180,in=90] (-.2,0);
  \draw[->,thin,darkblue] (0.2,0) to[out=90,in=0] (0,.2);
 \draw[-,thin,darkblue] (-.2,0) to[out=-90,in=180] (0,-0.2);
  \draw[-,thin,darkblue] (0,-0.2) to[out=0,in=-90] (0.2,0);
      \node at (0.2,0) {$\dot$};
      \node at (0.55,0) {$\color{darkblue}\scriptstyle n-k$};
\end{tikzpicture}
}$ for $n \geq 0$.
Using
$\h_{n} = \det \left(\e_{i-j+1}\right)_{i,j=1,\dots,n}$ and (\ref{d2}), it follows that 
$\gamma$ sends $\h_n \mapsto 
(-1)^{n-1} tz
\mathord{
\begin{tikzpicture}[baseline = 1.25mm]
  \draw[<-,thin,darkblue] (0,0.4) to[out=180,in=90] (-.2,0.2);
  \draw[-,thin,darkblue] (0.2,0.2) to[out=90,in=0] (0,.4);
 \draw[-,thin,darkblue] (-.2,0.2) to[out=-90,in=180] (0,0);
  \draw[-,thin,darkblue] (0,0) to[out=0,in=-90] (0.2,0.2);
   \node at (0,0.2) {$\color{darkblue}+$};
   \node at (-0.5,0.2) {$\color{darkblue}\scriptstyle{n+k}$};
\end{tikzpicture}
}$
for $n\leq -k$. Then the identity we are trying to prove follows by
applying $\gamma$ to
the identity
$\sum_{c \geq 1} (-1)^{-k-b-c} \e_{k+a+c} \h_{-k-b-c} = \delta_{a,b}$,
which is (\ref{symid}).

To complete the proof, we must show that $\Theta$ and $\Phi$ are indeed
two-sided inverses. To check that $\Theta \circ \Phi = \operatorname{Id}$,
the only difficulty is to see that
$$
\Theta\bigg(\Phi\Big(\:\mathord{
\begin{tikzpicture}[baseline = 1mm]
	\draw[-,thin,darkblue] (0.4,0) to[out=90, in=0] (0.1,0.4);
	\draw[->,thin,darkblue] (0.1,0.4) to[out = 180, in = 90] (-0.2,0);
      \node at (0.12,0.6) {$\color{darkblue}\scriptstyle{a}$};
      \node at (0.12,0.37) {$\heart$};
\end{tikzpicture}
}\:\Big)\bigg) = \mathord{
\begin{tikzpicture}[baseline = 1mm]
	\draw[-,thin,darkblue] (0.4,0) to[out=90, in=0] (0.1,0.4);
	\draw[->,thin,darkblue] (0.1,0.4) to[out = 180, in = 90] (-0.2,0);
      \node at (0.12,0.6) {$\color{darkblue}\scriptstyle{a}$};
      \node at (0.12,0.37) {$\heart$};
\end{tikzpicture}
}.
$$
When $a=0$, this is immediate from (\ref{leftwards}), while if $0 < a <
-k$ it follows from (\ref{nakano3}) and (\ref{nakano5}).
To check that $\Phi \circ \Theta = \operatorname{Id}$, the only difficulty is to
see that
$$
\Phi\Big(\:\mathord{
\begin{tikzpicture}[baseline = 1mm]
	\draw[-,thin,darkblue] (0.4,0) to[out=90, in=0] (0.1,0.4);
	\draw[->,thin,darkblue] (0.1,0.4) to[out = 180, in = 90] (-0.2,0);
\end{tikzpicture}
}\:\Big)
=\mathord{
\begin{tikzpicture}[baseline = 1mm]
	\draw[-,thin,darkblue] (0.4,0) to[out=90, in=0] (0.1,0.4);
	\draw[->,thin,darkblue] (0.1,0.4) to[out = 180, in = 90] (-0.2,0);
\end{tikzpicture}
}\:,
\qquad
\Phi\Big(\:
\mathord{
\begin{tikzpicture}[baseline = 1mm]
	\draw[-,thin,darkblue] (0.4,0.4) to[out=-90, in=0] (0.1,0);
	\draw[->,thin,darkblue] (0.1,0) to[out = 180, in = -90] (-0.2,0.4);
 \end{tikzpicture}
}\:\Big)
=
\mathord{
\begin{tikzpicture}[baseline = 1mm]
	\draw[-,thin,darkblue] (0.4,0.4) to[out=-90, in=0] (0.1,0);
	\draw[->,thin,darkblue] (0.1,0) to[out = 180, in = -90] (-0.2,0.4);
 \end{tikzpicture}
}.
$$
These follow from (\ref{leftwards}) and (\ref{r2alt})--(\ref{r1alt}).
\end{proof}


\begin{lemma}\label{lateaddition}
Suppose that $\mathcal C$ is a strict $\k$-linear monoidal category
containing objects
$\up$ and $\down$
and
morphisms
$\mathord{
\begin{tikzpicture}[baseline = 0]
	\draw[->,thin] (0.08,-.1) to (0.08,.3);
      \node at (0.08,0.07) {$\dot$};
\end{tikzpicture}
}$,
$
\mathord{\begin{tikzpicture}[baseline = 0]
	\draw[->,thin] (0.18,-.1) to (-0.18,.3);
		\draw[-,white,line width=4pt] (-0.18,-.1) to (0.18,.3);
	\draw[thin,->] (-0.18,-.1) to (0.18,.3);
\end{tikzpicture}}\;$,
$
\mathord{\begin{tikzpicture}[baseline = 0]
	\draw[thin,->] (-0.18,-.1) to (0.18,.3);
	\draw[-,white,line width=4pt] (0.18,-.1) to (-0.18,.3);
	\draw[->,thin] (0.18,-.1) to (-0.18,.3);
\end{tikzpicture}}\;$,
$\:\mathord{
\begin{tikzpicture}[baseline = .5mm]
	\draw[<-,thin] (0.35,0.3) to[out=-90, in=0] (0.1,0);
	\draw[-,thin] (0.1,0) to[out = 180, in = -90] (-0.15,0.3);
\end{tikzpicture}
}\:$
and $\:\mathord{
\begin{tikzpicture}[baseline = .5mm]
	\draw[<-,thin] (0.35,0) to[out=90, in=0] (0.1,0.3);
	\draw[-,thin] (0.1,0.3) to[out = 180, in = 90] (-0.15,0);
\end{tikzpicture}
}\:$
 satisfying (\ref{rr3})--(\ref{rightadj}).
 Then $\mathcal C$ contains at most one pair of morphisms
$\:\mathord{
\begin{tikzpicture}[baseline = .5mm]
	\draw[-,thin] (0.35,0.3) to[out=-90, in=0] (0.1,0);
	\draw[->,thin] (0.1,0) to[out = 180, in = -90] (-0.15,0.3);
\end{tikzpicture}
}\:$
and $\:\mathord{
\begin{tikzpicture}[baseline = .5mm]
	\draw[-,thin] (0.35,0) to[out=90, in=0] (0.1,0.3);
	\draw[->,thin] (0.1,0.3) to[out = 180, in = 90] (-0.15,0);
\end{tikzpicture}
}\:$ which 
satisfy
(\ref{pos})--(\ref{morecurls})
(for the sideways crossings and the
$(+)$-bubbles defined via
(\ref{rotate}) and (\ref{cold})--(\ref{d2})).
\end{lemma}

\begin{proof}
If $k \leq 0$, 
Theorem~\ref{toronto} implies that the morphism (\ref{invrel}) is invertible in $\mathcal C$, 
and $\posleft$ is the $(1,1)$-entry of the inverse matrix. This property characterizes $\posleft$ uniquely as a morphism in $\mathcal C$ when $k \leq 0$, independent of the choices of $\leftcap$ or $\leftcup$. Similarly, when $k \geq 0$, the morphism (\ref{invrel1}) is invertible in $\mathcal C$, and $\negleft$ is the $(1,1)$-entry of the inverse matrix. Thus $\negleft$ is characterized uniquely when $k \leq 0$.
To complete the proof when $k=0$, it remains to use (\ref{r2alt})--(\ref{r1alt}), since these show how to express $\leftcap$ and $\leftcup$ in terms of $\rightcap$ and $\rightcup$ and the two leftward crossings.
To complete the proof when $k < 0$, we note instead that the $(2,1)$-entry of the inverse of (\ref{invrel}) is $-tz \leftcap$, so $\leftcap$ is uniquely determined in $\mathcal C$. Then $\leftcup$ may be recovered uniquely using the relation (\ref{leftwards}) and our knowledge of $\posleft$.
Finally when $k > 0$, the $(1,2)$-entry of the inverse of (\ref{invrel1}) gives $t^{-1}z \leftcup$ and then $\leftcap$ may be recovered using (\ref{gloop}) and our knowledge of $\negleft$.
\end{proof}

To conclude the section, we formulate three more important sets of
relations.
The first of these explains how to expand
{\em curls}. It is quite surprising that we have
never needed to simplify left curls when $k > 0$ (or right curls when
$k < 0$) before this point.

\begin{lemma}
The following relations hold for any $a \in \Z$:
\begin{align}\label{dog1}
\mathord{
\begin{tikzpicture}[baseline = -0.5mm]
	\draw[<-,thin,darkblue] (0,0.6) to (0,0.3);
	\draw[-,thin,darkblue] (-0.3,-0.2) to [out=180,in=-90](-.5,0);
	\draw[-,thin,darkblue] (-0.5,0) to [out=90,in=180](-.3,0.2);
	\draw[-,thin,darkblue] (-0.3,.2) to [out=0,in=90](0,-0.3);
	\draw[-,thin,darkblue] (0,-0.3) to (0,-0.6);
	\draw[-,line width=4pt,white] (0,0.3) to [out=-90,in=0] (-.3,-0.2);
	\draw[-,thin,darkblue] (0,0.3) to [out=-90,in=0] (-.3,-0.2);
   \node at (-0.7,0.0) {$\color{darkblue}\scriptstyle{a}$};
      \node at (-0.5,0.0) {$\dot$};
\end{tikzpicture}
}&=
z\sum_{b\geq 0}
\mathord{\begin{tikzpicture}[baseline = 0.5mm]
  \draw[-,thin,darkblue] (0,0.35) to[out=180,in=90] (-.2,0.15);
  \draw[->,thin,darkblue] (0.2,0.15) to[out=90,in=0] (0,.35);
 \draw[-,thin,darkblue] (-.2,0.15) to[out=-90,in=180] (0,-0.05);
  \draw[-,thin,darkblue] (0,-0.05) to[out=0,in=-90] (0.2,0.15);
   \node at (0,.15) {$\color{darkblue}+$};
   \node at (.47,.15) {$\color{darkblue}\scriptstyle{a-b}$};
\end{tikzpicture}
}
\!\!\mathord{
\begin{tikzpicture}[baseline = -1mm]
	\draw[->,thin,darkblue] (0.08,-.4) to (0.08,.4);
   \node at (0.26,0) {$\color{darkblue}\scriptstyle{b}$};
      \node at (0.08,0) {$\dot$};
\end{tikzpicture}
}-
z\sum_{b< 0}
\mathord{\begin{tikzpicture}[baseline = 0.5mm]
  \draw[-,thin,darkblue] (0,0.35) to[out=180,in=90] (-.2,0.15);
  \draw[->,thin,darkblue] (0.2,0.15) to[out=90,in=0] (0,.35);
 \draw[-,thin,darkblue] (-.2,0.15) to[out=-90,in=180] (0,-0.05);
  \draw[-,thin,darkblue] (0,-0.05) to[out=0,in=-90] (0.2,0.15);
   \node at (0,.15) {$\color{darkblue}-$};
   \node at (.47,.15) {$\color{darkblue}\scriptstyle{a-b}$};
\end{tikzpicture}
}
\!\!\mathord{
\begin{tikzpicture}[baseline = -1mm]
	\draw[->,thin,darkblue] (0.08,-.4) to (0.08,.4);
   \node at (0.26,0) {$\color{darkblue}\scriptstyle{b}$};
      \node at (0.08,0) {$\dot$};
\end{tikzpicture}
},\\
\mathord{
\begin{tikzpicture}[baseline = -0.5mm]
	\draw[<-,thin,darkblue] (0,0.6) to (0,0.3);
	\draw[-,thin,darkblue] (-0.3,-0.2) to [out=180,in=-90](-.5,0);
	\draw[-,thin,darkblue] (-0.5,0) to [out=90,in=180](-.3,0.2);
	\draw[-,thin,darkblue] (0,-0.3) to (0,-0.6);
	\draw[-,thin,darkblue] (0,0.3) to [out=-90,in=0] (-.3,-0.2);
	\draw[-,line width=4pt,white] (-0.3,.2) to [out=0,in=90](0,-0.3);
	\draw[-,thin,darkblue] (-0.3,.2) to [out=0,in=90](0,-0.3);
   \node at (-0.7,0.0) {$\color{darkblue}\scriptstyle{a}$};
      \node at (-0.5,0.0) {$\dot$};
\end{tikzpicture}
}&=
z\sum_{b> 0}
\mathord{\begin{tikzpicture}[baseline = 0.5mm]
  \draw[-,thin,darkblue] (0,0.35) to[out=180,in=90] (-.2,0.15);
  \draw[->,thin,darkblue] (0.2,0.15) to[out=90,in=0] (0,.35);
 \draw[-,thin,darkblue] (-.2,0.15) to[out=-90,in=180] (0,-0.05);
  \draw[-,thin,darkblue] (0,-0.05) to[out=0,in=-90] (0.2,0.15);
   \node at (0,.15) {$\color{darkblue}+$};
   \node at (.47,.15) {$\color{darkblue}\scriptstyle{a-b}$};
\end{tikzpicture}
}
\!\!\mathord{
\begin{tikzpicture}[baseline = -1mm]
	\draw[->,thin,darkblue] (0.08,-.4) to (0.08,.4);
   \node at (0.26,0) {$\color{darkblue}\scriptstyle{b}$};
      \node at (0.08,0) {$\dot$};
\end{tikzpicture}
}-
z\sum_{b\leq 0}
\mathord{\begin{tikzpicture}[baseline = 0.5mm]
  \draw[-,thin,darkblue] (0,0.35) to[out=180,in=90] (-.2,0.15);
  \draw[->,thin,darkblue] (0.2,0.15) to[out=90,in=0] (0,.35);
 \draw[-,thin,darkblue] (-.2,0.15) to[out=-90,in=180] (0,-0.05);
  \draw[-,thin,darkblue] (0,-0.05) to[out=0,in=-90] (0.2,0.15);
   \node at (0,.15) {$\color{darkblue}-$};
   \node at (.47,.15) {$\color{darkblue}\scriptstyle{a-b}$};
\end{tikzpicture}
}
\!\!\mathord{
\begin{tikzpicture}[baseline = -1mm]
	\draw[->,thin,darkblue] (0.08,-.4) to (0.08,.4);
   \node at (0.26,0) {$\color{darkblue}\scriptstyle{b}$};
      \node at (0.08,0) {$\dot$};
\end{tikzpicture}
},\\
\mathord{
\begin{tikzpicture}[baseline = -0.5mm]
	\draw[<-,thin,darkblue] (0,0.6) to (0,0.3);
	\draw[-,thin,darkblue] (0,0.3) to [out=-90,in=180] (.3,-0.2);
	\draw[-,thin,darkblue] (0.3,-0.2) to [out=0,in=-90](.5,0);
	\draw[-,thin,darkblue] (0.5,0) to [out=90,in=0](.3,0.2);
	\draw[-,thin,darkblue] (0,-0.3) to (0,-0.6);
	\draw[-,line width=4pt,white] (0.3,.2) to [out=180,in=90](0,-0.3);
	\draw[-,thin,darkblue] (0.3,.2) to [out=180,in=90](0,-0.3);
   \node at (0.65,0.0) {$\color{darkblue}\scriptstyle{a}$};
      \node at (0.5,.0) {$\dot$};
\end{tikzpicture}
}&=
z\sum_{b\leq 0}
\mathord{
\begin{tikzpicture}[baseline = -.5mm]
	\draw[->,thin,darkblue] (0.08,-.3) to (0.08,.5);
   \node at (-0.12,0.1) {$\color{darkblue}\scriptstyle{b}$};
      \node at (0.08,0.1) {$\dot$};
\end{tikzpicture}
}
\mathord{\begin{tikzpicture}[baseline = -1.5mm]
  \draw[<-,thin,darkblue] (0,0.2) to[out=180,in=90] (-.2,0);
  \draw[-,thin,darkblue] (0.2,0) to[out=90,in=0] (0,.2);
 \draw[-,thin,darkblue] (-.2,0) to[out=-90,in=180] (0,-0.2);
  \draw[-,thin,darkblue] (0,-0.2) to[out=0,in=-90] (0.2,0);
   \node at (-.45,0) {$\color{darkblue}\scriptstyle{a-b}$};
      \node at (0,0) {$\color{darkblue}-$};
\end{tikzpicture}
}-z\sum_{b> 0}
\mathord{
\begin{tikzpicture}[baseline = -.5mm]
	\draw[->,thin,darkblue] (0.08,-.3) to (0.08,.5);
   \node at (-0.11,0.1) {$\color{darkblue}\scriptstyle{b}$};
      \node at (0.08,0.1) {$\dot$};
\end{tikzpicture}
}
\mathord{\begin{tikzpicture}[baseline = -1.5mm]
  \draw[<-,thin,darkblue] (0,0.2) to[out=180,in=90] (-.2,0);
  \draw[-,thin,darkblue] (0.2,0) to[out=90,in=0] (0,.2);
 \draw[-,thin,darkblue] (-.2,0) to[out=-90,in=180] (0,-0.2);
  \draw[-,thin,darkblue] (0,-0.2) to[out=0,in=-90] (0.2,0);
   \node at (-.45,0) {$\color{darkblue}\scriptstyle{a-b}$};
      \node at (0,0) {$\color{darkblue}+$};
\end{tikzpicture}
}\:,\\\label{lastofthecurls}
\mathord{
\begin{tikzpicture}[baseline = -0.5mm]
	\draw[<-,thin,darkblue] (0,0.6) to (0,0.3);
	\draw[-,thin,darkblue] (0.3,-0.2) to [out=0,in=-90](.5,0);
	\draw[-,thin,darkblue] (0.5,0) to [out=90,in=0](.3,0.2);
	\draw[-,thin,darkblue] (0,-0.3) to (0,-0.6);
	\draw[-,thin,darkblue] (0.3,.2) to [out=180,in=90](0,-0.3);
	\draw[-,line width=4pt,white] (0,0.3) to [out=-90,in=180] (.3,-0.2);
	\draw[-,thin,darkblue] (0,0.3) to [out=-90,in=180] (.3,-0.2);
   \node at (0.65,0.0) {$\color{darkblue}\scriptstyle{a}$};
      \node at (0.5,.0) {$\dot$};
\end{tikzpicture}
}&=
z\sum_{b< 0}
\mathord{
\begin{tikzpicture}[baseline = -.5mm]
	\draw[->,thin,darkblue] (0.08,-.3) to (0.08,.5);
   \node at (-0.12,0.1) {$\color{darkblue}\scriptstyle{b}$};
      \node at (0.08,0.1) {$\dot$};
\end{tikzpicture}
}
\mathord{\begin{tikzpicture}[baseline = -1.5mm]
  \draw[<-,thin,darkblue] (0,0.2) to[out=180,in=90] (-.2,0);
  \draw[-,thin,darkblue] (0.2,0) to[out=90,in=0] (0,.2);
 \draw[-,thin,darkblue] (-.2,0) to[out=-90,in=180] (0,-0.2);
  \draw[-,thin,darkblue] (0,-0.2) to[out=0,in=-90] (0.2,0);
   \node at (-.45,0) {$\color{darkblue}\scriptstyle{a-b}$};
      \node at (0,0) {$\color{darkblue}-$};
\end{tikzpicture}
}-z\sum_{b\geq 0}
\mathord{
\begin{tikzpicture}[baseline = -.5mm]
	\draw[->,thin,darkblue] (0.08,-.3) to (0.08,.5);
   \node at (-0.11,0.1) {$\color{darkblue}\scriptstyle{b}$};
      \node at (0.08,0.1) {$\dot$};
\end{tikzpicture}
}
\mathord{\begin{tikzpicture}[baseline = -1.5mm]
  \draw[<-,thin,darkblue] (0,0.2) to[out=180,in=90] (-.2,0);
  \draw[-,thin,darkblue] (0.2,0) to[out=90,in=0] (0,.2);
 \draw[-,thin,darkblue] (-.2,0) to[out=-90,in=180] (0,-0.2);
  \draw[-,thin,darkblue] (0,-0.2) to[out=0,in=-90] (0.2,0);
   \node at (-.45,0) {$\color{darkblue}\scriptstyle{a-b}$};
      \node at (0,0) {$\color{darkblue}+$};
\end{tikzpicture}
}\:.
\end{align}
\end{lemma}

The following lemma gives a braid relation for {\em alternating
  crossings}. All other variations on the braid relation can be
deduced from this plus the original braid relation from
(\ref{rr0}), by arguments similar to the proof of the braid relations in (\ref{spit1}).

\begin{lemma}
The following relation holds:
\begin{align}
\mathord{
\begin{tikzpicture}[baseline = 2mm]
	\draw[<-,thin,darkblue] (0.45,.8) to (-0.45,-.4);
        \draw[-,thin,darkblue] (-0.45,0.2) to[out=90,in=-90] (0,0.8);
	\draw[-,line width=4pt,white] (0.45,-.4) to (-0.45,.8);
	\draw[->,thin,darkblue] (0.45,-.4) to (-0.45,.8);
        \draw[-,line width=4pt,white] (0,-.4) to[out=90,in=-90] (-.45,0.2);
        \draw[<-,thin,darkblue] (0,-.4) to[out=90,in=-90] (-.45,0.2);
\end{tikzpicture}
}
-
\mathord{
\begin{tikzpicture}[baseline = 2mm]
	\draw[<-,thin,darkblue] (0.45,.8) to (-0.45,-.4);
        \draw[<-,thin,darkblue] (0,-.4) to[out=90,in=-90] (.45,0.2);
	\draw[-,line width=4pt,white] (0.45,-.4) to (-0.45,.8);
	\draw[->,thin,darkblue] (0.45,-.4) to (-0.45,.8);
        \draw[-,line width=4pt,white] (0.45,0.2) to[out=90,in=-90] (0,0.8);
        \draw[-,thin,darkblue] (0.45,0.2) to[out=90,in=-90] (0,0.8);
\end{tikzpicture}
}
&=
z^3\displaystyle\sum_{\substack{a,b \geq 0\\c > 0}}
\mathord{
\begin{tikzpicture}[baseline = 0]
	\draw[-,thin,darkblue] (0.3,0.6) to[out=-90, in=0] (0,0.2);
	\draw[->,thin,darkblue] (0,0.2) to[out = 180, in = -90] (-0.3,0.6);
  \draw[->,thin,darkblue] (-0.8,0) to[out=90,in=0] (-1,0.2);
  \draw[-,thin,darkblue] (-1,0.2) to[out=180,in=90] (-1.2,0);
\draw[-,thin,darkblue] (-1.2,0) to[out=-90,in=180] (-1,-0.2);
  \draw[-,thin,darkblue] (-1,-0.2) to[out=0,in=-90] (-0.8,0);
   \node at (-1,0) {$+$};
   \node at (-.34,0) {$\color{darkblue}\scriptstyle{-a-b-c}$};
   \node at (-0.23,0.33) {$\dot$};
   \node at (-0.43,0.33) {$\color{darkblue}\scriptstyle{a}$};
	\draw[<-,thin,darkblue] (0.3,-.6) to[out=90, in=0] (0,-0.2);
	\draw[-,thin,darkblue] (0,-0.2) to[out = 180, in = 90] (-0.3,-.6);
   \node at (-0.25,-0.4) {$\dot$};
   \node at (-.4,-.4) {$\color{darkblue}\scriptstyle{b}$};
	\draw[->,thin,darkblue] (.68,-0.6) to (.68,0.6);
   \node at (0.68,0) {$\dot$};
   \node at (.85,0) {$\color{darkblue}\scriptstyle{c}$};
\end{tikzpicture}
}&&\text{if $k \geq 0$,}\\
\mathord{
\begin{tikzpicture}[baseline = 2mm]
	\draw[->,thin,darkblue] (0.45,-.4) to (-0.45,.8);
        \draw[<-,thin,darkblue] (0,-.4) to[out=90,in=-90] (-.45,0.2);
	\draw[-,line width=4pt,white] (0.45,.8) to (-0.45,-.4);
	\draw[<-,thin,darkblue] (0.45,.8) to (-0.45,-.4);
        \draw[-,line width=4pt,white] (-0.45,0.2) to[out=90,in=-90] (0,0.8);
        \draw[-,thin,darkblue] (-0.45,0.2) to[out=90,in=-90] (0,0.8);
\end{tikzpicture}
}
-
\mathord{
\begin{tikzpicture}[baseline = 2mm]
	\draw[->,thin,darkblue] (0.45,-.4) to (-0.45,.8);
        \draw[-,thin,darkblue] (0.45,0.2) to[out=90,in=-90] (0,0.8);
	\draw[-,line width=4pt,white] (0.45,.8) to (-0.45,-.4);
	\draw[<-,thin,darkblue] (0.45,.8) to (-0.45,-.4);
        \draw[-,line width=4pt,white] (0,-.4) to[out=90,in=-90] (.45,0.2);
        \draw[<-,thin,darkblue] (0,-.4) to[out=90,in=-90] (.45,0.2);
\end{tikzpicture}
}
&=
z^3\displaystyle\sum_{\substack{a,b \geq 0\\c > 0}}
\mathord{
\begin{tikzpicture}[baseline = 0]
	\draw[<-,thin,darkblue] (0.3,0.6) to[out=-90, in=0] (0,0.2);
	\draw[-,thin,darkblue] (0,0.2) to[out = 180, in = -90] (-0.3,0.6);
  \draw[-,thin,darkblue] (1.2,0) to[out=90,in=0] (1,0.2);
  \draw[<-,thin,darkblue] (1,0.2) to[out=180,in=90] (.8,0);
\draw[-,thin,darkblue] (.8,0) to[out=-90,in=180] (1,-0.2);
  \draw[-,thin,darkblue] (1,-0.2) to[out=0,in=-90] (1.2,0);
   \node at (.3,0) {$\color{darkblue}\scriptstyle{-a-b-c}$};
   \node at (0.23,0.33) {$\dot$};
\node at (1,0) {$\color{darkblue}+$};
   \node at (0.43,0.43) {$\color{darkblue}\scriptstyle{a}$};
	\draw[-,thin,darkblue] (0.3,-.6) to[out=90, in=0] (0,-0.2);
	\draw[->,thin,darkblue] (0,-0.2) to[out = 180, in = 90] (-0.3,-.6);
   \node at (0.25,-0.4) {$\dot$};
   \node at (.4,-.4) {$\color{darkblue}\scriptstyle{b}$};
	\draw[->,thin,darkblue] (-.68,-0.6) to (-.68,0.6);
   \node at (-0.68,0) {$\dot$};
   \node at (-.85,0) {$\color{darkblue}\scriptstyle{c}$};
\end{tikzpicture}
}&&\text{if $k \leq 0$}.
\label{altbraid}
\end{align}
\end{lemma}

Finally we have the {\em bubble slides}:

\begin{lemma}\label{bubbleslides}
The following relations hold for any $a \in \Z$:
\begin{align}\label{bs1}
\mathord{\begin{tikzpicture}[baseline = -1mm]
  \draw[<-,thin,darkblue] (0,0.2) to[out=180,in=90] (-.2,0);
  \draw[-,thin,darkblue] (0.2,0) to[out=90,in=0] (0,.2);
 \draw[-,thin,darkblue] (-.2,0) to[out=-90,in=180] (0,-0.2);
  \draw[-,thin,darkblue] (0,-0.2) to[out=0,in=-90] (0.2,0);
\node at (0,0) {$\color{darkblue}+$};
   \node at (-0.32,0) {$\color{darkblue}\scriptstyle{a}$};
\end{tikzpicture}
}
\:\:\mathord{
\begin{tikzpicture}[baseline = -1mm]
	\draw[->,thin,darkblue] (0.08,-.4) to (0.08,.4);
\end{tikzpicture}
}
&=
\mathord{
\begin{tikzpicture}[baseline = -1mm]
	\draw[->,thin,darkblue] (0.08,-.4) to (0.08,.4);
\end{tikzpicture}
}
\:
\mathord{\begin{tikzpicture}[baseline = -1mm]
  \draw[<-,thin,darkblue] (0,0.2) to[out=180,in=90] (-.2,0);
  \draw[-,thin,darkblue] (0.2,0) to[out=90,in=0] (0,.2);
 \draw[-,thin,darkblue] (-.2,0) to[out=-90,in=180] (0,-0.2);
  \draw[-,thin,darkblue] (0,-0.2) to[out=0,in=-90] (0.2,0);
\node at (0,0) {$\color{darkblue}+$};
   \node at (-0.32,0) {$\color{darkblue}\scriptstyle{a}$};
\end{tikzpicture}
}
-z^2 
\sum_{\substack{b \geq 0\\c > 0}}
\,
\mathord{
\begin{tikzpicture}[baseline = -1mm]
	\draw[->,thin,darkblue] (0.08,-.4) to (0.08,.4);
   \node at (-.23,0) {$\color{darkblue}\scriptstyle{b+c}$};
      \node at (.08,0) {$\dot$};
\end{tikzpicture}
}
\:\mathord{\begin{tikzpicture}[baseline = -1mm]
  \draw[<-,thin,darkblue] (0,0.2) to[out=180,in=90] (-.2,0);
  \draw[-,thin,darkblue] (0.2,0) to[out=90,in=0] (0,.2);
 \draw[-,thin,darkblue] (-.2,0) to[out=-90,in=180] (0,-0.2);
  \draw[-,thin,darkblue] (0,-0.2) to[out=0,in=-90] (0.2,0);
\node at (0,0) {$\color{darkblue}+$};
   \node at (-0.59,0) {$\color{darkblue}\scriptstyle{a-b-c}$};
\end{tikzpicture}
}\:,\\
\label{bs2}
\mathord{
\begin{tikzpicture}[baseline = -1mm]
	\draw[->,thin,darkblue] (0.08,-.4) to (0.08,.4);
\end{tikzpicture}
}
\,\:
\mathord{\begin{tikzpicture}[baseline = -1mm]
  \draw[-,thin,darkblue] (0,0.2) to[out=180,in=90] (-.2,0);
  \draw[->,thin,darkblue] (0.2,0) to[out=90,in=0] (0,.2);
 \draw[-,thin,darkblue] (-.2,0) to[out=-90,in=180] (0,-0.2);
  \draw[-,thin,darkblue] (0,-0.2) to[out=0,in=-90] (0.2,0);
\node at (0,0) {$\color{darkblue}+$};
   \node at (0.32,0) {$\color{darkblue}\scriptstyle{a}$};
\end{tikzpicture}
}
&=
\mathord{\begin{tikzpicture}[baseline = -1mm]
  \draw[-,thin,darkblue] (0,0.2) to[out=180,in=90] (-.2,0);
  \draw[->,thin,darkblue] (0.2,0) to[out=90,in=0] (0,.2);
 \draw[-,thin,darkblue] (-.2,0) to[out=-90,in=180] (0,-0.2);
  \draw[-,thin,darkblue] (0,-0.2) to[out=0,in=-90] (0.2,0);
\node at (0,0) {$\color{darkblue}+$};
   \node at (0.32,0) {$\color{darkblue}\scriptstyle{a}$};
\end{tikzpicture}
}
\,\:
\mathord{
\begin{tikzpicture}[baseline = -1mm]
	\draw[->,thin,darkblue] (0.08,-.4) to (0.08,.4);
\end{tikzpicture}
}
-
z^2 
\sum_{\substack{b \geq 0\\c > 0}}
\:
\mathord{\begin{tikzpicture}[baseline = -1mm]
  \draw[-,thin,darkblue] (0,0.2) to[out=180,in=90] (-.2,0);
  \draw[->,thin,darkblue] (0.2,0) to[out=90,in=0] (0,.2);
 \draw[-,thin,darkblue] (-.2,0) to[out=-90,in=180] (0,-0.2);
  \draw[-,thin,darkblue] (0,-0.2) to[out=0,in=-90] (0.2,0);
\node at (0,0) {$\color{darkblue}+$};
   \node at (.59,0) {$\color{darkblue}\scriptstyle{a-b-c}$};
\end{tikzpicture}
}
\mathord{
\begin{tikzpicture}[baseline = -1mm]
	\draw[->,thin,darkblue] (0.08,-.4) to (0.08,.4);
   \node at (.4,0) {$\color{darkblue}\scriptstyle{b+c}$};
      \node at (.08,0) {$\dot$};
\end{tikzpicture}
}
\:,
\\
\mathord{\begin{tikzpicture}[baseline = -1mm]
  \draw[<-,thin,darkblue] (0,0.2) to[out=180,in=90] (-.2,0);
  \draw[-,thin,darkblue] (0.2,0) to[out=90,in=0] (0,.2);
 \draw[-,thin,darkblue] (-.2,0) to[out=-90,in=180] (0,-0.2);
  \draw[-,thin,darkblue] (0,-0.2) to[out=0,in=-90] (0.2,0);
\node at (0,0) {$\color{darkblue}-$};
   \node at (-0.32,0) {$\color{darkblue}\scriptstyle{a}$};
\end{tikzpicture}
}
\:\:\mathord{
\begin{tikzpicture}[baseline = -1mm]
	\draw[->,thin,darkblue] (0.08,-.4) to (0.08,.4);
\end{tikzpicture}
}
&=
\mathord{
\begin{tikzpicture}[baseline = -1mm]
	\draw[->,thin,darkblue] (0.08,-.4) to (0.08,.4);
\end{tikzpicture}
}
\:
\mathord{\begin{tikzpicture}[baseline = -1mm]
  \draw[<-,thin,darkblue] (0,0.2) to[out=180,in=90] (-.2,0);
  \draw[-,thin,darkblue] (0.2,0) to[out=90,in=0] (0,.2);
 \draw[-,thin,darkblue] (-.2,0) to[out=-90,in=180] (0,-0.2);
  \draw[-,thin,darkblue] (0,-0.2) to[out=0,in=-90] (0.2,0);
\node at (0,0) {$\color{darkblue}-$};
   \node at (-0.32,0) {$\color{darkblue}\scriptstyle{a}$};
\end{tikzpicture}
}
-z^2 
\sum_{\substack{b \leq 0\\c < 0}}
\:
\mathord{
\begin{tikzpicture}[baseline = -1mm]
	\draw[->,thin,darkblue] (0.08,-.4) to (0.08,.4);
   \node at (-.28,0) {$\color{darkblue}\scriptstyle{b+c}$};
      \node at (.08,0) {$\dot$};
\end{tikzpicture}
}
\:\mathord{\begin{tikzpicture}[baseline = -1mm]
  \draw[<-,thin,darkblue] (0,0.2) to[out=180,in=90] (-.2,0);
  \draw[-,thin,darkblue] (0.2,0) to[out=90,in=0] (0,.2);
 \draw[-,thin,darkblue] (-.2,0) to[out=-90,in=180] (0,-0.2);
  \draw[-,thin,darkblue] (0,-0.2) to[out=0,in=-90] (0.2,0);
\node at (0,0) {$\color{darkblue}-$};
   \node at (-0.59,0) {$\color{darkblue}\scriptstyle{a-b-c}$};
\end{tikzpicture}
}\:,\label{bs3}\\\label{bs4}
\mathord{
\begin{tikzpicture}[baseline = -1mm]
	\draw[->,thin,darkblue] (0.08,-.4) to (0.08,.4);
\end{tikzpicture}
}
\,\:
\mathord{\begin{tikzpicture}[baseline = -1mm]
  \draw[-,thin,darkblue] (0,0.2) to[out=180,in=90] (-.2,0);
  \draw[->,thin,darkblue] (0.2,0) to[out=90,in=0] (0,.2);
 \draw[-,thin,darkblue] (-.2,0) to[out=-90,in=180] (0,-0.2);
  \draw[-,thin,darkblue] (0,-0.2) to[out=0,in=-90] (0.2,0);
\node at (0,0) {$\color{darkblue}-$};
   \node at (0.32,0) {$\color{darkblue}\scriptstyle{a}$};
\end{tikzpicture}
}
&=
\mathord{\begin{tikzpicture}[baseline = -1mm]
  \draw[-,thin,darkblue] (0,0.2) to[out=180,in=90] (-.2,0);
  \draw[->,thin,darkblue] (0.2,0) to[out=90,in=0] (0,.2);
 \draw[-,thin,darkblue] (-.2,0) to[out=-90,in=180] (0,-0.2);
  \draw[-,thin,darkblue] (0,-0.2) to[out=0,in=-90] (0.2,0);
\node at (0,0) {$\color{darkblue}-$};
   \node at (0.32,0) {$\color{darkblue}\scriptstyle{a}$};
\end{tikzpicture}
}
\,\:
\mathord{
\begin{tikzpicture}[baseline = -1mm]
	\draw[->,thin,darkblue] (0.08,-.4) to (0.08,.4);
\end{tikzpicture}
}
-
z^2 
\sum_{\substack{b \leq 0\\c < 0}}
\:
\mathord{\begin{tikzpicture}[baseline = -1mm]
  \draw[-,thin,darkblue] (0,0.2) to[out=180,in=90] (-.2,0);
  \draw[->,thin,darkblue] (0.2,0) to[out=90,in=0] (0,.2);
 \draw[-,thin,darkblue] (-.2,0) to[out=-90,in=180] (0,-0.2);
  \draw[-,thin,darkblue] (0,-0.2) to[out=0,in=-90] (0.2,0);
\node at (0,0) {$\color{darkblue}-$};
   \node at (.59,0) {$\color{darkblue}\scriptstyle{a-b-c}$};
\end{tikzpicture}
}
\mathord{
\begin{tikzpicture}[baseline = -1mm]
	\draw[->,thin,darkblue] (0.08,-.4) to (0.08,.4);
   \node at (.41,0) {$\color{darkblue}\scriptstyle{b+c}$};
      \node at (.08,0) {$\dot$};
\end{tikzpicture}
}
\:.
\end{align}
\end{lemma}

\section{Action on representations of quantum $GL_n$}\label{qgln}

In this section, we construct an action of $\HEIS_0(z,t)$ on the category
of modules over $\Uq(\mathfrak{gl}_n)$
and use this action to produce a family of generators for the center of
$\Uq(\mathfrak{gl}_n)$.
These central
elements were introduced originally by Bracken, Gould and Zhang
\cite{GZB}. We also determine their images under the
Harish-Chandra homomorphism, giving a new approach to some
results of Li \cite{Li}.
Throughout the section, we work in the generic case, setting
\[
    \k := \Q(q), \qquad z := q-q^{-1}, \qquad t := q^n
\]
for an indeterminate $q$.
In fact, the formulae which we
derive are defined over $\Z[q,q^{-1}]$, hence, they make
sense over any ground ring for any invertible
$q$ (including roots of unity).

For the precise definition of $\Uq(\mathfrak{gl}_n)$,
we follow the conventions of
\cite[$\S$3]{Bskein}, denoting
its standard generators by
$\left\{e_i, f_i, d_j^{\pm 1}\:\big|\:i=1,\dots,n-1, j=1,\dots, n\right\}$.
The usual diagonal generator $k_i$ of the subalgebra
$\Uq(\mathfrak{sl}_n)$ is $d_i d_{i+1}^{-1}$.
The subalgebras of $\Uq(\mathfrak{gl}_n)$
generated by the $e_i, f_i$ and $d_j^{\pm}$
are $\Uq(\mathfrak{gl}_n)^+$,
$\Uq(\mathfrak{gl}_n)^-$ and $\Uq(\mathfrak{gl}_n)^0$, respectively.
We also have the Borel subalgebras $\Uq(\mathfrak{gl}_n)^\sharp :=
\Uq(\mathfrak{gl}_n)^0 \Uq(\mathfrak{gl}_n)^+$ and 
$\Uq(\mathfrak{gl}_n)^\flat :=
\Uq(\mathfrak{gl}_n)^0 \Uq(\mathfrak{gl}_n)^-$.
We will often cite 
 Lusztig's book \cite{Lubook}, noting that
our $q$ and $k_i$ are Lusztig's $v^{-1}$ and $K_i^{-1}$.

The natural module
$V^+$ and dual natural module 
$V^-$ are the left $\Uq(\mathfrak{gl}_n)$-modules with bases
\[
    \left\{v_i^+ \mid 1 \leq i \leq n\right\}
    \quad \text{and} \quad
    \left\{v_i^- \mid 1 \leq i\leq  n\right\},
\]
respectively, on which the generators act
by
\begin{align}\label{birds1}
f_i v^{+}_j &= \delta_{i,j} v^{+}_{i+1},
&e_i v^{+}_j &= \delta_{i+1,j} v^{+}_i,
&
d_i  v^+_j &= q^{\delta_{i,j}}  v^+_j,
\\
f_i v^{-}_j &= \delta_{i+1,j} v^{-}_i,
&
e_i v^{-}_j &= \delta_{i,j} v^{-}_{i+1},
&d_i  v^-_j &= q^{- \delta_{i,j}} v^-_j.\label{birds2}
\end{align}
We denote the weight of $v_i^+$ by $\eps_i$; then $v_i^-$ is of weight $-\eps_i$.
Let $\Lambda := \bigoplus_{i=1}^n \Z \eps_i$ be the {\em weight lattice} with inner product
$(\cdot,\cdot)$ defined so that $\eps_1,\dots,\eps_n$ are orthonormal.
The {\em positive roots} are $\{\eps_i-\eps_j\:|\:1 \leq i < j \leq n\}$.
By a {\em weight module} we mean a $\Uq(\mathfrak{gl}_n)$-module $V$
that is the sum of its weight spaces
$V_\lambda := \left\{v \in V\:\big|\:d_i v = q^{(\lambda,\eps_i)} v\right\}$
for all $\lambda \in \Lambda$.
The {\em Weyl group} is the symmetric group $\SG_n$. It acts in obvious
ways on
$\Lambda$ and on $\Uq(\mathfrak{gl}_n)^0 =
\k[d_1^{\pm 1},\dots,d_n^{\pm 1}]$,
permuting the generators.
Denote the longest element of $\SG_n$ by $w_0$.

We work with the Hopf algebra structure on $\Uq(\mathfrak{gl}_n)$
whose comultiplication $\Delta$
satisfies
\begin{align}
\Delta(e_i) &= d_i^{-1} d_{i+1} \otimes e_i + e_i \otimes 1,&
\Delta(f_i)&=1 \otimes f_i + f_i \otimes d_i d_{i+1}^{-1},&
\Delta(d_j) &= d_j \otimes d_j.
\end{align}
We also need various (anti)automorphisms.
First, we have the {\em bar involution}, which is the antilinear
automorphism $-:\Uq \rightarrow \Uq$
defined from $\overline{e_i} := e_i, \overline{f_i} := f_i$ and
$\overline{d_i} := d_i^{-1}$. Then there are linear
antiautomorphisms $T$ and $G$
defined from
\begin{align}
T(e_i) &:= f_i, &T(f_i) &:=
e_i,&T(d_i) &:= d_i,\\
G(e_i) &:= e_{n-i}, &G(f_i) &:= f_{n-i},
&G(d_i) &:= d_{n+1-i}.
\end{align}
The maps $-, T$ and $G$ commute with each other.
Finally, we have Lusztig's braid group action, under which the $i$th
generator of the braid group acts by the
automorphism $T_i:\Uq(\mathfrak{gl}_n)\rightarrow\Uq(\mathfrak{gl}_n)$ (which is $T_{i,-}''$ from
\cite[$\S$37.1.3]{Lubook}) 
defined for $|j-i| > 1$ and $k \neq i, i+1$ by
\begin{align*}
T_i(e_i) &= -f_i d_i d_{i+1}^{-1},
&T_i(e_{i\pm 1}) &= e_i e_{i\pm 1} -q^{-1} e_{i\pm 1} e_i,
&
T_i(e_j) &= e_j,\\
T_i(f_i) &= -d_i^{-1} d_{i+1} e_i,&
T_i(f_{i\pm 1}) &= f_{i\pm 1} f_i - q f_i f_{i \pm 1},&
T_i(f_j) &= f_j,\\
T_i(d_i) &= d_{i+1},&
T_i(d_{i+1}) &= d_i,&
T_i(d_k) &= d_k.
\end{align*}

A key role 
is played by the $R$-matrix. 
We recall its definition
following the approach from \cite[$\S$32.1]{Lubook}.
Let $\Theta$ be the {\em quasi-$R$-matrix} from \cite[$\S$4.1]{Lubook}. This is an infinite sum of 
components 
$\Theta_\alpha \in \Uq(\mathfrak{gl}_n)^-_{-\alpha} \otimes
\Uq(\mathfrak{gl}_n)^+_\alpha$
as $\alpha$ runs over the positive root lattice 
$\bigoplus_{i=1}^{n-1} \N (\eps_i - \eps_{i+1})$.
Let $P:V \otimes W \rightarrow W \otimes V$ be the tensor flip.
Assuming in addition that $V$ and $W$ are weight modules,
let $\Pi:V \otimes W \rightarrow V \otimes W$ be the diagonal map
defined from
$$
\Pi(v \otimes w) := q^{(\lambda,\mu)} v
\otimes w
$$ 
for $v$ of weight $\lambda$ and $w$ of weight $\mu$.
Then, for finite-dimensional weight modules $V$ and $W$,
the {\em $R$-matrix} 
\begin{equation}
R_{V,W}:V \otimes W \stackrel{\sim}{\rightarrow} W \otimes V
\end{equation}
is the $\Uq(\mathfrak{gl}_n)$-module isomorphism defined by the composition
$\Theta \circ P \circ \Pi$, 
which makes sense since all but finitely many of the components
$\Theta_\alpha$ act as zero.
The inverse $R_{V,W}^{-1}:W \otimes V \rightarrow V \otimes W$ is the
map
$\Pi^{-1} \circ P^{-1} \circ \overline{\Theta}$, where
$\overline{\Theta}$ is obtained from $\Theta$ by applying the bar involution to each
tensor factor.
For finite-dimensional weight modules $U, V$ and $W$, we have the {\em
  hexagon property}:
\begin{align}\label{hexagon}
R_{U,W} \otimes \id_V \circ \id_U \otimes R_{V,W} &= R_{U\otimes V, W},
&\id_V \otimes R_{U,W} \circ R_{U,V}\otimes \id_W &= R_{U,V\otimes W}.
\end{align}
This is proved in \cite[Proposition
32.2.2]{Lubook} (our $R_{V,W}$
is Lusztig's ${_f}\mathcal{R}_{W,V}$
taking the function $f$ from \cite[$\S$31.1.3]{Lubook} to be
$f(\lambda,\mu) := -(\lambda,\mu)$).

In fact, to define the isomorphism $R_{V,W}$, one only needs {\em one}
of the modules $V$ or $W$ to be a finite-dimensional weight module;
the other can be an arbitrary $\Uq(\mathfrak{gl}_n)$-module.
To see this, one just needs to observe that $\Pi$ 
extends to a linear map $V \otimes W \rightarrow V \otimes W$
when just one of $V$ or $W$ is a weight module
on setting
$$
\Pi (v \otimes w) := \left\{
\begin{array}{ll}
(d_\lambda \otimes 1) (v \otimes w)&\text{if $w$ is a weight vector of weight
                         $\lambda$,}\\
(1 \otimes d_\lambda) (v \otimes w)&\text{if $v$ is a weight vector of weight
                         $\lambda$,}
\end{array}
\right.
$$
where $d_\lambda := d_1^{(\lambda,\eps_1)} \cdots
d_n^{(\lambda,\eps_n)}$.
Then the same formula $R_{V,W} := \Theta \circ P\circ \Pi$ makes sense when
only one of $V$ or $W$ is a finite-dimensional weight module, and it still
gives an isomorphism of $\Uq(\mathfrak{gl}_n)$-modules. 
Moreover, the hexagon property (\ref{hexagon}) remains true if only
two of $U, V$ and $W$ are finite-dimensional weight modules.
These assertions follow from the known results in
the previous paragraph. For example, to prove that $R_{V,W}$ is an
isomorphism assuming that $W$ is a finite-dimensional weight
module, let $\rho_W \colon \Uq(\mathfrak{gl}_n) \to \End_\k(W)$ 
be the corresponding representation.  Then
\[
    (\rho_W \otimes 1)(\Theta) \in \End_\k(W) \otimes \Uq(\mathfrak{gl}_n)
    \quad \text{and} \quad
    (\rho_W \otimes 1)(\overline{\Theta}) \in \End_\k(W) \otimes \Uq(\mathfrak{gl}_n).
\]
It suffices to show that these are inverse to each other, since then $R_{V,W} = (\rho_W \otimes 1)(\Theta) \circ P \circ \Pi$ has inverse $\Pi^{-1} \circ P^{-1} \circ (\rho_W \otimes 1)(\overline{\Theta})$ for any module $V$.
We have that
\[
    (\rho_W \otimes 1)(\Theta) \circ (\rho_W \otimes 1)(\overline{\Theta}) \in \End_\k(W) \otimes \Uq(\mathfrak{gl_n})
\]
and, for any finite-dimensional weight module $V$ with corresponding representation $\rho_V$, we have
\[
    (1 \otimes \rho_V)((\rho_W \otimes 1)(\Theta) \circ (\rho_W \otimes 1)(\overline{\Theta})  ) = 1
\]
by the known result.  Since the intersection of the annihilators of all finite-dimensional weight modules is zero, this implies that 
$(\rho_W \otimes 1)(\Theta) \circ (\rho_W \otimes 1)(\overline{\Theta})   =
1$.  The proof that $(\rho_W \otimes 1)(\overline{\Theta}) \circ (\rho_W
\otimes 1)(\Theta) = 1$ is analogous, as is the proof of the hexagon
property when just two of the modules are finite-dimensional weight modules.

The goal now is to derive explicit formulae for $R_{V^{\pm}, M}$
and $R_{M, V^{\pm}}$ for any module $M$.
Similar formulae were established already in \cite[$\S$III]{GZB}
following the older conventions of Drinfeld and Jimbo.
They involve the {\em higher root elements} defined as follows.
Let 
\begin{align}
e_{i,i} = f_{i,i} &:= z^{-1},& e_{i,i+1} &:= e_i, 
&f_{i,i+1} &:=
f_i.
\end{align} 
Then when $j-i > 1$ we  recursively define
\begin{align}\label{birds}
e_{i,j} &:= e_{i,r} e_{r,j} - q^{-1} e_{r,j} e_{i,r},
&
f_{i,j} &:= f_{r,j} f_{i,r} - q^{-1} f_{i,r} f_{r,j},
\end{align}
where $r$ is any index chosen so that $i < r < j$. It is an
induction exercise to see that these elements are
well defined independent of the choice of $r$; see the proof of
the following lemma for a more conceptual explanation of this.
Alternatively, $e_{i,j}$ and $f_{i,j}$ can be defined using the braid
group action: we have that
\begin{align*}
e_{i,j} &= 
T_{j-1} \cdots T_{i+1}(e_i),
&
f_{i,j} &= \overline{T_{j-1}\cdots T_{i+1} (f_i)}.
\end{align*}
Note that
\begin{align}\label{brads}
T(e_{i,j}) &= f_{i,j},
&
T(f_{i,j}) &= e_{i,j},\\\label{brods}
G(e_{i,j}) &= e_{n+1-j,n+1-i},
&G(f_{i,j}) &= f_{n+1-j,n+1-i}.
\end{align}
However, the bar involution does not fix $e_{i,j}$ or
$f_{i,j}$ (except when $j=i+1$).

\begin{lemma}\label{socks}
For any $i < j$, the $(\eps_i-\eps_j)$-component
$\Theta_{i,j}$ of the quasi-$R$-matrix $\Theta$ satisfies
\begin{align*}
\Theta_{i,j} &= 
\!\!\!\!\sum_{\substack{r\geq 1\\i=i_0 < i_1 < \cdots < i_r = j}}
\!\!\!\!z^r f_{i_{r-1}, i_r} \cdots f_{i_0, i_1} \otimes
\overline{e_{i_{r-1},i_r} \cdots e_{i_0,i_1}}
= \!\!\!\!\sum_{\substack{r \geq 1 \\ i=i_0 < i_1 < \cdots < i_r = j}}
\!\!\!\!z^r \overline{f_{i_0,i_1} \cdots f_{i_{r-1},i_r}}
\otimes
e_{i_0,i_1} \cdots e_{i_{r-1},i_r}.
\end{align*}
\end{lemma}

\begin{proof}
It suffices to derive the first expression. Then the second follows
using (\ref{brads}) and the identitiy
$(T \otimes T)(\Theta_\alpha)= P(\Theta_\alpha)$,
which may easily be deduced from the
characterization in \cite[Theorem 4.1.2(a)]{Lubook}.
To prove the first expression, we appeal to further results of
Lusztig from \cite{Lubook}. 
Let $\mathbf{f}$ be Lusztig's ``half'' quantum group with its
standard generators $\theta_1,\dots,\theta_{n-1}$; see also
\cite[$\S$2.1]{BKM} which follows the same conventions as here.
There are two isomorphisms
\begin{align*}
(-)^+:\mathbf{f} &\stackrel{\sim}{\rightarrow} \Uq(\mathfrak{gl}_n)^+,
\quad\theta_i^+:= e_i,&
(-)^-:\mathbf{f} &\stackrel{\sim}{\rightarrow} \Uq(\mathfrak{gl}_n)^-,
\quad\theta_i^- := f_i.
\end{align*}
Consider the convex ordering on the positive roots defined so that
$\eps_i - \eps_j < \eps_p -\eps_q$ if either $i < p$ or ($i = p$ and
$j < q$); this is the ``standard order'' as in \cite[Example
A.1]{BKM}.
Let $\theta_{i,j}$ be Lusztig's higher root element associated to this
ordering,
which was denoted $r_{\eps_i-\eps_j}$ in
\cite[$\S$2.4]{BKM}.
Noting that $(\eps_m - \eps_j, \eps_i - \eps_m)$ is a minimal pair for
$\eps_i-\eps_j$, \cite[Theorem 4.2]{BKM} implies that 
these satisfy the following recursion:
$\theta_{i,i+1} = \theta_i$ and
$\theta_{i,j} = \theta_{i,r} \theta_{r,j} - q \theta_{r,j} \theta_{i,r}$
for any $i < r < j$.
Comparing with (\ref{birds}), it follows that
$\theta_{i,j}^+=\overline{e_{i,j}}$ and 
$\theta_{i,j}^-=(-q)^{j-i-1} f_{i,j}$; in particular, these equalities justify the independence
of $r$ in (\ref{birds}).
Then we appeal to \cite[Theorem 2.7]{BKM} (which was extracted from
\cite{Lubook})
to see that 
$\left\{
\theta_{i_{r-1},i_r} \cdots \theta_{i_0,i_1}\:\big|\:r \geq 1, i = i_0 <
\cdots < i_r = j\right\}$
and $\left\{(1-q^2)^r\theta_{i_{r-1},i_r} \cdots \theta_{i_0,i_1}\:\big|\:r \geq 1, i = i_0 <
\cdots < i_r = j\right\}$
are a pair of dual bases for $\mathbf{f}_{\eps_i-\eps_j}$ with respect
to Lusztig's form.
Finally the formula from \cite[Theorem 4.1.2(b)]{Lubook} gives that
$$
\Theta_{i,j} = \sum_{\substack{r \geq 1\\i = i_0 < \cdots < i_r = j}}
(-q)^{i-j}
(1-q^2)^r 
\theta^-_{i_{r-1},i_r} \cdots \theta^-_{i_0,i_1}
\otimes
\theta^+_{i_{r-1},i_r} \cdots \theta^+_{i_0,i_1}.
$$
This simplifies to the desired formula.
\end{proof}

For $1 \leq i,j \leq n$ let $e^+_{i,j} \in \End_{\k}(V^+)$
(resp.\ $e^-_{i,j} \in \End_{\k}(V^-)$) be the $ij$-matrix unit with 
respect to the basis
$v_1^+,\dots,v_n^+$ (resp.\ $v_1^-,\dots,v_n^-$).
Then for $i < j$
and $v^{\pm} \in V^{\pm}$
 we have that
\begin{align}\label{Acter}
e_{i,j} v^+ &= e^+_{i,j} v^+,&
f_{i,j} v^+ &= e^+_{j,i} v^+,&
e_{i,j} v^- &= (-q)^{i-j+1}e^-_{j,i} v^-,&
f_{i,j} v^- &= (-q)^{i-j+1} e^-_{i,j} v^-,\\
\overline{e_{i,j}} v^+ &= e^+_{i,j} v^+,&
\overline{f_{i,j}} v^+ &= e^+_{j,i} v^+,&
\overline{e_{i,j}} v^- &= (-q)^{j-i-1}e^-_{j,i} v^-,&
\overline{f_{i,j}} v^- &= (-q)^{j-i-1} e^-_{i,j} v^-.
\end{align}
These follow easily by induction on $j-i$ using (\ref{birds1})--(\ref{birds2}) and (\ref{birds}).
Also let
\begin{align}\label{lower}
x_{i,j}&:=z^2\sum_{r=1}^{\min(i,j)}e_{r,i} d_r
  f_{r,j} d_j,
&
y_{i,j} &:=z^2\sum_{r=\max(i,j)}^n d_i f_{i,r}d_r e_{j,r}
\end{align}
for any $1 \leq i,j \leq n$.
Then for $m \geq 0$ we set
\begin{align}\label{higher}
x_{i,j}^{(m)} &:= \sum_{i=i_0,i_1,\dots,i_{m-1},i_m=j} x_{i_0,i_1}
  \cdots x_{i_{m-1},i_m},&
y_{i,j}^{(m)} &:= \sum_{i=i_0,i_1,\dots,i_{m-1},i_m=j} y_{i_0,i_1}
  \cdots y_{i_{m-1},i_m}.
\end{align}
In particular, $x_{i,j}^{(0)} = y_{i,j}^{(0)} = \delta_{i,j}$.
From (\ref{brods}), we get that
\begin{align}
\label{bruds}
G\left(x_{i,j}^{(m)}\right) &= y_{n+1-j,n+1-i}^{(m)},
&G\left(y_{i,j}^{(m)}\right) &= x_{n+1-j,n+1-i}^{(m)}.
\end{align}

\begin{lemma}\label{breakfast}
For any $\Uq(\mathfrak{gl}_n)$-module $M$, the
endomorphisms
$R_{V^{\pm}, M}$ and $R_{M, V^{\pm}}$ and their inverses are given explicitly by 
the
following operators:
\begin{align*}
R_{V^+,M}
&= 
zP \circ \sum_{i \leq j}
e^+_{i,j}\otimes f_{i,j} d_j,
&
R_{V^+,M}^{-1} &=
-z P \circ \sum_{i \leq j}                  
\overline{d_i f_{i,j}}\otimes e^+_{i,j}
,\\
R_{M,V^+} &=zP \circ \sum_{i \leq j} e_{i,j}
            d_i \otimes e^+_{j,i},&
R_{M,V^+}^{-1}&=-zP \circ \sum_{i \leq j}e^+_{j,i}\otimes
                \overline{d_j e_{i,j}},\\
R_{V^-,M} &= -zP \circ \sum_{i \leq j} (-q)^{i-j} e^-_{j,i}\otimes  \overline{d_i f_{i,j}},&
R_{V^-,M}^{-1}&=zP \circ \sum_{i \leq j} (-q)^{i-j}
                f_{i,j}d_j\otimes e^-_{j,i},
\\
R_{M, V^-} &= -zP \circ \sum_{i \leq j} (-q)^{i-j}
\overline{d_j e_{i,j}}  \otimes          e^-_{i,j},&
R_{M,V^-}^{-1}&=zP \circ \sum_{i \leq j} (-q)^{i-j}
e^-_{i,j} \otimes e_{i,j} d_i.
\end{align*}
\end{lemma}

\begin{proof}
These are all proved by similar calculations, so we just go through
the argument for $R_{M,V^-}$.
Take $v \otimes v_j^- \in M \otimes V^-$.
By definition, $R_{M,V^-}(v\otimes v_j^-) = \Theta(v_j^- \otimes
d_j^{-1} v)$.
To compute the action of $\Theta$, we observe by weight considerations
that only its weight components $\Theta_{\eps_i-\eps_j}$ for $i \leq j$ are
non-zero on $v_j^- \otimes d_j^{-1} v$.
Moreover, in the first expression for $\Theta_{i,j}$ from
Lemma~\ref{socks},
all of the monomials with $r > 1$ act on $v_j^-$ as zero. We deduce
that
$$
R_{M,V^-}(v \otimes v_j^-) = v_j^- \otimes d_j^{-1} v + z\sum_{i < j}
f_{i,j} v_j^-\otimes \overline{e_{i,j} d_j} v.
$$
Then we use (\ref{Acter}) to replace $f_{i,j}$ with
$(-q)^{i-j+1}e^-_{i,j}$, the relation $e_{i,j} d_j = q d_j e_{i,j}$, and the
definition
$\overline{e_{j,j}} = -z^{-1}$ to get
$$
R_{M,V^-}(v \otimes v_j^-) = -z e^-_{j,j} v_j^- \otimes
\overline{d_j e_{j,j} } v- z\sum_{i < j}
(-q)^{i-j} e^-_{i,j} v_j^-\otimes \overline{e_{i,j} d_j} v.
$$
Now observe that the expression on the right-hand side of the formula we are trying to
prove acts on $v \otimes v_j^-$ in the same way.
\end{proof}

\begin{corollary}\label{finger}
For any $\Uq(\mathfrak{gl}_n)$-module $M$ and $m \in\Z$, we have that
\begin{align*}
\left(R_{M,V^+} \circ R_{V^+,M}\right)^m
&=
\left\{
\begin{array}{ll}
\displaystyle \sum_{i, j=1}^n e^+_{i,j}\otimes x_{i,j}^{(m)}&\text{if $m \geq 0$,}\\
\displaystyle\sum_{i,j=1}^n
e^+_{i,j} \otimes \overline{y_{i,j}^{(-m)}}&\text{if $m \leq 0$.}
\end{array}\right.
\end{align*}
\end{corollary}

\begin{proof}
This follows from Lemma~\ref{breakfast} and the definitions (\ref{lower})--(\ref{higher}).
\end{proof}

Now we 
return to the
Heisenberg category $\HEIS_0(z,t)$
taking $t := q^n$.
Let $\OS(z,t)$ be the {\em HOMFLY-PT skein category} as
defined in the introduction of \cite{Bskein}, which is 
Turaev's Hecke category from \cite{Turaev1}.
By \cite[Theorem 1.1]{Bskein}, $\OS(z,t)$ has a presentation by
generators and
relations which is very similar to the presentation of $\HEIS_0(z,t)$
from Definition~\ref{def1} but {\em without} the morphism $x$.
Consequently,
there is a strict $\k$-linear monoidal
functor
$\OS(z,t) \rightarrow \HEIS_0(z,t)$.
By \cite[Lemma 4.2]{Bskein}, 
this functor is faithful, so we may use it to {\em identify}
$\OS(z,t)$ with a subcategory of $\HEIS_0(z,t)$.
Thus, $\OS(z,t)$ is the monoidal subcategory of $\HEIS_0(z,t)$ consisting of all
objects and all morphisms which do not involve dots (i.e., $x$ or $y$).
In fact, as noted already after Definition~\ref{def1}, $\HEIS_0(z,t)$ is
the
{\em affine} HOMFLY-PT skein category from \cite[$\S$4]{Bskein}.

Let $\Uq(\mathfrak{gl}_n)\Mod$ be the category of {all} left 
$\Uq(\mathfrak{gl}_n)$-modules.
By \cite[Lemma 3.1]{Bskein} (although the result is much older, e.g., it
was exploited already in \cite{Turaev1}), there is a monoidal
functor
\begin{equation}\label{psi}
\Psi:\OS(z,t) \rightarrow \Uq(\mathfrak{gl}_n)\Mod
\end{equation}
to the category of left $\Uq(\mathfrak{gl}_n)$-modules.
The functor $\Psi$ sends the generating objects 
$\up$ and $\down$ to 
$V^+$ and $V^-$, respectively.
It maps the various generating morphisms to the following
$\Uq(\mathfrak{gl}_n)$-module homomorphisms:
\begin{align}
\label{R}
\begin{tikzpicture}[baseline = -.5mm]
	\draw[->,thin,darkblue] (0.2,-.2) to (-0.2,.3);
	\draw[line width=4pt,white,-] (-0.2,-.2) to (0.2,.3);
	\draw[thin,darkblue,->] (-0.2,-.2) to (0.2,.3);
\end{tikzpicture} :
v_i^+ \otimes v_j^+&\mapsto
\left\{
\begin{array}{l}
v_j^+ \otimes v_i^+\\
q v_j^+ \otimes v_i^+\\
v_j^+ \otimes v_i^+ + z v_i^+ \otimes v_j^+
\end{array}
\right.
&&
\begin{array}{l}
\text{if $i < j$},\\
\text{if $i = j$},\\
\text{if $i > j$};
\end{array}
\\
\begin{tikzpicture}[baseline = -.5mm]
	\draw[<-,thin,darkblue] (0.2,-.2) to (-0.2,.3);
	\draw[line width=4pt,white,-] (-0.2,-.2) to (0.2,.3);
	\draw[thin,darkblue,->] (-0.2,-.2) to (0.2,.3);
\end{tikzpicture} :
v^+_i \otimes v^-_j &\mapsto
\left\{
\begin{array}{ll}
v^-_j \otimes v^+_i\\
\displaystyle q^{-1} v^-_j \otimes v^+_i - z\sum_{r=1}^{i-1}
  (-q)^{-r} v^-_{j-r}\otimes v^+_{i-r}\\
\end{array}
\right.
&&
\begin{array}{l}
\text{if $i \neq j$},\\
\text{if $i = j$};
\end{array}
\\
\begin{tikzpicture}[baseline = -.5mm]
	\draw[<-,thin,darkblue] (0.2,-.2) to (-0.2,.3);
	\draw[line width=4pt,white,-] (-0.2,-.2) to (0.2,.3);
	\draw[thin,darkblue,<-] (-0.2,-.2) to (0.2,.3);
\end{tikzpicture} :
v^-_i \otimes v^-_j &=
\left\{
\begin{array}{ll}
 v^-_j \otimes v^-_i\\
 qv^-_j \otimes v^-_i&\\
v^-_j \otimes v^-_i+ z v^-_i \otimes v^-_j\hspace{28.4mm}\end{array}
\right.
&&
\begin{array}{l}
\text{if $i > j$},\\
\text{if $i = j$},\\
\text{if $i < j$};
\end{array}
\\
\begin{tikzpicture}[baseline = -.5mm]
	\draw[->,thin,darkblue] (0.2,-.2) to (-0.2,.3);
	\draw[line width=4pt,white,-] (-0.2,-.2) to (0.2,.3);
	\draw[thin,darkblue,<-] (-0.2,-.2) to (0.2,.3);
\end{tikzpicture} :
v^-_i \otimes v^+_j &\mapsto
\left\{
\begin{array}{ll}
v^+_j \otimes v^-_i\\
\displaystyle q^{-1} v^+_j \otimes v^-_i - z\sum_{r=1}^{n-i} (-q)^{-r} v^+_{j+r}\otimes v^-_{i+r}\\
\end{array}
\right.
&&
\begin{array}{l}
\text{if $i \neq j$},\\
\text{if $i = j$};
\end{array}\label{R4}
\end{align}
\begin{align}
\label{R1}
\begin{tikzpicture}[baseline = -.5mm]
	\draw[thin,darkblue,->] (-0.2,-.2) to (0.2,.3);
	\draw[-,line width=4pt,white] (0.2,-.2) to (-0.2,.3);
	\draw[->,thin,darkblue] (0.2,-.2) to (-0.2,.3);
\end{tikzpicture} :
v_i^+ \otimes v_j^+&\mapsto
\left\{
\begin{array}{l}
v_j^+ \otimes v_i^+\\
q^{-1} v_j^+ \otimes v_i^+\\
v_j^+ \otimes v_i^+ - z v_i^+ \otimes v_j^+
\end{array}
\right.
&&
\begin{array}{l}
\text{if $i > j$},\\
\text{if $i = j$},\\
\text{if $i < j$};
\end{array}
\\
\begin{tikzpicture}[baseline = -.5mm]
	\draw[thin,darkblue,->] (-0.2,-.2) to (0.2,.3);
	\draw[-,line width=4pt,white] (0.2,-.2) to (-0.2,.3);
	\draw[<-,thin,darkblue] (0.2,-.2) to (-0.2,.3);
\end{tikzpicture} :
v^+_i \otimes v^-_j &\mapsto
\left\{
\begin{array}{ll}
v^-_j \otimes v^+_i\\
\displaystyle q v^-_j \otimes v^+_i + z\sum_{r=1}^{n-i}
  (-q)^{r} v^-_{j+r}\otimes v^+_{i+r}\\
\end{array}
\right.
&&
\begin{array}{l}
\text{if $i \neq j$},\\
\text{if $i = j$};
\end{array}
\\
\begin{tikzpicture}[baseline = -.5mm]
	\draw[thin,darkblue,<-] (-0.2,-.2) to (0.2,.3);
	\draw[-,line width=4pt,white] (0.2,-.2) to (-0.2,.3);
	\draw[<-,thin,darkblue] (0.2,-.2) to (-0.2,.3);
\end{tikzpicture} :
v^-_i \otimes v^-_j &=
\left\{
\begin{array}{ll}
 v^-_j \otimes v^-_i\\
 q^{-1}v^-_j \otimes v^-_i&\\
v^-_j \otimes v^-_i- z v^-_i \otimes v^-_j\hspace{28.4mm}\end{array}
\right.
&&
\begin{array}{l}
\text{if $i < j$},\\
\text{if $i = j$},\\
\text{if $i > j$};
\end{array}
\\
\begin{tikzpicture}[baseline = -.5mm]
	\draw[thin,darkblue,<-] (-0.2,-.2) to (0.2,.3);
	\draw[-,line width=4pt,white] (0.2,-.2) to (-0.2,.3);
	\draw[->,thin,darkblue] (0.2,-.2) to (-0.2,.3);
\end{tikzpicture} :
v^-_i \otimes v^+_j &\mapsto
\left\{
\begin{array}{ll}
v^+_j \otimes v^-_i\\
\displaystyle q v^+_j \otimes v^-_i + z\sum_{r=1}^{i-1} (-q)^{r} v^+_{j-r}\otimes v^-_{i-r}\\
\end{array}
\right.
&&
\begin{array}{l}
\text{if $i \neq j$},\\
\text{if $i = j$};
\end{array}\label{R2}
\end{align}
\begin{align}
\mathord{
\begin{tikzpicture}[baseline = 1mm]
	\draw[<-,thin,darkblue] (0.3,0.3) to[out=-90, in=0] (0.1,0);
	\draw[-,thin,darkblue] (0.1,0) to[out = 180, in = -90] (-0.1,0.3);
\end{tikzpicture}
}\:&:1 \mapsto
\sum_{j=1}^n (-1)^j q^j
v_j^- \otimes v_j^+,
&\mathord{
\begin{tikzpicture}[baseline = 1mm]
	\draw[-,thin,darkblue] (0.3,0.3) to[out=-90, in=0] (0.1,0);
	\draw[->,thin,darkblue] (0.1,0) to[out = 180, in = -90] (-0.1,0.3);
\end{tikzpicture}
}\:&:1\mapsto \sum_{j=1}^n (-1)^j q^{n+1-j} v_j^+\otimes v_j^-,\label{easyjet1}
\\
\mathord{
\begin{tikzpicture}[baseline = 1mm]
	\draw[<-,thin,darkblue] (0.3,0) to[out=90, in=0] (0.1,0.3);
	\draw[-,thin,darkblue] (0.1,0.3) to[out = 180, in = 90] (-0.1,0);
\end{tikzpicture}
}\:&:v_i^+ \otimes v_j^- \mapsto
(-1)^i q^{-i} \delta_{i,j},&
\mathord{
\begin{tikzpicture}[baseline = 1mm]
	\draw[-,thin,darkblue] (0.3,0) to[out=90, in=0] (0.1,0.3);
	\draw[->,thin,darkblue] (0.1,0.3) to[out = 180, in = 90] (-0.1,0);
\end{tikzpicture}
}\:&:v_i^-\otimes v_j^+ \mapsto
(-1)^i q^{i-n-1}\delta_{i,j}.\label{easyjet2}
\end{align}
These formulae are recorded in many places
in the literature going back to the original work \cite{Turaev1}, but
one finds many different choices of normalization.
For our choices,
(\ref{R})--(\ref{R4})  and (\ref{R1})--(\ref{R2}) follow
from the formulae for the $R$-matrix and its inverse from
Lemma~\ref{breakfast}, while
the formulae (\ref{easyjet1})--(\ref{easyjet2}) are derived in \cite[$\S$3]{Bskein}.

\begin{theorem}\label{psihat}
Assuming $t = q^n$ and $z=q-q^{-1}$,
there is a strict $\k$-linear monoidal functor
\begin{equation*}
\widehat{\Psi} :\HEIS_0(z,t) \rightarrow
\mathcal{E}nd_\k\left(\Uq(\mathfrak{gl}_n)\Mod\right)
\end{equation*}
such that $\Psi = \operatorname{Ev}\circ\widehat{\Psi}\,\big|_{\OS(z,t)}$,
where $\operatorname{Ev}$ denotes evaluation on the
trivial module.
On objects, $\widehat{\Psi}$ takes $X$ to the endofunctor $\Psi(X)
\otimes -$, e.g., $\widehat{\Psi}(\up) = V^+\otimes-$
and $\widehat{\Psi}(\down) = V^-\otimes-$.
On morphisms, $\widehat{\Psi}$ sends $f \in \Hom_{\OS(z,t)}(X,Y)$
to the natural transformation $\Psi(f) \otimes
1:\Psi(X)\otimes-\rightarrow \Psi(Y)\otimes-$.
Finally, on the additional generating morphism
$x$,
it is defined by
\begin{align*}
\widehat{\Psi}(x)_M := 
R_{M, V^+}\circ R_{V^+, M}:
&V^+\otimes M \rightarrow V^+\otimes M,
&
v_j^+ \otimes m &\mapsto \sum_{i=1}^n v_i^+ \otimes x_{i,j} m.
\end{align*}
\end{theorem}

\begin{proof}
We just need to verify that the relations from Definition~\ref{def1}
are satisfied. All of the ones that do not involve $x$ follow
immediately since they are already satisfied by the morphisms in the
image of the 
monoidal functor $\Psi$.
Also 
$R_{V^+, M} \circ R_{M, V^+}$ is invertible since each of these
$R$-matrices is invertible. 
It just remains to check the relation (\ref{rr3}).
In fact,
this is a formal consequence of the hexagon property; see
e.g. \cite[Proposition 3.1.1]{V}. The argument goes as follows.
By (\ref{hexagon}),
we have for any $\Uq(\mathfrak{gl}_n)$-module $M$ that
$$
R_{V^+\otimes M, V^+}
\circ 
 R_{V^+, V^+\otimes M}
= R_{V^+,V^+}\otimes \id_M \circ \id_{V^+}\otimes R_{M,V^+}
\circ \id_{V^+}\otimes R_{V^+,M} 
\circ R_{V^+,V^+}\otimes \id_M.
$$
This establishes that the image under $\widehat{\Psi}$ of the relation
$$
\mathord{
\begin{tikzpicture}[baseline = -1mm]
	\draw[->,thin,darkblue] (0.18,-.6) to (0.18,.6);
 	\draw[->,thin,darkblue] (-0.18,-.6) to (-0.18,.6);
    \node at (-0.18,0) {$\dot$};
 \end{tikzpicture}
}
=\mathord{
\begin{tikzpicture}[baseline = -1mm]
	\draw[-,thin,darkblue] (0.28,-.6) to[out=90,in=-90] (-0.28,0);
	\draw[->,thin,darkblue] (0.28,0) to[out=90,in=-90] (-0.28,.6);
	\draw[-,line width=4pt,white] (-0.28,-.6) to[out=90,in=-90] (0.28,0);
	\draw[-,thin,darkblue] (-0.28,-.6) to[out=90,in=-90] (0.28,0);
	\draw[-,line width=4pt,white] (-0.28,0) to[out=90,in=-90] (0.28,.6);
	\draw[->,thin,darkblue] (-0.28,0) to[out=90,in=-90] (0.28,.6);
     \node at (0.28,0) {$\dot$};
\end{tikzpicture}
}
$$
is satisfied,
from which (\ref{rr3}) easily follows.
\end{proof}

Let $Z_q(\mathfrak{gl}_n)$ be the center of
$\Uq(\mathfrak{gl}_n)$.
It is identified with
the endomorphism algebra of the
identity functor
$\operatorname{Id}_{\Uq(\mathfrak{gl}_n)\Mod}$; indeed,
evaluation on the identity element of the regular representation
defines a canonical algebra isomorphism 
$\End\left(\operatorname{Id}_{\Uq(\mathfrak{gl}_n)\Mod}\right)
\stackrel{\sim}{\rightarrow} Z_q(\mathfrak{gl}_n).$
Dotted bubbles are endomorphisms of the unit object of $\HEIS_0(z,t)$.
Applying the monoidal functor $\widehat{\Psi}$ from
Theorem~\ref{psihat}, we 
obtain
natural transformations
\begin{equation}
\widehat{\Psi}\left(\:\mathord{
\begin{tikzpicture}[baseline = .8mm]
  \draw[-,thin,darkblue] (0.2,0.2) to[out=90,in=0] (0,.4);
  \draw[->,thin,darkblue] (0,0.4) to[out=180,in=90] (-.2,0.2);
\draw[-,thin,darkblue] (-.2,0.2) to[out=-90,in=180] (0,0);
  \draw[-,thin,darkblue] (0,0) to[out=0,in=-90] (0.2,0.2);
\node at (.2,0.2) {$\color{darkblue}\dot$};
\node at (.4,0.2) {$\color{darkblue}\scriptstyle m$};
\end{tikzpicture}}
\right): \operatorname{Id}_{\Uq(\mathfrak{gl}_n)\Mod}\rightarrow
\operatorname{Id}_{\Uq(\mathfrak{gl}_n)\Mod},
\end{equation}
hence, central elements $z_m \in Z(\Uq(\mathfrak{gl}_n))$
for each $m \in \Z$.
A calculation using (\ref{easyjet1})--(\ref{easyjet2}) and
Corollary~\ref{finger} shows that
\begin{equation}\label{chickensandwich}
z_m
=
\left\{
\begin{array}{ll}
\displaystyle\sum_{i=1}^n
q^{2i-n-1} x_{i,i}^{(m)}&\text{if $m \geq 0$,}\\
\displaystyle\sum_{i=1}^n
q^{2i-n-1} \overline{y_{i,i}^{(-m)}}&\text{if $m \leq 0$.}
\end{array}\right.
\end{equation}
We have trivially that $z_0 = [n]_q$.
The goal in the remainder of the section is to compute explicit
formulae for the 
images of all the others under the
Harish-Chandra homomorphism. 

Our argument uses
the Harish-Chandra homomorphism in two different forms
adapted to the positive and negative Borel subalgebras, respectively.
To review the definitions,
let $\rho_+ := -\eps_2-2\eps_3-\cdots-(n-1)\eps_n$ and $\rho_- :=
-(n-1)\eps_1-\cdots - 2 \eps_{n-2} - \eps_{n-1}$, i.e., $\rho_- =
w_0(\rho_+)$.
For any $\lambda \in \Lambda$, we have the {\em shift automorphism} 
\begin{equation}
S_{\!\lambda}:\Uq(\mathfrak{gl}_n)^0
\rightarrow \Uq(\mathfrak{gl}_n)^0,
\qquad
d_i 
\mapsto q^{(\lambda,\eps_i)} d_i.
\end{equation}
For example, $S_{\!-\rho_+}(d_i) = q^{i-1} d_i$ and
$S_{\!-\rho_-}(d_i) = q^{n-i} d_i$.
Let 
$\Uq(\mathfrak{gl}_n)_0$ be the zero weight space of
$\Uq(\mathfrak{gl}_n)$, which is a subalgebra containing $\Uq(\mathfrak{gl}_n)^0$.
Let $I_+$ (resp.\ $I_-$) 
be the intersection of $\Uq(\mathfrak{gl}_n)_0$ with the
left ideal of $\Uq(\mathfrak{gl}_n)$
generated by $e_1,\dots,e_{n-1}$
(resp.\ $f_1,\dots,f_{n-1}$).
Equivalently, $I_+$ (resp.\ $I_-$) is the intersection of
$\Uq(\mathfrak{gl}_n)_0$ with the right ideal generated by
$f_1,\dots,f_{n-1}$
(resp.\ $e_1,\dots, e_{n-1}$).
It follows that $I_{\pm}$ is a two-sided ideal of $\Uq(\mathfrak{gl}_n)_0$.
Let $\pr_{\pm}:\Uq(\mathfrak{gl}_n)_0 \rightarrow
\Uq(\mathfrak{gl}_n)^0$ be the algebra homomorphism defined by 
projection
along the direct sum decomposition
$\Uq(\mathfrak{gl}_n)_0 = \Uq(\mathfrak{gl}_n)^0 \oplus I_{\pm}.$
The two versions of the {\em Harish-Chandra homomorphism} are 
\begin{align}
HC_{\pm} := S_{\!-\rho_{\pm}} \circ \pr_{\pm} :\Uq(\mathfrak{gl}_n)_0
&\rightarrow \Uq(\mathfrak{gl}_n)^0.
\end{align}
The following is an extension of the well-known description of
$Z_q(\mathfrak{sl}_n)$ from e.g. \cite[6.25]{Jantzen}.

\begin{lemma}[{\cite[Lemma 2.1]{Li}}]\label{restr}
The restriction 
$HC := HC_+\big|_{Z_q(\mathfrak{gl}_n)}$
defines an algebra
isomorphism
between 
$Z_q(\mathfrak{gl}_n)$ and the algebra
$\k\big[(d_1\cdots d_n)^{-1}, d_1^2,\dots,d_n^2\big]^{\SG_n}$.
\end{lemma}

The following facts are also well known, but we could not find a suitable reference.

\begin{lemma}\label{bfix}
Each braid group generator
$T_i:\Uq(\mathfrak{gl}_n)\rightarrow \Uq(\mathfrak{gl}_n)$ fixes $Z_q(\mathfrak{gl}_n)$
pointwise.
\end{lemma}

\begin{proof}
Take $c \in Z_q(\mathfrak{gl}_n)$.
Let $V$ be an
integrable highest 
weight module. Since $V$ 
is irreducible, both $c$ and $T_i(c)$ act on $V$ as scalars.
These scalars are equal because there is an automorphism $T_i:V \rightarrow V$ 
such that $T_i(cv) =
T_i(c) T_i(v)$; see \cite[$\S$37.1.1]{Lubook}.
This shows that $c-T_i(c)$ acts as zero on every integrable highest
weight module. 
The intersection of the annihilators of all integrable highest weight modules is zero,
so this proves that $c=T_i(c)$.
\end{proof}

\begin{lemma}\label{both}
The restriction
$HC = HC_+\big|_{Z_q(\mathfrak{gl}_n)}$
is equal also to the restriction
$HC_- \big|_{Z_q(\mathfrak{gl}_n)}$.
\end{lemma}

\begin{proof}
Let $T_{\!w_0}$ be the product of simple braid group generators $T_i$
taken in some order
corresponding to a reduced expression of $w_0$.
This is an automorphism of $\Uq(\mathfrak{gl}_n)$ which switches
$\Uq(\mathfrak{gl}_n)^\sharp$ and
$\Uq(\mathfrak{gl}_n)^\flat$,
and it 
sends $d_i \mapsto d_{n+1-i}$.
It follows that
\begin{equation}
\label{hundy}
HC_{\mp} \circ T_{\!w_0} = T_{\!w_0} \circ HC_{\pm}.
\end{equation}
Clearly, $T_{\!w_0}$ fixes
$\k[(d_1\cdots d_n)^{-1}, d_1^2,\dots,d_n^2]^{\SG_n}$ pointwise.
It also fixes
$Z_q(\mathfrak{gl}_n)$ pointwise by Lemma~\ref{bfix}.
Hence,
$HC_-\big|_{Z_q(\mathfrak{gl}_n)} = 
HC_-\circ T_{\!w_0} 
\big|_{Z_q(\mathfrak{gl}_n)} = 
T_{\!w_0} \circ HC_+\big|_{Z_q(\mathfrak{gl}_n)} = HC_+\big|_{Z_q(\mathfrak{gl}_n)}.
$
\end{proof}

\begin{lemma}\label{fi}
The antiautomorphism $G$ fixes $Z_q(\mathfrak{gl}_n)$ pointwise.
\end{lemma}

\begin{proof}
We have that
\begin{equation}\label{handy}
HC_{\mp} \circ G = G \circ HC_{\pm}.
\end{equation}
Combined with Lemma~\ref{both},
it follows that
$HC_+ \circ G \big|_{Z_q(\mathfrak{gl}_n)} = G \circ HC_+ \big|_{Z_q(\mathfrak{gl}_n)}$.
Also $G$ clearly fixes 
$\k[(d_1\cdots d_n)^{-1}, d_1^2,\dots,d_n^2]^{\SG_n}$ pointwise.
Hence, $HC_+ \circ G\big|_{Z_q(\mathfrak{gl}_n)} = HC_+\big|_{Z_q(\mathfrak{gl}_n)}$, which
implies the result since $HC_+$ is injective on
$Z_q(\mathfrak{gl}_n)$.
\end{proof}

In particular, this shows that $G\left(z_m\right) = z_m$,
hence, on 
applying $G$ to the right-hand side of (\ref{chickensandwich}) using
(\ref{bruds}), we obtain another formula for $z_m$:
\begin{equation}\label{chickensandwich2}
z_m =
\left\{
\begin{array}{ll}
\displaystyle\sum_{i=1}^n
q^{n+1-2i} y_{i,i}^{(m)}&\text{if $m \geq 0$,}\\
\displaystyle\sum_{i=1}^n
q^{n+1-2i} \overline{x_{i,i}^{(-m)}}&\text{if $m \leq 0$.}
\end{array}\right.
\end{equation}
Comparing with (\ref{chickensandwich}), it follows that
\begin{equation}
z_{-m}
=
\overline{z_m}
\end{equation}
for every $m \in \Z$.
From now on, we only consider $z_m$ for $m \geq 1$.

Finally, consider the {\em modified complete symmetric
  polynomials}
\begin{equation}
\widetilde{\h}_m(x_1,\dots,x_n) := \sum_{1 \leq i_1 \leq \cdots \leq
  i_m\leq n} \left(q^{-1} z\right)^{\#\{i_1,\dots,i_m\}-1} x_{i_1}\cdots x_{i_m}.
\end{equation}
We will use these for all values of $n \geq 0$ (not just the $n$ fixed
above for $\mathfrak{gl}_n$).
We have that
\[
    \widetilde{\h}_m(x_1,\dots,x_n) = qz^{-1}
    \text{ if } m=0
    \quad \text{and} \quad
    \widetilde{\h}_m(x_1,\dots,x_n) = 0
    \text{ if } m > 0 \text{ but } n=0.
\]
These elements obviously 
satisfy the recurrence relation
\begin{equation}
\widetilde{\h}_m(x_1,\dots,x_n) = \widetilde{\h}_m(x_1,\dots,x_{n-1}) +
q^{-1} z\sum_{r=1}^m \widetilde{\h}_{m-r}(x_1,\dots,x_{n-1})
x_n^r\label{rec}
\end{equation}
for $n > 0$.

\begin{lemma}\label{dorking}
$\widetilde{\h}_m(x_1,\dots,x_n) = \widetilde{\h}_m(x_1,\dots,x_{n-1})+ \widetilde{\h}_{m-1}(x_1,\dots,x_n)x_n - q^{-2}
\widetilde{\h}_{m-1}(x_1,\dots,x_{n-1}) x_n$.
\end{lemma}

\begin{proof}
By (\ref{rec}) with $m$ replaced by $m-1$, we have that 
\begin{align*}
\widetilde{\h}_{m-1}(x_1,\dots,x_n) x_n &= 
\widetilde{\h}_{m-1}(x_1,\dots,x_{n-1})x_n+
q^{-1}z \sum_{r=1}^{m-1} \widetilde{\h}_{m-r-1}(x_1,\dots,x_{n-1}) x_n^{r+1}\\
&= 
\widetilde{\h}_{m-1}(x_1,\dots,x_{n-1})x_n+
q^{-1}z \sum_{r=2}^{m} \widetilde{\h}_{m-r}(x_1,\dots,x_{n-1}) x_n^{r}\\
&= 
q^{-2}\widetilde{\h}_{m-1}(x_1,\dots,x_{n-1})x_n+
q^{-1}z \sum_{r=1}^{m}
  \widetilde{\h}_{m-r}(x_1,\dots,x_{n-1}) x_n^{r}.
\end{align*}
Given this, it is easy to see that the right-hand side of the identity
we are trying to prove is equal to the right-hand side of (\ref{rec}).
\end{proof}

\begin{theorem}\label{images}
For any $m \geq 1$ we have that
$HC\left(z_m\right)= q^{n-1}
\widetilde{\h}_m\left(d_1^2,\dots,d_n^2\right)$.
\end{theorem}

\begin{proof}
Noting that $q^{1-n} z_m =
\sum_{i=1}^n q^{2i-2n} x_{i,i}^{(m)}$ according to
(\ref{chickensandwich}),
this follows from the following
claim:
{\em for any $m \geq 1$ and $i=1,\dots,n$, we have that}
\begin{align}\label{i2}
HC_+\left(x_{i,i}^{(m)}\right) &= 
\widetilde{\h}_m\left(d_1^2,\dots,d_i^2\right) -
q^{-2} \widetilde{\h}_m\left(d_1^2,\dots, d_{i-1}^2\right).
\end{align}
To prove (\ref{i2}), we 
proceed by induction on $m+n$. 
The result is easy to check  when
$n=1$. Now assume that $n > 1$.
The Harish-Chandra homomorphism $HC_+$ is compatible with the usual ``top left corner'' embedding of
$\Uq(\mathfrak{gl}_{n-1})$
into $\Uq(\mathfrak{gl}_n)$. 
This follows because 
the restriction of $\rho_+$ for $\mathfrak{gl}_n$ is the weight
$\rho_+$ for $\mathfrak{gl}_{n-1}$. 
Also the elements $x_{1,1}^{(m)}, \dots, x_{n-1,n-1}^{(m)}$
of $\Uq(\mathfrak{gl}_{n-1})$ are the same as these elements in
$\Uq(\mathfrak{gl}_n)$.
Thus we get (\ref{i2})
for each $i < n$ from the induction hypothesis. It remains
to prove (\ref{i2}) when $i=n$.
We have that
$$
q^{1-n} HC_-\left(z_m\right)
= \sum_{i=1}^n q^{2i-2n} \sum_{j_1,\dots,j_{m}} 
HC_-\left(
z^{2m} e_{j_1,i} d_{j_1}  f_{j_1,j_2} d_{j_2}
\cdots
e_{j_{m},j_{m-1}} d_{j_{m}} f_{j_{m},i} d_i
\right).
$$
By the definition of $HC_-$, the terms in this expansion are zero if
either
$j_1 < i$ or $j_m < i$.
Thus, the
sum simplifies to give
$$
q^{1-n} HC_-\left(z_m\right) =
\sum_{i=1}^n q^{2i-2n} HC_-\left(y_{i,i}^{(m-1)} d_i^2\right)
=
\sum_{i=1}^n HC_-\left(y_{i,i}^{(m-1)}\right) d_i^2.
$$
Now we apply $G$, using Lemma~\ref{fi}, (\ref{handy}) and
(\ref{brods}), to see that
$$
q^{1-n} HC_+\left(z_m\right) 
= \sum_{i=1}^n HC_+\left(x_{i,i}^{(m-1)}\right) d_i^2.
$$
Remembering (\ref{chickensandwich}), we have now proved that
\begin{equation}\label{i3}
\sum_{i=1}^n q^{2i-2n} HC_+\left(x_{i,i}^{(m)}\right)
= 
\sum_{i=1}^n HC_+\left(x_{i,i}^{(m-1)}\right) d_i^2.
\end{equation}
The same identity with $n$ replaced by $(n-1)$ gives
\begin{equation}\label{i4}
\sum_{i=1}^{n-1} q^{2i-2(n-1)} HC_+\left(x_{i,i}^{(m)}\right)
= \sum_{i=1}^{n-1} HC_+\left(x_{i,i}^{(m-1)}\right) d_i^2.
\end{equation}
By the induction hypothesis, the left-hand side of (\ref{i4}) is equal
to $\widetilde{\h}_m\left(d_1^2,\dots,d_{n-1}^2\right)$.
Hence, (\ref{i3}) can be rewritten to obtain
\begin{multline*}
HC_+\left(x_{n,n}^{(m)}\right)
+q^{-2}
\widetilde{\h}_m\left(d_1^2,\dots,d_{n-1}^2\right)= 
HC_+\left(x_{n,n}^{(m-1)}\right) d_n^2
+\widetilde{\h}_m\left(d_1^2,\dots,d_{n-1}^2\right)\\
=\widetilde{\h}_m\left(d_1^2,\dots, d_{n-1}^2\right)+
\widetilde{\h}_{m-1}\left(d_1^2,\dots,d_n^2\right)d_n^2 - q^{-2}\widetilde{\h}_{m-1}\left(d_1^2,\dots,d_{n-1}^2\right) d_n^2,
\end{multline*}
where we have used the induction hypothesis again to establish
the second equality.
This is equal to $\widetilde{\h}_m\left(d_1^2,\dots,d_n^2\right)$ thanks to
Lemma~\ref{dorking}.
The conclusion follows.
\end{proof}

\begin{corollary}[{\cite[Theorem 4.1]{Li}}]
$Z_q(\mathfrak{gl}_n)$ is generated by $z_1,\dots,z_n$
and $(d_1\cdots d_n)^{-1}$.
\end{corollary}

\begin{proof}
This follows from Lemma~\ref{restr} and Theorem~\ref{images} since
$\k[x_1,\dots,x_n]^{\SG_n}$ is generated
by the modified complete symmetric functions
$\widetilde{\h}_1(x_1,\dots,x_n),\dots,\widetilde{\h}_n(x_1,\dots,x_n)$.
\end{proof}

\section{Action on modules over cyclotomic Hecke algebras}\label{qcyclo}

Throughout the section, we assume that we are given a polynomial
\begin{align}\label{fwone}
f(w) &= f_0 w^l + f_1 w^{l-1}+\cdots+f_l \in \k[w]
\end{align}
of degree $l \geq 0$ such that $f_0 = 1$ and $f_l = t^2$.
Recall from the introduction that the affine Hecke algebra
$AH_n$ with its standard generators
$x_1,\dots,x_n, \tau_{1},\dots,\tau_{n-1}$
is identified with the endomorphism algebra
$\End_{\AH(z)}(\up^{\otimes n})$ so that 
and $x_{i}$ 
is the dot on the $i$th
string
and
$\tau_{j}$ 
is the positive crossing of the $j$th
and $(j+1)$th strings 
(numbering strings $1,\dots,n$ from right to left). 
The {\em cyclotomic Hecke algebra} $H_n^f$ of level $l$ 
associated to the polynomial $f(w)$
is the quotient of
$AH_n$ by the two-sided ideal generated by 
$f(x_1)$.
We also include the possibility $n=0$ with the convention that $H_0^f = \k$.

The basis theorem proved in \cite[Theorem 3.10]{AK} shows that
the following gives a basis for $H_n^f$ as a free $\k$-module:
\begin{equation}\label{akbasis}
\left\{x_1^{r_1} \cdots x_n^{r_n} \tau_{g} \:\big|\:0 \leq r_1,\dots,r_n < l, g
\in \SG_n\right\},
\end{equation}
where 
$\tau_{g}$ denotes the element of the finite Hecke algebra defined from a
reduced expression
for the permutation $g$.
By the basis theorem, the obvious homomorphism $H_n^f \rightarrow
H_{n+1}^f$
sending the generators 
$x_i$ and $\tau_{j}$ to the elements of $H_{n+1}^f$ with the same names
is {\em injective}. So we may identify $H_n^f$ with a subalgebra of
$H_{n+1}^f$.
We denote the
induction and restriction
functors by
\begin{align}
\ind_n^{n+1} := H_{n+1}^f \otimes_{H_n^f} -&:H_n^f\mod \rightarrow H_{n+1}^f\mod,\\
\res_n^{n+1} &: H_{n+1}^f\mod\rightarrow H_n^f\mod.
\end{align}
We are going to make the Abelian category 
$\bigoplus_{n \geq 0} H_n^f\mod$
into a left $\HEIS_{-l}(z,t)$-module
category,
with $\up$ and $\down$ acting as induction and restriction, respectively.
In order to do this, we need the {\em Mackey theorem} for $H_n^f$:
there is an isomorphism of functors
\begin{equation}\label{ind}
\ind_{n-1}^{n} \circ \res_{n-1}^{n} \oplus \operatorname{Id}^{\oplus l}
\stackrel{\sim}{\rightarrow} \res_n^{n+1} \circ \ind_n^{n+1}.
\end{equation}
The standard proof shows that
the map
\begin{align}\label{stdpf}
H_{n}^f \otimes_{H_{n-1}^f} H_n^f \oplus \bigoplus_{r=0}^{l-1} H_n^f
&\rightarrow H_{n+1}^f,&
(u\otimes v, w_0,\dots,w_{l-1}) \mapsto 
u \tau_{n} v + \sum_{r=0}^{l-1} w_r x_{n+1}^{r} 
\end{align}
is an isomorphism of $(H_n^f, H_n^f)$-bimodules.
This implies that there is a unique $(H_n^f, H_n^f)$-bimodule
homomorphism
\begin{equation}
\tr^f_n:H_{n+1}^f \rightarrow H_n^f
\end{equation}
such that $\tr^f_n(\tau_n) = 0$ and $\tr^f_n(x_{n+1}^r) = \delta_{r,0}$ for $0 \leq
r < l$. 

\begin{lemma}\label{lastrel}
For any $n \geq 0$, we have that $\tr^f_n\left(f(x_{n+1})\right) = 0$.
\end{lemma}

\begin{proof}
For $u, v \in H_{n+1}^f$, write $u \equiv_n v$ as shorthand for 
$u=v$ in case $n=0$, or $u-v \in
H_n^f \tau_n H_n^f$ in case $n > 0$.
We first show by induction on $n=0,1,\dots$ that
\begin{equation}
\tau_{n} \cdots \tau_{1} x_1^a \tau_{1} \cdots \tau_{n} \equiv_n
\left\{
\begin{array}{ll}
\displaystyle
\sum_{\substack{b+c_1+\cdots+c_n = a\\b > 0, c_1,\dots,c_n \geq 0}}
\left(\prod_{i\text{ with }c_i \neq 0} (-z^2 c_i)\right)
x_{n+1}^b x_n^{c_n}\cdots x_1^{c_1}&\text{if $a > 0$,}\\
\displaystyle
\sum_{\substack{b+c_1+\cdots+c_n = a\\b \leq 0, c_1,\dots,c_n \leq 0}}
\left(\prod_{i\text{ with }c_i \neq 0} (z^2 c_i)\right)
x_{n+1}^b x_n^{c_n}\cdots x_1^{c_1}&\text{if $a \leq 0$.}
\end{array}
\right.\label{enchiladas}
\end{equation}
We explain this in detail in the case $a > 0$, since the
case $a \leq 0$ is similar.
The base case is trivial. For the induction step, using the relations depicted in (\ref{teaminus})--(\ref{teaplus}),
we have that
\begin{align*}
\tau_{n} x_n^a \tau_{n} & = \tau_{n} \tau_{n}^{-1} x_{n+1}^a - z
                    \sum_{\substack{b+c=a\\b,c > 0}} \tau_{n} x_{n+1}^b
  x_n^c\\
&= 
x_{n+1}^a
- z
\sum_{\substack{b+c=a\\b,c > 0}} \tau_{n}^{-1} x_{n+1}^b
x_n^c 
-z^2 
\sum_{\substack{b+c=a\\b,c > 0}} x_{n+1}^b x_n^c\\
&\equiv_n
x_{n+1}^a
- z^2 
\sum_{\substack{b+c+d=a\\b,c,d > 0}} x_{n+1}^b
x_n^{c+d} 
-z^2 
\sum_{\substack{b+c=a\\b,c > 0}} x_{n+1}^b x_n^c
=x_{n+1}^a - z^2 \sum_{\substack{b+c=a\\b,c > 0}} c x_{n+1}^b x_n^c.
\end{align*}
Now take the expression
for
$\tau_{n-1}\cdots \tau_{1} x_1^a \tau_{1} \cdots \tau_{n-1}$
given by the induction hypothesis, multiply on left
and right by $\tau_{n}$, and use
the above identity plus the observation
$$
\tau_{n} \left(H_{n-1}^f \tau_{n-1} H_{n-1}^f\right) \tau_{n} 
= H_{n-1}^f \tau_{n} \tau_{n-1} \tau_{n} H_{n-1}^f 
= H_{n-1}^f \tau_{n-1} \tau_{n} \tau_{n-1} H_{n-1}^f
\subseteq H_n^f \tau_{n} H_n^f.
$$

Finally, to deduce the lemma, 
we multiply (\ref{enchiladas}) by $f_{l-a}$ and
sum over $a = 0,1,\dots,l$ to show
$$
\tau_{n} \cdots \tau_{1} f(x_1) \tau_{1} \cdots \tau_{n} \equiv_n
f_l + 
\displaystyle
\sum_{a = 1}^l
f_{l-a}
\sum_{\substack{b+c_1+\cdots+c_n = a\\b > 0, c_1,\dots,c_n \geq 0}}
\left(\prod_{i\text{ with }c_i \neq 0} (-z^2 c_i)\right)
x_{n+1}^b x_n^{c_n}\cdots x_1^{c_1}.
$$
The left-hand side is zero by the cyclotomic relation in $H_{n+1}^f$.
The right-hand side is equal to $f(x_{n+1})$ plus terms in the kernel
of $\tr_n^f$.
\end{proof}

\begin{theorem}\label{catact}
There is a unique strict $\k$-linear monoidal functor
$$
\Psi_f:\HEIS_{-l}(z,t) \rightarrow \mathcal{E}nd_\k\left(\bigoplus_{n \geq 0}
H_n^f\mod\right)
$$
sending the generating object $\up$ (resp.\ $\down$) to the additive endofunctor
that takes an $H_n^f$-module $M$ to $\ind_{n}^{n+1} M$
(resp.\ $\res^n_{n-1} M$), and the generating morphisms $x,
\tau, c$ and $d$ to the natural transformations
defined on the $H_n^f$-module $M$ as follows:
\begin{itemize}
\item 
$\Psi_f(x)_M:H_{n+1}^f \otimes_{H_n^f} M \rightarrow H_{n+1}^f
\otimes_{H_n^f} M$, $u\otimes v\mapsto  u x_{n+1} \otimes v$;
\item 
$\Psi_f(\tau)_M:H_{n+2}^f \otimes_{H_n^f} M \rightarrow H_{n+2}^f
\otimes_{H_n^f} M$, $u\otimes v \mapsto  u \tau_{n+1} \otimes v$ (where we
have identified $\ind_{n+1}^{n+2} \circ \ind_{n}^{n+1}$ with
$\ind_{n}^{n+2}$ in the obvious way);
\item 
$\Psi_f(c)_M:M \rightarrow \res^{n+1}_n\left(H_{n+1}^f \otimes_{H_n^f} M\right)$, $v
\mapsto 1 \otimes v$, i.e., it is the unit of the
canonical adjunction making $(\ind_{n}^{n+1}, \res_n^{n+1})$ into an
adjoint pair of functors;
\item
$\Psi_f(d)_M:H_n^f \otimes_{H_{n-1}^f} (\res^n_{n-1} M ) \rightarrow M$, $u
\otimes v
\mapsto uv$, i.e., it is the counit of the
canonical adjunction making $(\ind_{n-1}^{n}, \res_{n-1}^{n})$ into an
adjoint pair of functors.
\end{itemize}
\end{theorem}

\begin{proof}
We use the presentation for $\HEIS_{-l}(z,t)$ from
Definition~\ref{def1}. Let us first treat the case $l=0$.
In this case, the polynomial $f(w)$ from (\ref{fwone}) is 1 and $t^2=1$.
The category $\bigoplus_{n \geq 0}
H_n^f\mod$ is simply the category of left $\k$-modules, and all of the
induction and restriction functors are zero.
Consequently, almost of the relations are
trivially true. 
The only one that requires any thought is
the relation
$\clockright
= (t z^{-1}-t^{-1} z^{-1}) 1_\unit$ from (\ref{impose}).
This holds because 
the scalar on 
the right-hand side is zero as
$t^2 = 1$.

Henceforth, we assume that $l > 0$.
Then $\HEIS_{-l}(z,t)$ is
generated by the objects $\up$ and
$\down$ and morphisms
$x, \tau, c$ and $d$ 
subject to the relations
(\ref{rr3})--(\ref{rightadj}), plus two more relations:
\begin{itemize}
\item[(1)] 
$\left[\:
\mathord{
\begin{tikzpicture}[baseline = 0]
	\draw[->,thin,darkblue] (-0.28,-.3) to (0.28,.4);
	\draw[-,line width=4pt,white] (0.28,-.3) to (-0.28,.4);
	\draw[<-,thin,darkblue] (0.28,-.3) to (-0.28,.4);
\end{tikzpicture}
}\:\:\:
\mathord{
\begin{tikzpicture}[baseline = -0.9mm]
	\draw[<-,thin,darkblue] (0.4,0.2) to[out=-90, in=0] (0.1,-.2);
	\draw[-,thin,darkblue] (0.1,-.2) to[out = 180, in = -90] (-0.2,0.2);
\end{tikzpicture}
}
\:\:\:
\mathord{
\begin{tikzpicture}[baseline = -0.9mm]
	\draw[<-,thin,darkblue] (0.4,0.2) to[out=-90, in=0] (0.1,-.2);
	\draw[-,thin,darkblue] (0.1,-.2) to[out = 180, in = -90] (-0.2,0.2);
      \node at (0.38,0) {$\dot$};
\end{tikzpicture}
}
\:\:\:\cdots
\:\:\:
\mathord{
\begin{tikzpicture}[baseline = -0.9mm]
	\draw[<-,thin,darkblue] (0.4,0.2) to[out=-90, in=0] (0.1,-.2);
	\draw[-,thin,darkblue] (0.1,-.2) to[out = 180, in = -90] (-0.2,0.2);
     \node at (0.73,0) {$\color{darkblue}\scriptstyle{l-1}$};
      \node at (0.38,0) {$\dot$};
\end{tikzpicture}
}
\right]$ is invertible where $\sigma := \mathord{
\begin{tikzpicture}[baseline = 0]
	\draw[->,thin,darkblue] (-0.28,-.3) to (0.28,.4);
	\draw[-,line width=4pt,white] (0.28,-.3) to (-0.28,.4);
	\draw[<-,thin,darkblue] (0.28,-.3) to (-0.28,.4);
\end{tikzpicture}
}$ is defined by (\ref{rotate});
\item[(2)] 
$\mathord{\begin{tikzpicture}[baseline = .8mm]
  \draw[-,thin,darkblue] (0.2,0.2) to[out=90,in=0] (0,.4);
  \draw[->,thin,darkblue] (0,0.4) to[out=180,in=90] (-.2,0.2);
\draw[-,thin,darkblue] (-.2,0.2) to[out=-90,in=180] (0,0);
  \draw[-,thin,darkblue] (0,0) to[out=0,in=-90] (0.2,0.2);
      \node at (0.2,0.2) {$\dot$};
      \node at (0.33,0.2) {$\color{darkblue}\scriptstyle{l}$};
 \end{tikzpicture}
}
= t z^{-1} 1_\unit$ where $\gamma := 
\mathord{
\begin{tikzpicture}[baseline = 1mm]
	\draw[-,thin,darkblue] (0.4,0) to[out=90, in=0] (0.1,0.4);
	\draw[->,thin,darkblue] (0.1,0.4) to[out = 180, in = 90] (-0.2,0);
\end{tikzpicture}
}$ is
defined by (\ref{leftwards}), i.e., it is $-t^{-1} z^{-1}$ times the
$(2,1)$-entry of the inverse of the matrix in (1).
\end{itemize}
The relations (\ref{rr3})--(\ref{rightadj}) are straightforward to check.
On $H_n^f$-modules,
$\Psi_f(\sigma)$ comes from the $(H_n^f, H_n^f)$-bimodule homomorphism
$H_n^f\otimes_{H_{n-1}^f} H_n^f \rightarrow
H_{n+1}^f, u \otimes v \mapsto u \tau_{n} v$.
So we get the relation (1)
since (\ref{stdpf}) is invertible by the proof of the Mackey
theorem.
Moreover, we see from (\ref{stdpf}) and the definition that
$\Psi_f(\gamma)$ comes from the $(H_n^f, H_n^f)$-bimodule homomorphisms
$-t^{-1} z^{-1} \tr^f_n:H_{n+1}^f \rightarrow H_n^f$ for all $n  \geq
0$.
So for (2) we must show that
$-t^{-1} z^{-1} \tr^f_n\left(x_{n+1}^{l}\right) = t z^{-1}$.
This follows from Lemma~\ref{lastrel} and the definition of
$\tr_n^f$,
remembering that $t^2=f_l$.
\end{proof}

If we switch the roles of induction and restriction, we can
reformulate Theorem~\ref{catact} in terms of Heisenberg categories of positive
central charge.
We prefer for this to
replace the induction functor $\ind_n^{n+1}$ from before (which is
the canonical left adjoint to restriction) with the {\em coinduction functor}
\begin{equation}\label{coind}
\coind_n^{n+1} := \Hom_{H_n^f}(H_{n+1}^f, -):H_n^f\mod \rightarrow
H_{n+1}^f\mod\end{equation}
which is its canonical right adjoint. 

\begin{theorem}\label{catact2}
There is a unique strict $\k$-linear monoidal functor
$$
\Psi^\vee_f:\HEIS_{l}(z,t^{-1}) \rightarrow \mathcal{E}nd_\k\left(\bigoplus_{n \geq 0}
H_n^f\mod\right)
$$
sending the generating object $\up$ (resp.\ $\down$) to the additive endofunctor
that takes an $H_n^f$-module $M$ to $\res_{n-1}^{n} M$
(resp.\ $\coind^{n+1}_{n} M$), and the generating morphisms $x,
\tau, c$ and $d$ to the natural transformations
defined on the $H_n^f$-module $M$ as follows:
\begin{itemize}
\item 
$\Psi^\vee_f(x)_M:\res^n_{n-1} M \rightarrow \res^n_{n-1} M$,
$v\mapsto  x_n v$;
\item 
$\Psi^\vee_f(\tau)_M:\res^n_{n-2} M \rightarrow \res^n_{n-2} M$, 
 $v \mapsto - \tau_{n-1}^{-1} v$;
\item 
$\Psi^\vee_f(c)_M:M \rightarrow \Hom_{H_{n-1}^f}(H_n^f, \res^n_{n-1} M)$,
$v
\mapsto (u \mapsto uv)$, i.e., it is the unit of the
canonical adjunction making $(\res_{n-1}^{n}, \coind_{n-1}^{n})$ into an
adjoint pair of functors;
\item
$\Psi^\vee_f(d)_M:\res^{n+1}_n \left(\Hom_{H_n^f}(H_{n+1}^f,
  M)\right)\rightarrow M$,
$\theta 
\mapsto \theta(1)$, i.e., it is the counit of the
canonical adjunction making $(\res_{n}^{n+1}, \coind_{n}^{n+1})$ into an
adjoint pair of functors.
\end{itemize}
\end{theorem}

\begin{proof}
This may be proved directly in a similar way to the proof of
Theorem~\ref{catact}. One uses the presentation for $\HEIS_l(z,t^{-1})$ from
Definition~\ref{def2} instead of the one from Definition~\ref{def1},
plus the Mackey isomorphism (\ref{stdpf}) and Lemma~\ref{lastrel} as before.
We leave the details to the reader.
\end{proof}

In fact, we have that $\ind_n^{n+1} \cong
\coind_n^{n+1}$.
This follows by the uniqueness of adjoints,
since 
Lemma~\ref{rigiditylemma} and
Theorem~\ref{catact} (resp.\ Theorem~\ref{catact2})
implies that $\ind_n^{n+1}$ is right adjoint to restriction
as well as being left adjoint (resp.\ $\coind_n^{n+1}$ is left adjoint
to restriction as well as being right adjoint).
It follows that all three functors (induction, coinduction and
restriction) 
send finitely generated projective
modules to finitely generated projective modules. 
Hence:

\begin{lemma}\label{advising1}
The restrictions of the functors $\Psi_f$ and $\Psi^\vee_f$ 
constructed in Theorems~\ref{catact}--\ref{catact2} give
strict $\k$-linear monoidal functors
\begin{align*}
\Psi_f:\HEIS_{-l}(z,t) &\rightarrow \mathcal{E}nd_\k\left(\bigoplus_{n \geq 0} H_n^f\proj\right),&
\Psi^\vee_f:\HEIS_{l}(z,t^{-1}) &\rightarrow
                               \mathcal{E}nd\left(\bigoplus_{n\geq 0}
                               H_n^f\proj\right),
\end{align*}
where $H_n^f\proj$ denotes the category of finitely generated
projective left $H_n^f$-modules.
\end{lemma}

\section{Action on category $\mathcal O$ for rational Cherednik algebras}\label{scherednik}

The Heisenberg action
on $\bigoplus_{n \geq 0} H_n^f \mod$ from Theorem~\ref{catact}
can also be extended to an action on
the category $\mathcal O$ for rational Cherednik algebras,
following an argument of Shan.
To explain this in more detail, assume that $\k = \mathbb{C}$, and consider the complex reflection group $G(l,1,n)\cong \SG_n\wr \Z/l\Z$ for $l \geq 1$, with reflection representation $\k^n$ defined as in \cite[$\S$3.1]{Shan}. Defining a rational Cherednik algebra requires a choice of parameters, for which there are a bewildering number of different parameterizations. We have:\begin{itemize}
    \item a single parameter $\kappa\in\k$, which is the parameter $k_{H,1}$ in \cite[Remark 3.2]{GGOR} for a reflecting hyperplane $H$ on which the difference of two coordinates vanish; 
    \item  an $l$-tuple $(\kappa_1,\dots,\kappa_l) \in \k^l$ of parameters, which corresponds to the family $\{k_{H,i}\}_{0 \leq i \leq l}$ of parameters in \cite[Remark 3.2]{GGOR} associated to a reflecting hyperplane $H$ on which a single coordinate vanishes so that $\kappa_i = k_{H,i}$. In {\em loc.\ cit.}, it is assumed that $k_{H,0}=k_{H,l}=0$, but adding a constant to all $k_{H,i}$ leaves the algebra unchanged. It is useful for us to incorporate an additional degree of freedom, so we drop the vanishing condition here: our parameter $\kappa_l$ may be non-zero.
\end{itemize}
Let $\Cherednik_n$ be the rational Cherednik algebra
attached to these parameters as in \cite[\S 3]{GGOR}.

Let $q:=\exp(\sqrt{-1}\pi\kappa)$ and $q_i:=\exp(\sqrt{-1}\pi(\kappa_{i}-i/\ell))$ for $i=1,\dots,l$. 
One can relate these to the parameters in \cite{Shan} by choosing integers $e \geq 2$ and $(s_1,\dots, s_l)$ then letting $\kappa := \frac{1}{e}$ and $\kappa_{i}:=\kappa s_i +i/\ell$, so $q_i=q^{s_i}$, for $i=1,\dots,l$; note that the parameter $q$ in {\em loc.\ cit.} is our $q^2$.
Let $\mathcal O = \mathcal O_{\kappa;\kappa_1,\dots,\kappa_l} := \bigoplus_{n\geq 0} O_n$ where $\mathcal O_n$ is the category of $\Cherednik_n$-modules introduced in \cite[$\S$3]{GGOR}.
Also define
$$
f(w) := \prod_{i=1}^{l} \left(w + q_i^2\right),
\qquad
t := q_1\dots q_{l}.
$$
By \cite[Theorem 5.16]{GGOR}, there is an exact functor
\begin{equation}
\KZ:\mathcal O \rightarrow
\bigoplus_{n \geq 0} H_n^f\mod.
\end{equation}
Note that this functor depends on a choice for each $n$ of a basepoint in the subset of $\mathbb{C}^n$ where all entries are distinct and non-zero. Different basepoints give  isomorphic functors, but the isomorphism depends on the homotopy class of a path between the basepoints.  For simplicity, we assume these basepoints are chosen to lie 
in the set $\left\{(b_1,\dots, b_n)\in \mathbb{R}^n\:\big|\:0< b_1 <\dots <b_n\right\}$. Since this is a contractible space, the resulting $\KZ$ functors are all canonically isomorphic, and there is no need for us to be more specific.  

Matching with the formulae in \cite{GGOR, Shan} requires using the isomorphism from the cyclotomic Hecke algebra in \cite[$\S$3.1]{Shan} to ours that sends the
generators $T_0, T_1,\dots,T_{n-1}$
to  $-x_1, q \tau_1,\dots,q \tau_{n-1}$. The Hecke algebra generators $T_i\:(i=1,\dots,n-1)$ in \cite{Shan} are of the form $-T$ for Hecke algebra generators $T$ from \cite[$\S$5.2.5]{GGOR} associated to reflections in the first type of hyperplane above.
Also, $T_0$ is a scalar multiple (depending on the choice of $\kappa_l$) of the Hecke algebra generator $T$ 
in \cite[$\S$5.2.5]{GGOR} associated to a
reflection of the second type. The key point in all of this is that the minimal polynomials for $x_1$ and $\tau_i\:(i =1,\dots,n-1)$ 
arising from the key formula in \cite[\S 5.2.5]{GGOR} are $f(w)$ and $(w-q)(w+q^{-1})$ (up to scalars), i.e., we do indeed get defining relations of $H_n^f$.

The functor $\KZ$ is fully faithful on projectives \cite[Theorem 5.16]{GGOR}. Moreover, it intertwines the Bezrukavnikov-Etingof induction and restriction functors denoted $\ind_{b_{n+1}}$ and $\res_{b_{n+1}}$ in \cite[$\S$3.2]{Shan}
with the functors
$\ind_n^{n+1}$ and $\res_{n}^{n+1}$ thanks to \cite[Theorem 2.1]{Shan}.  These induction and restriction functors also depend on a choice of basepoint with a particular stabilizer, which following Shan we fix to be $(0,0,\dots, 0,1)$. (It would be more philosophically consistent with our previous conventions to say that whenever we choose a basepoint for restriction, we choose one of the form  $(b_1,\dots, b_n)\in \mathbb{R}^n$ such that $0\leq b_1\leq b_2\leq \cdots \leq b_n$; whether we have equality or strict inequality depends on which stabilizer we wish to have under the action of  $G(l,1,n)$.  As before, all such choices give canonically isomorphic functors.)
\begin{theorem}
There is a strict $\k$-linear monoidal functor
\begin{equation}
\widehat{\Psi}_f:\HEIS_{-l}(z,t) \rightarrow \mathcal{E}nd_\k\left(\mathcal{O}\right).
\end{equation}
that makes $\mathcal O$ into a module category over $\HEIS_{-l}(z,t)$,
with $\up$ and $\down$ acting as Bezrukavnikov-Etingof induction and restriction functors, respectively.
This can be done
in such a way that 
$\KZ$ is a morphism of $\HEIS_{-l}(z,t)$-module categories, viewing $\bigoplus_{n \geq 0} H_n^f\mod$ as a module category via the functor $\Psi_f$ from Theorem~\ref{catact}. 
\end{theorem}

\begin{proof}
Our argument is exactly as in the proof of \cite[Theorem 5.1]{Shan} using \cite[Lemma 2.4]{Shan}.   We need to show that there are certain natural transformations of functors satisfying specific relations.  Theorem~\ref{catact} allows us to define these on the image of the functor $\KZ$ via the action of
$\HEIS_{-l}(z,t)$.   The full-faithfulness of $\KZ$ allows us to transfer this to an action on the full subcategory of projectives in $\mathcal{O}$.  Since $\mathcal{O}$ has enough projectives by \cite[Corollary 2.8]{GGOR}, this action can be extended to an arbitrary object $X$ by presenting $X$ as the cokernel of a map between projectives.  The resulting action is well-defined due to the fact that endomorphisms of an object lift to any projective resolution uniquely up to homotopy.
\end{proof}

\begin{remark}
This quantum Heisenberg action is in many ways more convenient for working with category $\cO$ over Cherednik algebras than a Kac-Moody 2-category action, since the Heisenberg action requires no special assumptions on parameters.  In fact, this action is still well defined if $\k$ is replaced by a complete local ring, so one can extend the Heisenberg action to deformed category $\cO$.
\end{remark}

\section{Categorical comultiplication}\label{scc}

In this section, we construct the quantum analog of the categorical
comultiplication from \cite[Theorem 5.4]{BSW1}.  As discussed in \cite[Theorem 1.3]{BSW1}, the name ``categorical comultiplication'' derives from the relationship of this map to the usual comultiplication on the universal enveloping algebra of the Heisenberg Lie algebra.  Since in the quantum case an explicit description of 
$K_0(\Kar(\HEIS_k(z,t)))$ analogous to that of \cite[Theorem 1.1]{BSW1} is not available, 
we will not make a precise statement along these lines here, but we fully expect an analogue of \cite[Theorem 1.3]{BSW1} to hold in all situations where the Grothendieck ring has the expected form.
As well as the quantum Heisenberg
category $\HEIS_k(z,t)$, we will work with 
$\blue{\HEIS_{l}(z,u)}$ and $\red{\HEIS_m(z,v)}$
for $\blue{l}, \red{m} \in \Z$ and $\blue{u}, \red{v} \in \k^\times$
chosen so that
\begin{align}
k &= \blue{l}+\red{m},
&
t &=\blue{u}\red{v}.
\end{align}
To avoid confusion between these 
different categories, the reader will want to view the material in
this section in color.

Let $\blue{\HEIS_l(z,u)}\odot \red{\HEIS_m(z,v)}$ be the symmetric product
of $\blue{\HEIS_l(z,u)}$ and $\red{\HEIS_m(z,v)}$ as defined \cite[$\S$3]{BSW1}.  This is the strict $\k$-linear monoidal category defined by first taking the free product of
$\blue{\HEIS_l(z,u)}$ and $\red{\HEIS_m(z,v)}$, i.e., 
the strict $\k$-linear monoidal category defined by the
disjoint union of the given generators and relations of $\blue{\HEIS_l(z,u)}$ and of $\red{\HEIS_m(z,v)}$, then adjoining isomorphisms $\sigma_{X,Y}:X \otimes Y \stackrel{\sim}{\rightarrow} Y \otimes
X$ for each pair of objects $X \in \blue{\HEIS_l(z,u)}$ and $Y \in \red{\HEIS_m(z,v)}$
subject to the relations
\begin{align*}
    \sigma_{X_1 \otimes X_2, Y} &= (\sigma_{X_1,Y} \otimes 1_{X_2}) \circ
      (1_{X_1} \otimes \sigma_{X_2,Y}),&
    \sigma_{X_2,Y} \circ (f \otimes 1_Y)  &= (1_Y \otimes f) \circ \sigma_{X_1,Y},\\
    \sigma_{X, Y_1 \otimes Y_2} &= (1_{Y_1} \otimes \sigma_{X,Y_2}) \circ
    (\sigma_{X, Y_1} \otimes 1_{Y_2}),&
    \sigma_{X,Y_2} \circ (1_X \otimes g) &= (g \otimes 1_X)\circ \sigma_{X,Y_1}
\end{align*}
for all $X, X_1,X_2 \in\blue{\HEIS_l(z,u)}$, $Y, Y_1,Y_2 \in
\red{\HEIS_m(z,v)}$ and $f:X_1\rightarrow X_2$, $g:Y_1\rightarrow Y_2$.
Morphisms in 
$\blue{\HEIS_l(z,u)} \odot \red{\HEIS_m(z,v)}$ are linear combinations of diagrams
colored both blue and red. In these diagrams,
as well as the generating morphisms
of
$\blue{\HEIS_l(z,u)}$ and $\red{\HEIS_m(z,v)}$, 
we have the additional two-color crossings
$$
\mathord{
\begin{tikzpicture}[baseline = 0]
	\draw[->,thin,red] (0.28,-.3) to (-0.28,.4);
	\draw[thin,->,blue] (-0.28,-.3) to (0.28,.4);
\end{tikzpicture}
}\:,\qquad
\mathord{
\begin{tikzpicture}[baseline = 0]
	\draw[<-,thin,red] (0.28,-.3) to (-0.28,.4);
	\draw[thin,->,blue] (-0.28,-.3) to (0.28,.4);
\end{tikzpicture}
}\:,\qquad
\mathord{
\begin{tikzpicture}[baseline = 0]
	\draw[->,thin,red] (0.28,-.3) to (-0.28,.4);
	\draw[thin,<-,blue] (-0.28,-.3) to (0.28,.4);
\end{tikzpicture}
}\:,\qquad
\mathord{
\begin{tikzpicture}[baseline = 0]
	\draw[<-,thin,red] (0.28,-.3) to (-0.28,.4);
	\draw[thin,<-,blue] (-0.28,-.3) to (0.28,.4);
\end{tikzpicture}\:,
}\
$$
which represent the isomorphisms $\sigma_{X,Y}$ for $X \in
\{\blue{\up}, \blue{\down}\}$ and $Y \in \{\red{\up}, \red{\down}\}$,
and their inverses
$$
\mathord{
\begin{tikzpicture}[baseline = 0]
	\draw[->,thin,blue] (0.28,-.3) to (-0.28,.4);
	\draw[thin,->,red] (-0.28,-.3) to (0.28,.4);
\end{tikzpicture}
}\:,\qquad
\mathord{
\begin{tikzpicture}[baseline = 0]
	\draw[->,thin,blue] (0.28,-.3) to (-0.28,.4);
	\draw[thin,<-,red] (-0.28,-.3) to (0.28,.4);
\end{tikzpicture}
}\:,\qquad
\mathord{
\begin{tikzpicture}[baseline = 0]
	\draw[<-,thin,blue] (0.28,-.3) to (-0.28,.4);
	\draw[thin,->,red] (-0.28,-.3) to (0.28,.4);
\end{tikzpicture}
}\:,\qquad
\mathord{
\begin{tikzpicture}[baseline = 0]
	\draw[<-,thin,blue] (0.28,-.3) to (-0.28,.4);
	\draw[thin,<-,red] (-0.28,-.3) to (0.28,.4);
\end{tikzpicture}
}\:.
$$

\begin{definition}\label{loca}
Given a diagram $D$ representing a morphism in
$\blue{\HEIS_l(z,u)}\odot\red{\HEIS_m(z,v)}$
and two generic 
points in this diagram, one on a red string and the other on a blue
string, we will denote the morphism represented by
$$\text{($D$ with an extra 
dot at the red point)
$-$ ($D$ with an extra dot at the blue point)}
$$
by labelling the points with dots joined by
a dotted line. 
For example:
\begin{equation}
\mathord{
\begin{tikzpicture}[baseline = -1mm]
 	\draw[->,thin,red] (0.18,-.4) to (0.18,.4);
	\draw[->,thin,blue] (-0.38,-.4) to (-0.38,.4);
	\draw[-,dotted] (-0.38,0.01) to (0.18,0.01);
     \node at (0.18,0) {$\dot$};
     \node at (-0.38,0) {$\dot$};
\end{tikzpicture}
}:= \:\;
\mathord{
\begin{tikzpicture}[baseline = -1mm]
 	\draw[->,thin,red] (0.18,-.4) to (0.18,.4);
	\draw[->,thin,blue] (-0.38,-.4) to (-0.38,.4);
     \node at (0.18,0) {$\red{\dot}$};
\end{tikzpicture}
}-
\mathord{
\begin{tikzpicture}[baseline = -1mm]
 	\draw[->,thin,red] (0.18,-.4) to (0.18,.4);
	\draw[->,thin,blue] (-0.38,-.4) to (-0.38,.4);
     \node at (-0.38,0) {$\blue{\dot}$};
\end{tikzpicture}}
\:.
\end{equation}
Let 
$\blue{\HEIS_l(z,u)} \;\overline{\odot}\; \red{\HEIS_m(z,v)}$ be the strict $\k$-linear
monoidal category obtained by localizing at 
$\mathord{
\begin{tikzpicture}[baseline = -1mm]
 	\draw[->,thin,red] (0.18,-.2) to (0.18,.25);
	\draw[->,thin,blue] (-0.38,-.2) to (-0.38,.25);
	\draw[-,dotted] (-0.38,0.01) to (0.18,0.01);
     \node at (0.18,0) {$\dot$};
     \node at (-0.38,0.01) {$\dot$};
\end{tikzpicture}}
$.
This means that we adjoin a two-sided inverse to this morphism, which
we denote as a dumbbell
\begin{equation}
\mathord{
\begin{tikzpicture}[baseline = -1mm]
 	\draw[->,thin,red] (0.18,-.4) to (0.18,.4);
	\draw[->,thin,blue] (-0.38,-.4) to (-0.38,.4);
	\draw[-] (-0.38,0.01) to (0.18,0.01);
     \node at (0.18,0) {$\dot$};
     \node at (-0.38,0) {$\dot$};
\end{tikzpicture}
}:=
\left(\mathord{
\begin{tikzpicture}[baseline = -1mm]
 	\draw[->,thin,red] (0.18,-.4) to (0.18,.4);
	\draw[->,thin,blue] (-0.38,-.4) to (-0.38,.4);
	\draw[-,dotted] (-0.38,0.01) to (0.18,0.01);
     \node at (0.18,0) {$\dot$};
     \node at (-0.38,0) {$\dot$};
\end{tikzpicture}
}\right)^{-1}.
\end{equation}
Just 
as explained in the degenerate case in \cite[$\S\S$4--5]{BSW1}, all morphisms
whose string diagram is that of an
identity morphism with a horizontal dotted line joining two points of
different colors are also automatically invertible in the localized
category. We also denote the
inverses of such morphisms by using a solid dumbbell in place of the
dotted  one.
For instance:
$$
\mathord{
\begin{tikzpicture}[baseline = -1mm]
 	\draw[->,thin,red] (0.86,-.4) to (0.86,.4);
 	\draw[<-,thin,blue] (0.38,-.4) to (0.38,.4);
\draw[<-,thin,red] (-.1,-.4) to (-.1,.4);
	\draw[->,thin,blue] (-0.58,-.4) to (-0.58,.4);
	\draw[->,thin,blue] (-1.06,-.4) to (-1.06,.4);
	\draw[-] (-1.06,0.01) to (0.86,0.01);
     \node at (0.86,0) {$\dot$};
     \node at (-1.06,0) {$\dot$};
\end{tikzpicture}
}
=
\mathord{
\begin{tikzpicture}[baseline = -1mm]
 	\draw[-,thin,red] (0.86,-.8) to[out=90,in=-90] (-0.1,0);
 	\draw[->,thin,red] (-0.1,-0) to[out=90,in=-90] (0.86,.8);
 	\draw[<-,thin,blue] (0.38,-.8) to[out=90,in=-90] (0.86,0);
 	\draw[-,thin,blue] (0.86,0) to[out=90,in=-90] (0.38,.8);
\draw[<-,thin,red] (-.1,-.8) to [out=90,in=-90] (-1.06,0);
\draw[-,thin,red] (-1.06,0) to[out=90,in=-90] (-.1,.8);
	\draw[->,thin,blue] (0.38,0) to[out=90,in=-90] (-0.58,.8);
	\draw[-,thin,blue] (-0.58,-.8) to[out=90,in=-90] (0.38,0);
	\draw[->,thin,blue] (-.58,0) to [out=90,in=-90] (-1.06,.8);
	\draw[-,thin,blue] (-1.06,-.8) to [in=-90,out=90] (-.58,0);
	\draw[-] (-.58,0.01) to (-0.1,0.01);
     \node at (-0.58,0) {$\dot$};
     \node at (-.1,0) {$\dot$};
\end{tikzpicture}
}=
\left(\:
\mathord{
\begin{tikzpicture}[baseline = -1mm]
 	\draw[-,thin,red] (0.86,-.8) to[out=90,in=-90] (-0.1,0);
 	\draw[->,thin,red] (-0.1,-0) to[out=90,in=-90] (0.86,.8);
 	\draw[<-,thin,blue] (0.38,-.8) to[out=90,in=-90] (0.86,0);
 	\draw[-,thin,blue] (0.86,0) to[out=90,in=-90] (0.38,.8);
\draw[<-,thin,red] (-.1,-.8) to [out=90,in=-90] (-1.06,0);
\draw[-,thin,red] (-1.06,0) to[out=90,in=-90] (-.1,.8);
	\draw[->,thin,blue] (0.38,0) to[out=90,in=-90] (-0.58,.8);
	\draw[-,thin,blue] (-0.58,-.8) to[out=90,in=-90] (0.38,0);
	\draw[->,thin,blue] (-.58,0) to [out=90,in=-90] (-1.06,.8);
	\draw[-,thin,blue] (-1.06,-.8) to [in=-90,out=90] (-.58,0);
	\draw[-,dotted] (-.58,0.01) to (-0.1,0.01);
     \node at (-0.58,0) {$\dot$};
     \node at (-.1,0) {$\dot$};
\end{tikzpicture}
}\:
\right)^{-1} =
\left(\:
\mathord{
\begin{tikzpicture}[baseline = -1mm]
 	\draw[->,thin,red] (0.86,-.4) to (0.86,.4);
 	\draw[<-,thin,blue] (0.38,-.4) to (0.38,.4);
\draw[<-,thin,red] (-.1,-.4) to (-.1,.4);
	\draw[->,thin,blue] (-0.58,-.4) to (-0.58,.4);
	\draw[->,thin,blue] (-1.06,-.4) to (-1.06,.4);
	\draw[-,dotted] (-1.06,0.01) to (0.86,0.01);
     \node at (0.86,0) {$\dot$};
     \node at (-1.06,0) {$\dot$};
\end{tikzpicture}
}
\:\right)^{-1}\:.
$$
We also need the following morphisms, which we refer to as {\em internal bubbles}:
\begin{align}
\mathord{
\begin{tikzpicture}[baseline=2mm]
	\draw[->,thin,red] (0,-.2) to[out=90,in=-90] (0,.8);
     \node at (0,.3) {$\color{blue}\anticlockleft$};
\end{tikzpicture}
}&:=
z\sum_{a \geq 0}
\mathord{
\begin{tikzpicture}[baseline=2mm]
	\draw[->,thin,red] (0,-.2) to[out=90,in=-90] (0,.8);
     \node at (-.85,.3) {$\color{blue}\anticlockplus$};
     \node at (0,0.3) {$\color{red}\dot$};
     \node at (0.2,0.3) {$\color{red}\scriptstyle a$};
     \node at (-0.46,0.3) {$\color{blue}\scriptstyle -a$};
\end{tikzpicture}
}+
z\mathord{
\begin{tikzpicture}[baseline=2mm]
	\draw[->,thin,red] (0,-.2) to[out=90,in=-90] (0,.8);
     \node at (-.6,.3) {$\color{blue}\anticlockleft$};
	\draw[-] (0,0.31) to (-0.4,0.31);
     \node at (0,0.3) {$\dot$};
     \node at (-0.4,0.3) {$\dot$};
     \node at (-0.48,0.45) {$\color{blue}\dot$};
\end{tikzpicture}
}
\:,
&
\mathord{
\begin{tikzpicture}[baseline=2mm]
	\draw[->,thin,red] (0,-.2) to[out=90,in=-90] (0,.8);
     \node at (0,.3) {$\color{blue}\clockright$};
\end{tikzpicture}
}
&:=
z\sum_{a \geq 0}
\mathord{
\begin{tikzpicture}[baseline=2mm]
	\draw[->,thin,red] (0,-.2) to[out=90,in=-90] (0,.8);
     \node at (.85,.3) {$\blue{\clockplus}$};
     \node at (0,0.3) {$\color{red}\dot$};
     \node at (-0.2,0.3) {$\color{red}\scriptstyle a$};
     \node at (0.46,0.3) {$\color{blue}\scriptstyle -a$};
\end{tikzpicture}
}+
z\mathord{
\begin{tikzpicture}[baseline=2mm]
	\draw[->,thin,red] (0,-.2) to[out=90,in=-90] (0,.8);
     \node at (.6,.3) {$\color{blue}\clockright$};
	\draw[-] (0,0.31) to (0.4,0.31);
     \node at (0,0.3) {$\dot$};
     \node at (0.4,0.3) {$\dot$};
     \node at (0.53,0.45) {$\color{blue}\dot$};
\end{tikzpicture}
}\:,\label{odd1}\\\mathord{
\begin{tikzpicture}[baseline=2mm]
	\draw[->,thin,blue] (0,-.2) to[out=90,in=-90] (0,.8);
     \node at (0,.3) {$\color{red}\anticlockleft$};
\end{tikzpicture}
}&:=
z\sum_{a \geq 0}
\mathord{
\begin{tikzpicture}[baseline=2mm]
	\draw[->,thin,blue] (0,-.2) to[out=90,in=-90] (0,.8);
     \node at (-.85,.3) {$\color{red}\anticlockplus$};
     \node at (0,0.3) {$\color{blue}\dot$};
     \node at (0.2,0.3) {$\color{blue}\scriptstyle a$};
     \node at (-0.46,0.3) {$\color{red}\scriptstyle -a$};
\end{tikzpicture}
}-
z\mathord{
\begin{tikzpicture}[baseline=2mm]
	\draw[->,thin,blue] (0,-.2) to[out=90,in=-90] (0,.8);
     \node at (-.6,.3) {$\color{red}\anticlockleft$};
	\draw[-] (0,0.31) to (-0.4,0.31);
     \node at (0,0.3) {$\dot$};
     \node at (-0.4,0.3) {$\dot$};
     \node at (-0.48,0.45) {$\color{red}\dot$};
\end{tikzpicture}
}
\:,
&
\mathord{
\begin{tikzpicture}[baseline=2mm]
	\draw[->,thin,blue] (0,-.2) to[out=90,in=-90] (0,.8);
     \node at (0,.3) {$\color{red}\clockright$};
\end{tikzpicture}
}
&:=
z\sum_{a \geq 0}
\mathord{
\begin{tikzpicture}[baseline=2mm]
	\draw[->,thin,blue] (0,-.2) to[out=90,in=-90] (0,.8);
     \node at (.85,.3) {$\red{\clockplus}$};
     \node at (0,0.3) {$\color{blue}\dot$};
     \node at (-0.2,0.3) {$\color{blue}\scriptstyle a$};
     \node at (0.46,0.3) {$\color{red}\scriptstyle -a$};
\end{tikzpicture}
}-
z\mathord{
\begin{tikzpicture}[baseline=2mm]
	\draw[->,thin,blue] (0,-.2) to[out=90,in=-90] (0,.8);
     \node at (.6,.3) {$\color{red}\clockright$};
	\draw[-] (0,0.31) to (0.4,0.31);
     \node at (0,0.3) {$\dot$};
     \node at (0.4,0.3) {$\dot$};
     \node at (0.53,0.45) {$\color{red}\dot$};
\end{tikzpicture}
}\:.\label{odd2}\end{align}
\end{definition}

The category
$\blue{\HEIS_l(z,u)} \;\overline{\odot}\; \red{\HEIS_m(z,v)}$
possesses various symmetries which are often useful.
Derived from (\ref{om}), we have the strict $\k$-linear monoidal isomorphism
\begin{equation}\label{OMEGA}
\Omega_{\blue{l}|\red{m}}:
\blue{\HEIS_l(z,u)}\;\overline{\odot}\;\red{\HEIS_m(z,v)}\stackrel{\sim}{\rightarrow}
\left(\blue{\HEIS_{-l}(z,u^{-1})}\;\overline{\odot}\;\red{\HEIS_{-m}(z,v^{-1})}\right)^{\operatorname{op}},
\end{equation}
which takes a diagram to its mirror image in a horizontal plane
multiplied by
$(-1)^{x+y}$ where $x$ is the number of one-colored crossings and $y$
is the number of leftward cups and caps (including ones in $(+)$-, $(-)$- and
internal bubbles).
Also, 
we have
\begin{equation}
\eta
:
\blue{\HEIS_l(z,u)}\;\overline{\odot}\;\red{\HEIS_m(z,v)}
\stackrel{\sim}{\rightarrow}
\blue{\HEIS_m(z,v)}\;\overline{\odot}\;\red{\HEIS_l(z,u)}\label{ETA}
\end{equation}
defined on diagrams by switching the colors blue and red
then multiplying by $(-1)^z$ where $z$ is the total 
number of dumbbells (both solid and dotted) in the picture.
Finally, the category
$\blue{\HEIS_l(z,u)}\;\overline{\odot}\;\red{\HEIS_m(z,v)}$
is strictly pivotal, with duality functor
\begin{equation}\label{spiv2}
*: \blue{\HEIS_l(z,u)}\;\overline{\odot}\;\red{\HEIS_m(z,v)}\stackrel{\sim}{\rightarrow}
\left(\left(\blue{\HEIS_l(z,u)}\;\overline{\odot}\;\red{\HEIS_m(z,v)}\right)^{\operatorname{op}}\right)^{\operatorname{rev}}
\end{equation}
defined by rotating
diagrams through $180^\circ$ just like in (\ref{spiv}).

We denote
the duals of the internal bubbles
(\ref{odd1})--(\ref{odd2}) by
\begin{equation*}
\mathord{

}\:.
\end{equation*}
This definition ensures that internal bubbles commute past cups
and caps in all possible configurations. 
For example:
\begin{align*}
\mathord{
%
}
\:.
\end{align*}
Again as in \cite[$\S\S$4--5]{BSW1},
there are many other obvious commuting relations, such as
\begin{align*}
\mathord{
%
}
\:,
\end{align*}
as well as the mirror images of these under the symmetries $\Omega_{\blue{l}|\red{m}}, \eta$ and $*$.
We will appeal to all such relations below without further mention.

Here are some more interesting relations.
The first shows how to ``teleport'' dots across dumbbells (plus a correction term):
\begin{equation}\label{teleporting}
\mathord{
%
}
\end{equation}
for any $a \in \Z$.
We also have the following relations to commute dumbbells past
one-color crossings:
\begin{align}
\mathord{
%
}
\:.\label{c4}
\end{align}
These are all straightforward to prove: one first cancels the solid dumbbells
by composing  on the top and
bottom with their inverses then uses the affine
Hecke algebra relations (\ref{rr3})--(\ref{rr2}) to commute 
dots past crossings in the result.  For example, to prove the first relation in \eqref{c1}, we have
\[
    \mathord{
    %
    }\ .
\]
We then compose on the top with a solid dumbbell connecting the red strand and the leftmost blue strand, and compose on the bottom with a solid dumbbell connecting the red strand and the rightmost blue strand.

The following seven lemmas are the quantum analogs of \cite[Lemmas
  5.6--5.12]{BSW1}.
Their proofs are quite similar to the degenerate case.

\begin{lemma}\label{l1}
We have that
\[
\mathord{
%
}
\right)^{-1}.
\]
\end{lemma}

\begin{lemma}\label{l2}
For any $a \in \Z$, we have that
\[
\mathord{
%
}
.
\]
\end{lemma}

\begin{lemma}\label{l3a}
The following relations hold:
\begin{align*}
\mathord{
%
}
\:.
\end{align*}
\end{lemma}

\begin{lemma}\label{l4}
We have that
\[
\mathord{
%
}.
\]
\end{lemma}

\begin{lemma}\label{l6}
We have that
\[
\mathord{
%
}
.
\]
\end{lemma}

\begin{lemma}\label{l5}
We have that 
\[
\mathord{
%
}\:
.
\]
\end{lemma}

\begin{lemma}\label{l3b}
We have that 
\[
\mathord{
%
}.
\]
\end{lemma}

Using these, we can prove the main theorem of the section:

\begin{theorem}\label{comult}
For $k = \blue{l}+\red{m}$ and $t = \blue{u}\red{v}$, there is a unique strict $\k$-linear monoidal functor
$$\Delta_{\blue{l}|\red{m}}:\HEIS_k(z,t) \rightarrow
\Add\left(\blue{\HEIS_l(z,u)} \;\overline{\odot}\; \red{\HEIS_m(z,v)}\right)$$
such that $\up \mapsto \blue{\up} \oplus \red{\up}$,
$\down \mapsto \blue{\down}\oplus\red{\down}$, and on morphisms
\begin{align}\label{com1}
\mathord{
%
}\:.
\end{align}
Moreover, 
we have that
\begin{align}
\label{com4}
\Delta_{\blue{l}|\red{m}}\left(\:%
}\:.\\\intertext{Also, the following hold for all $a \in \Z$:}
\label{com5}
\Delta_{\blue{l}|\red{m}}\left(\mathord{%
}\:.\\\intertext{Equivalently, in terms of the generating functions (\ref{fourofem})--(\ref{igm})
and their analogs in $\blue{\HEIS_l(z,u)}$ and $\red{\HEIS_m(z,v)}$:}
\label{dt1}
\Delta_{\blue{l}|\red{m}}\left(\anticlockplus\,(w)\right) &= 
                                         \blue{\anticlockplus\,}(w)\;\red{\anticlockplus\,}(w),
&\Delta_{\blue{l}|\red{m}}\left(\clockplus\,(w)\right) &= 
\blue{\clockplus\,}(w)\red{\clockplus\,}(w),
\\\label{dt2}
\Delta_{\blue{l}|\red{m}}\left(\anticlockminus\,(w)\right) &= 
                                                             \blue{\anticlockminus\,}(w)\;\red{\anticlockminus\,}(w),
&
\Delta_{\blue{l}|\red{m}}\left(\clockminus\,(w)\right) &= 
                                                 \blue{\clockminus\,}(w)\red{\clockminus\,}(w).
\end{align}
\end{theorem}

\begin{remark}
For the proof, it is helpful to notice that
$\eta \circ \Delta_{\blue{l}|\red{m}} =
\Delta_{\blue{m}|\red{l}}$
(on extending $\eta$ to the additive
envelopes
in the obvious way).
However, $\Delta_{\blue{l}|\red{m}}$ 
does not commute with either of the other symmetries
$\Omega$ or $*$.
In fact, the map $\Omega_{\blue{-l}|\red{-m}} 
\circ \Delta_{\blue{-l}|\red{-m}} \circ  \Omega_k$ 
would be an equally good 
alternative
choice for the categorical comultiplication map.
The only change to the above
formulae if one uses this alternative is that one needs to replace $q$
with $-q^{-1}$ in
(\ref{com2})--(\ref{com2b}); this is the ``Galois symmetry'' 
in the choice of the root $q$ of the equation $x^2-zx-1=0$.
\end{remark}

\begin{proof}
In view of the uniqueness from Lemma~\ref{lateaddition}, we may take (\ref{com1})--(\ref{com4}) as the definition of $\Delta_{\blue{l}|\red{m}}$ on generating morphisms, and must check that the images of the relations (\ref{rr3})--(\ref{rightadj}) and 
(\ref{pos})--(\ref{morecurls}) from Definition~\ref{def3} are all satisfied in
$\Add\left(\blue{\HEIS_l(z,u)} \;\overline{\odot}\; \red{\HEIS_m(z,v)}\right)$;
 we must also check (\ref{com5})--(\ref{com6}).
The details are
sufficiently similar to the degenerate case from the proof of
\cite[Theorem 5.4]{BSW1} that we only sketch the steps needed below.

First one checks (\ref{rr3})--(\ref{rr0}).
For example, to check the skein relation, the image under $\Delta_{\blue{l}|\red{m}}$ of
$\mathord{
\begin{tikzpicture}[baseline = -.5mm]
	\draw[->,thin] (0.2,-.2) to (-0.2,.3);
	\draw[line width=4pt,-,white] (-0.2,-.2) to (0.2,.3);
	\draw[thin,->] (-0.2,-.2) to (0.2,.3);
\end{tikzpicture}
}
-\mathord{
\begin{tikzpicture}[baseline = -.5mm]
	\draw[thin,->] (-0.2,-.2) to (0.2,.3);
	\draw[line width=4pt,-,white] (0.2,-.2) to (-0.2,.3);
	\draw[->,thin] (0.2,-.2) to (-0.2,.3);
\end{tikzpicture}
}\:$
is $A+\eta(A)$ where
$$
A := \left(\mathord{
\begin{tikzpicture}[baseline = 0]
	\draw[->,thin,blue] (0.28,-.3) to (-0.28,.4);
	\draw[line width=4pt,-,white] (-0.28,-.3) to (0.28,.4);
	\draw[thin,->,blue] (-0.28,-.3) to (0.28,.4);
\end{tikzpicture}
}-
\mathord{
\begin{tikzpicture}[baseline = 0]
	\draw[thin,->,blue] (-0.28,-.3) to (0.28,.4);
	\draw[line width=4pt,-,white] (0.28,-.3) to (-0.28,.4);
	\draw[->,thin,blue] (0.28,-.3) to (-0.28,.4);
\end{tikzpicture}
}\right)
+z\left(\mathord{
\begin{tikzpicture}[baseline = -1mm]
 	\draw[->,thin,red] (0.2,-.4) to (0.2,.4);
	\draw[->,thin,blue] (-0.2,-.4) to (-0.2,.4);
	\draw[-] (-0.2,-0.15) to (0.2,-0.15);
     \node at (0.2,-0.16) {$\dot$};
     \node at (-0.2,-0.16) {$\dot$};
     \node at (0.2,0.1) {$\red{\dot}$};
\end{tikzpicture}
}-\mathord{
\begin{tikzpicture}[baseline = -1mm]
 	\draw[->,thin,red] (0.2,-.4) to (0.2,.4);
	\draw[->,thin,blue] (-0.2,-.4) to (-0.2,.4);
	\draw[-] (-0.2,-0.15) to (0.2,-0.15);
     \node at (0.2,-0.16) {$\dot$};
     \node at (-0.2,-0.16) {$\dot$};
     \node at (-0.2,0.1) {$\blue{\dot}$};
\end{tikzpicture}
}\right)
+z\left(\mathord{
\begin{tikzpicture}[baseline = 0]
	\draw[->,thin,red] (0.28,-.3) to (-0.28,.4);
	\draw[thin,->,blue] (-0.28,-.3) to (0.28,.4);
\end{tikzpicture}
}
+
\mathord{
\begin{tikzpicture}[baseline = 0]
	\draw[->,thin,red] (0.28,-.3) to (-0.28,.4);
	\draw[thin,->,blue] (-0.28,-.3) to (0.28,.4);
	\draw[-] (-0.15,-.14) to (0.15,-.14);
     \node at (0.15,-.15) {$\dot$};
     \node at (-0.15,-.15) {$\dot$};
     \node at (0.15,0.22) {$\blue{\dot}$};
\end{tikzpicture}
}-
\mathord{
\begin{tikzpicture}[baseline = 0]
	\draw[->,thin,red] (0.28,-.3) to (-0.28,.4);
	\draw[thin,->,blue] (-0.28,-.3) to (0.28,.4);
	\draw[-] (-0.15,-.14) to (0.15,-.14);
     \node at (0.15,-.15) {$\dot$};
     \node at (-0.15,-.15) {$\dot$};
     \node at (-0.15,0.22) {$\red{\dot}$};
\end{tikzpicture}
}\right)\:.
$$
Using the skein relation in $\blue{\HEIS_l(z,u)}$
plus (\ref{teleporting}), $A$ simplifies to 
$B := z\:\mathord{
\begin{tikzpicture}[baseline = -.5mm]
	\draw[->,thin,blue] (-0.15,-.2) to (-0.15,.3);
	\draw[thin,->,blue] (0.15,-.2) to (0.15,.3);
\end{tikzpicture}
}
\:+z\:
\mathord{
\begin{tikzpicture}[baseline = -.5mm]
	\draw[->,thin,blue] (-0.15,-.2) to (-0.15,.3);
	\draw[thin,->,red] (0.15,-.2) to (0.15,.3);
\end{tikzpicture}
}\:$. 
This is what is required since
the image under $\Delta_{\blue{l}|\red{m}}$ of $z \:\mathord{
\begin{tikzpicture}[baseline = -.5mm]
	\draw[->,thin] (-0.15,-.2) to (-0.15,.3);
	\draw[thin,->] (0.15,-.2) to (0.15,.3);
\end{tikzpicture}
}\:$ is
$B + \eta(B)$.
The other relations here are checked by similarly explicit
calculations. The one for the braid relation is rather long.

The relation (\ref{rightadj}) is easy.

To check (\ref{com5})--(\ref{com6}), 
we assume to start with that $k \geq 0$. Consider the clockwise $(+)$-bubble
$\mathord{
\begin{tikzpicture}[baseline = 1.25mm]
   \node at (0,0.2) {$\clockplus$};
   \node at (-0.3,0.2) {$\scriptstyle{a}$};
\end{tikzpicture}
}$.
When $a \leq 0$, this is just a scalar (usually zero) due to
(\ref{tanks}) and the
assumption $k \geq 0$,
and the relation
to be checked is trivial. So assume that $a > 0$. Then 
$\mathord{
\begin{tikzpicture}[baseline = 1.25mm]
   \node at (0,0.2) {$\clockplus$};
   \node at (-0.3,0.2) {$\scriptstyle{a}$};
\end{tikzpicture}
}
=
\mathord{
\begin{tikzpicture}[baseline = 1.25mm]
   \node at (0,0.2) {$\clockright$};
   \node at (-0.35,0.2) {$\scriptstyle{a}$};
   \node at (-0.2,0.2) {$\dot$};
\end{tikzpicture}
}
$, hence, its image under $\Delta_{\blue{l}|\red{m}}$
is 
$
-\displaystyle\mathord{
\begin{tikzpicture}[baseline=-1mm]
\draw[-,thin,blue] (0,-0.2) to[out=180,in=-90] (-.2,0);
\draw[-,thin,blue] (-0.2,0) to[out=90,in=180] (0,0.2);
\draw[->,thin,blue] (0,0.2) to[out=0,in=90] (0.2,0);
\draw[-,thin,blue] (0.2,0) to[out=-90,in=0] (0,-0.2);
     \node at (-.17,-.13) {$\color{red}\smallclock$};
     \node at (-.17,.12) {$\color{blue}\dot$};
     \node at (-.34,.13) {$\color{blue}\scriptstyle{a}$};
\end{tikzpicture}}
\:-\displaystyle\mathord{
\begin{tikzpicture}[baseline=-1mm]
\draw[-,thin,red] (0,-0.2) to[out=180,in=-90] (-.2,0);
\draw[-,thin,red] (-0.2,0) to[out=90,in=180] (0,0.2);
\draw[->,thin,red] (0,0.2) to[out=0,in=90] (0.2,0);
\draw[-,thin,red] (0.2,0) to[out=-90,in=0] (0,-0.2);
     \node at (-.17,-.13) {$\color{blue}\smallclock$};
     \node at (-.17,.12) {$\color{red}\dot$};
     \node at (-.34,.13) {$\color{red}\scriptstyle{a}$};
\end{tikzpicture}}
\:$,
which is indeed
equal to 
$-z\sum_{b \in \Z}
\mathord{
\begin{tikzpicture}[baseline = 1.25mm]
   \node at (0,0.2) {$\color{blue}\clockplus$};
   \node at (-.3,0.2) {$\color{blue}\scriptstyle{b}$};
\end{tikzpicture}
}
\mathord{
\begin{tikzpicture}[baseline = 1.25mm]
   \node at (0,0.2) {$\color{red}\clockplus$};
   \node at (-.45,0.2) {$\color{red}\scriptstyle{a-b}$};
\end{tikzpicture}
}
$ by Lemma~\ref{l2}.
This establishes the right-hand identity in (\ref{com5}), hence, the
right-hand identity in  (\ref{dt1}). The left-hand identity in
(\ref{dt1}) then follows using (\ref{igproper}),
thereby establishing the left-hand identity in (\ref{com5}) as well.
Next, consider the 
clockwise $(-)$-bubble
$\mathord{
\begin{tikzpicture}[baseline = 1.25mm]
   \node at (0,0.2) {$\clockminus$};
   \node at (-0.3,0.2) {$\scriptstyle{a}$};
\end{tikzpicture}
}$.
This time the relation to be checked is trivial when $a \geq 0$, so
assume that $a < 0$.
Then, using the assumption $k \geq 0$ again, we have that
$\mathord{
\begin{tikzpicture}[baseline = 1.25mm]
   \node at (0,0.2) {$\clockminus$};
   \node at (-0.3,0.2) {$\scriptstyle{a}$};
\end{tikzpicture}
}
=
\mathord{
\begin{tikzpicture}[baseline = 1.25mm]
   \node at (0,0.2) {$\clockright$};
   \node at (-0.35,0.2) {$\scriptstyle{a}$};
   \node at (-0.2,0.2) {$\dot$};
\end{tikzpicture}
}
$, hence, its image under $\Delta_{\blue{l}|\red{m}}$
is 
$
-\displaystyle\mathord{
\begin{tikzpicture}[baseline=-1mm]
\draw[-,thin,blue] (0,-0.2) to[out=180,in=-90] (-.2,0);
\draw[-,thin,blue] (-0.2,0) to[out=90,in=180] (0,0.2);
\draw[->,thin,blue] (0,0.2) to[out=0,in=90] (0.2,0);
\draw[-,thin,blue] (0.2,0) to[out=-90,in=0] (0,-0.2);
     \node at (-.17,-.13) {$\color{red}\smallclock$};
     \node at (-.17,.12) {$\color{blue}\dot$};
     \node at (-.34,.13) {$\color{blue}\scriptstyle{a}$};
\end{tikzpicture}}
\:-\displaystyle\mathord{
\begin{tikzpicture}[baseline=-1mm]
\draw[-,thin,red] (0,-0.2) to[out=180,in=-90] (-.2,0);
\draw[-,thin,red] (-0.2,0) to[out=90,in=180] (0,0.2);
\draw[->,thin,red] (0,0.2) to[out=0,in=90] (0.2,0);
\draw[-,thin,red] (0.2,0) to[out=-90,in=0] (0,-0.2);
     \node at (-.17,-.13) {$\color{blue}\smallclock$};
     \node at (-.17,.12) {$\color{red}\dot$};
     \node at (-.34,.13) {$\color{red}\scriptstyle{a}$};
\end{tikzpicture}}
\:$,
which is 
equal to 
$z\sum_{b \in \Z}
\mathord{
\begin{tikzpicture}[baseline = 1.25mm]
   \node at (0,0.2) {$\color{blue}\clockminus$};
   \node at (-.3,0.2) {$\color{blue}\scriptstyle{b}$};
\end{tikzpicture}
}
\mathord{
\begin{tikzpicture}[baseline = 1.25mm]
   \node at (0,0.2) {$\color{red}\clockminus$};
   \node at (-.45,0.2) {$\color{red}\scriptstyle{a-b}$};
\end{tikzpicture}
}
$ by Lemma~\ref{l2} (noting when $a < 0 \leq k$ that the term
involving $(+)$-bubbles is zero). Then we complete the proof of
(\ref{com6}) using the equivalent form (\ref{dt2}) and
(\ref{igproper}) once again.
It remains to treat $k \leq 0$. This follows by similar arguments; one starts by considering the counterclockwise
$(+)$- and $(-)$-bubbles using the identities obtained by applying $\Omega_{\blue{l}|\red{m}}$ to
Lemma~\ref{l2}, then gets the clockwise ones using
(\ref{igproper}).

Consider (\ref{curls})--(\ref{morecurls}).
The relations involving bubbles
follow easily from (\ref{com5})--(\ref{com6}).
Next consider the right curl relation in (\ref{curls}), so $k
\geq 0$. Applying $\Delta_{\blue{l}|\red{m}}$ to the relation reveals that
we must show that $A+\eta(A) = B+\eta(B)$ where
\begin{align*}
A &:= 
z\mathord{\begin{tikzpicture}[baseline=-1mm]
\draw[->,thin,blue] (-0.6,-0.4) to (-.6,0.4);
\draw[-,thin,red] (0,-0.2) to[out=180,in=-90] (-.2,0);
\draw[->,thin,red] (-0.2,0) to[out=90,in=180] (0,0.2);
\draw[-,thin,red] (0,0.2) to[out=0,in=90] (0.2,0);
\draw[-,thin,red] (0.2,0) to[out=-90,in=0] (0,-0.2);
     \node at (.2,0) {$\color{blue}\smallclock$};
	\draw[-] (-.2,0.01) to (-0.6,0.01);
     \node at (-0.2,0) {$\dot$};
     \node at (-0.6,0) {$\dot$};
     \node at (0,-.2) {$\red{\dot}$};
\end{tikzpicture}}
-
\mathord{
\begin{tikzpicture}[baseline = -1mm]
  \draw[-,thin,blue] (-0.45,-.4) to (-0.45,-.3)
        to[in=180,out=90] (-.2,.2) to[out=0,in=90] (0,0.05);
\draw[-,line width=4pt,white] (0,0.05)  to[out=-90,in=0] (-.2,-.2) to [out=180,in=-90] (-0.45,.3) to (-0.45,.4);
\draw[->,thin,blue] (0,0.05)  to[out=-90,in=0] (-.2,-.2) to [out=180,in=-90] (-0.45,.3) to (-0.45,.4);
   \node at (0,0) {$\color{red}\smallclock$};
\end{tikzpicture}
}\:,
&
B &:= \delta_{k,0} t^{-1}\:
\mathord{\begin{tikzpicture}[baseline=-1mm]
\draw[->,thin,blue] (0,-0.4) to (0,0.4);
\end{tikzpicture}}
\:.
\end{align*}
This follows from Lemma~\ref{l4}, 
noting that the only non-zero term in the
summation on the right-hand side of that
identity is the one with $a=b=0$ due to the assumption that $k \geq
0$.
The argument for the left curl in (\ref{morecurls}) is entirely similiar;
it uses the identity obtained by applying $* \circ \Omega_{\blue{l}|\red{m}}$ to Lemma~\ref{l4}.

Finally, one must check (\ref{pos})--(\ref{neg}).
This is a calculation just like in the final paragraph of the proof of
\cite[Theorem 5.4]{BSW1};
ultimately one uses Lemmas~\ref{l6}--\ref{l3b}.
\end{proof}

\section{Generalized cyclotomic quotients}\label{sgcq}

In this section, we define some $\k$-linear categories, namely,
the
{generalized cyclotomic quotients} of $\HEIS_k(z,t)$. 
Recall that
$x = 
\mathord{
\begin{tikzpicture}[baseline = -.5mm]
	\draw[->,thin,darkblue] (0.08,-.2) to (0.08,.3);
      \node at (0.08,0.02) {$\dot$};
\end{tikzpicture}
}$ and 
$y = 
\mathord{
\begin{tikzpicture}[baseline = -.5mm]
	\draw[<-,thin,darkblue] (0.08,-.2) to (0.08,.3);
      \node at (0.08,0.08) {$\dot$};
\end{tikzpicture}
}$.

\begin{definition}\label{gcq}
Suppose we are given monic polynomials
\begin{align}
f(w) &= f_0 w^l + f_1 w^{l-1}+\cdots+f_l \in \k[w],\label{fw}\\
g(w) &= g_0 w^m + g_1 w^{m-1}+\cdots+g_m \in \k[w]\label{gw}
\end{align}
such that $k=m-l$ and $t^2=f_l / g_m$.
Define
\begin{align}\label{am}
\A^+(w) &=t^{-1}z\sum_{n \in \Z} \A_n^+ w^{-n}:= g(w)/f(w) \in w^k + w^{k-1} \k\llbracket w^{-1}\rrbracket,\\
\B^+(w) &=- t z\sum_{n \in \Z} \B_n^+ w^{-n}:= f(w)/g(w) \in w^{-k} + w^{-k-1} \k\llbracket w^{-1}\rrbracket,\\
\A^-(w) &=- tz\sum_{n \in \Z} \A_n^- w^{-n}:= t^2 g(w)/f(w) \in 1 + w \k\llbracket w\rrbracket,\label{amm}\\
\B^-(w) &= t^{-1} z\sum_{n \in \Z} \B_n^- w^{-n}:= t^{-2} f(w)/g(w) \in 1 + w \k\llbracket w\rrbracket;\label{pm}
\end{align}
cf. (\ref{fourofem})--(\ref{igm}).
Let $\mathcal I(f|g)$ be the
left tensor ideal generated by the morphisms
\begin{equation}\label{gen1}
\Big\{
f(x),\:
\anticlockplus{\scriptstyle n} - \A_n^+ \unit
\:\Big|\:-k < n < l\Big\}.
\end{equation}
The {\em generalized cyclotomic quotient} associated
to the polynomials $f(w)$ and $g(w)$ is the quotient category
\begin{equation}
\H(f|g) := \HEIS_k(z,t) \big / \mathcal I(f|g).
\end{equation} 
It is a module category over $\HEIS_k(z,t)$.
\end{definition}

The following is the quantum analog of
\cite[Lemma 1.8]{Bheis}; see also \cite[Lemma 4.14]{BD} for the analog
in the setting of Kac-Moody 2-categories.

\begin{lemma}\label{clever}
In the setup of Definition~\ref{gcq}, 
$\mathcal I(f|g)$ may be defined equivalently as the left tensor
ideal generated by 
\begin{equation}\label{gen2}
\Big\{
g(y),\:
{\scriptstyle n}\clockplus - \B_n^+ \unit
\:\Big|\:k < n <m\Big\}.
\end{equation}
Moreover, it
contains 
$\: \anticlockplus{\scriptstyle{n}}-\A_n^+\unit,
\:\anticlockminus{\scriptstyle{n}}-\A_n^-\unit,
\:{\scriptstyle{n}}\clockplus-\B_n^+\unit$ and
$\:{\scriptstyle{n}}\clockminus-\B_n^-\unit$
for all $n \in \Z$.
\end{lemma}

\begin{proof}
For morphisms $\theta,\phi :X \rightarrow Y$,
we will write $\theta \equiv \phi$ as shorthand for
$\theta-\phi \in \mathcal I(f|g)$.
By (\ref{tanks})--(\ref{tanks2}), 
we have automatically that
$\anticlockplus{\scriptstyle{n}}\equiv \A_n^+1_\unit$ when $n \leq
-k$,
${\scriptstyle{n}}\clockplus\equiv \B_n^+1_\unit$ when $n \leq k$,
$\anticlockminus{\scriptstyle{n}}\equiv \A_n^-1_\unit$ when $n \geq 0$,
and ${\scriptstyle{n}}\clockminus\equiv \B_n^-1_\unit$
when $n \geq 0$.

In this paragraph, we use ascending induction on $n$ to show that 
$\anticlockplus{\scriptstyle{n}}\equiv \A_n^+1_\unit$
for all $n \in \Z$. 
This is immediate from (\ref{gen1}) if $n < l$, so assume that $n
\geq l$.
The fact that $f(x) \equiv 0$ implies that
\[
\sum_{a=0}^l f_{a} \:
\anticlockplus
{\scriptstyle n-a}
+
\sum_{a=0}^l f_{a} \:
\anticlockminus
{\scriptstyle n-a}
 =
\sum_{a=0}^l f_{a} \:
\begin{tikzpicture}[baseline=-.9mm]
\filldraw[white] (0,0) circle (1.72mm);
\draw[-,thin] (0,-0.18) to[out=180,in=-102] (-.178,0.02);
\draw[<-,thin] (-0.18,0) to[out=90,in=180] (0,0.18);
\draw[-,thin] (0.18,0) to[out=-90,in=0] (0,-0.18);
\draw[-,thin] (0,0.18) to[out=0,in=90] (0.18,0);
\node at (0.18,0) {$\dot$};
\end{tikzpicture}
\!{\scriptstyle n-a}\equiv 0.
\]
On the left-hand side of this,
the only non-zero 
$(-)$-bubble arises when $n=a=l$, so
it shows that
$\sum_{a=0}^l f_{a} \:\anticlockplus
{\scriptstyle n-a} \equiv\delta_{l,n} f_l t^{-1} z^{-1}1_\unit$.
Using the induction hypothesis and $f_l = g_m t^2$, we deduce that
$\anticlockplus
{\scriptstyle n}
+\sum_{a=1}^{l} f_{a}
\A^+_{n-a} 1_\unit
\equiv\delta_{l,n} g_m t z^{-1}1_\unit.
$
Equating $w^{l-n}$-coefficients in 
$f(w) \A^+(w) =g(w)$, we get that
$\sum_{a=0}^{l} f_a \A_{n-a}^+ = \delta_{l,n} g_m t z^{-1}$.
Hence,
$\anticlockplus
{\scriptstyle n} \equiv \A_n^+ 1_\unit$ as claimed.

Next, we show by descending induction on $n$ that
$\:\anticlockminus{\scriptstyle n}\,\equiv \A_n^-1_\unit$
for all $n \in \Z$.
We may assume that $n < 0$.
Equating $w^{-n}$-coefficients in 
$f(w) \A^+(w) = t^{-2} f(w) \A^-(w)$
gives that
\[
    \sum_{a=0}^l f_{l-a} \A^+_{a+n}=-\sum_{a=0}^l f_{l-a} \A_{a+n}^-.
\]
Using the induction hypothesis plus the previous paragraph, we deduce 
that
\[
    \sum_{a=0}^l f_{l-a} \anticlockplus\,{\scriptstyle a+n}  + f_l \A_n^- + \sum_{a=1}^l f_{l-a} \anticlockminus\,{\scriptstyle  a+n}\equiv 0.
\]
But also from $f(x) \equiv 0$ we get that
\[
\sum_{a=0}^l f_{l-a} \:
\anticlockplus
{\scriptstyle a+n}
+
\sum_{a=0}^l f_{l-a} \:
\anticlockminus
{\scriptstyle a+n}
=
\sum_{a=0}^l f_{l-a} \:
\begin{tikzpicture}[baseline=-.5mm]
\filldraw[white] (0,0) circle (1.72mm);
\draw[-,thin] (0,-0.18) to[out=180,in=-102] (-.178,0.02);
\draw[<-,thin] (-0.18,0) to[out=90,in=180] (0,0.18);
\draw[-,thin] (0.18,0) to[out=-90,in=0] (0,-0.18);
\draw[-,thin] (0,0.18) to[out=0,in=90] (0.18,0);
\node at (0.18,0) {$\dot$};
\end{tikzpicture}
\!{\scriptstyle a+n}\equiv 0.
\]
Taking the difference of these two identities establishes the induction step.

Using the notation of (\ref{fourofem})--(\ref{igm}),
we have now shown that $\anticlockplusminus\,(w)\equiv \A^{\pm}(w) 1_\unit$.
Taking inverses using (\ref{igproper}), we deduce that
$\clockplusminus\,(w) \equiv \B^{\pm}(w) 1_\unit$. Hence,
${\scriptstyle n}\clockplusminus
\equiv \B_n^\pm 1_\unit$ 
for all $n \in \Z$.
So we have established the last assertion from the lemma.

Equating $w^b$-coefficients in 
$g(w) = f(w) \A^+(w)$ shows that $g_{m-b} =
t^{-1}z\sum_{a=0}^{l} f_{l-a} \A^+_{a-b}$. Hence:
\begin{align*}
g(y) = 
\sum_{a=0}^l t^{-1}f_{l-a}\left( z\sum_{b\geq 0} 
\begin{tikzpicture}[baseline = 0mm]
	\draw[<-,thin,darkblue] (0,-0.35) to (0,0.4);
   \node at (-0.2,0.05) {$\color{darkblue}\scriptstyle{b}$};
      \node at (0,.05) {$\dot$};
\end{tikzpicture}
\anticlockplus{\scriptstyle a-b} \right)
\stackrel{(\ref{dog1})}{=} \sum_{a=0}^l t^{-1}f_{l-a} 
\mathord{
\begin{tikzpicture}[baseline = -0.5mm]
	\draw[-,thin,darkblue] (0,0.6) to (0,0.3);
	\draw[-,thin,darkblue] (0,0.3) to [out=-90,in=180] (.3,-0.2);
	\draw[-,thin,darkblue] (0.3,-0.2) to [out=0,in=-90](.5,0);
	\draw[-,thin,darkblue] (0.5,0) to [out=90,in=0](.3,0.2);
	\draw[->,thin,darkblue] (0,-0.3) to (0,-0.6);
	\draw[-,line width=4pt,white] (0.3,.2) to [out=180,in=90](0,-0.3);
	\draw[-,thin,darkblue] (0.3,.2) to [out=180,in=90](0,-0.3);
   \node at (0.7,0.0) {$\color{darkblue}\scriptstyle{a}$};
      \node at (0.5,.0) {$\dot$};
\end{tikzpicture}
} \equiv 0.
\end{align*}
We have now shown that $\mathcal I(f|g)$, the left tensor ideal
generated by (\ref{gen1}), contains (\ref{gen2}). 
Similarly, the left tensor ideal generated by (\ref{gen2})
contains (\ref{gen1}). This completes the proof.
\end{proof}

We assume for the the rest of the section that $\k$ is a field,
and that we are given a factorization
$t=uv^{-1}$ for $u, v \in \k^\times$ such that $u^2=f_l$ and $v^2 = g_m$.
Let $\mathcal{V}(f)$ and $\mathcal{V}(g)^\vee$
denote $\bigoplus_{n \geq 0} H_n^f\proj$ 
and $\bigoplus_{n \geq 0} H_n^g\proj$ viewed as module categories over $\blue{\HEIS_{-l}(z,u)}$ 
and $\red{\HEIS_m(z,v^{-1})}$ via the monoidal functors $\Psi_f$ and $\Psi_g^\vee$ from Lemma~\ref{advising1}.
Let 
\begin{equation} \label{piglet}
\mathcal{V}(f|g) := \mathcal{V}(f)\boxtimes \mathcal{V}(g)^\vee
\end{equation}
be their linearized 
Cartesian
product, i.e., the $\k$-linear category with objects that are pairs
$(X,Y)$ for $X \in \mathcal V(f), Y \in \mathcal V(g)^\vee$, and morphisms
$$
\Hom_{\mathcal V(f|g)}((X,Y),(U,V)) := \Hom_{\mathcal
  V(f)}(X,U) \otimes \Hom_{\mathcal V(g)^\vee}(Y,V)
$$ 
with the obvious 
composition law.
There is an equivalence of categories
$$
\mathcal{V}(f|g)    \rightarrow
\bigoplus_{r,s \geq 0} \left(H_r^f \otimes H_s^g\right)\proj,
$$
hence, $\mathcal{V}(f|g)$ is additive Karoubian.
Moreover, $\mathcal{V}(f|g)$
is a module
category over the symmetric product
$\blue{\HEIS_{-l}(z,u)} \odot \red{\HEIS_m(z,v^{-1})}$.

\begin{lemma}\label{stupid}
Let $V$ be a finite-dimensional $AH_2$-module.
All eigenvalues of $x_2$ on $V$ are of the form 
$\lambda, q^2 \lambda$ or 
$q^{-2} \lambda$ for eigenvalues $\lambda$ of $x_1$ on $V$.
\end{lemma}

\begin{proof}
We may assume for the proof that $\k$ is algebraically closed. Suppose that
$v \in V$ is a simultaneous eigenvector for the commuting operators 
$x_1$ and $x_2$ of
eigenvalues $\lambda_1$ and $\lambda_2$, respectively.
If $\tau_1 v = q v$ (resp.\ $\tau_1 v = -q^{-1} v$)
then $\lambda_2 = q^2 \lambda_1$ (resp.\ $\lambda_2 = q^{-2}
\lambda_1$), as 
follows easily from the relation $x_2 (\tau_1-z) v = \tau_1 x_1 v$.
Otherwise, $v$ and $\tau_1 v$ are linearly independent, in which case
the matrix describing the action of
$x_1$ on the subspace with basis $\{v, \tau_1 v\}$ is
$\left(\begin{array}{rr}\lambda_1&-z\lambda_2\\0&\lambda_2\end{array}\right).$
So $\lambda_2$ is another eigenvalue of $x_1$ on $V$. 
\end{proof}

\begin{lemma}\label{starbucks}
Assume that $f(w)$ and $g(w)$ split
as products of linear factors in $\k[w]$, and moreover assume that $\lambda
\mu^{-1} \notin \left\{q^{2i}\:\big|\:i \in \Z\right\}$ for all roots $\lambda$ of $f(w)$
and $\mu$ of $g(w)$. Then the categorical action of 
$\blue{\HEIS_{-l}(z,u)} \odot \red{\HEIS_m(z,v^{-1})}$ on $\mathcal V(f|g)$ defined above extends to an action of the localization
$\blue{\HEIS_{-l}(z,u)} \;\overline{\odot}\; \red{\HEIS_m(z,v^{-1})}$ from Definition~\ref{loca}.
\end{lemma}

\begin{proof}
Lemma~\ref{stupid} implies that
the eigenvalues of $x_1,\dots,x_n$ on any
finite-dimensional $H_n^f$-module are of the form $q^{2i} \lambda$ for
$i \in \Z$ and a root $\lambda$ of $f(w)$.
Consequently, the commuting endomorphisms 
defined by evaluating 
$\:\mathord{
\begin{tikzpicture}[baseline = -1mm]
 	\draw[->,thin,red] (0.18,-.2) to (0.18,.25);
	\draw[->,thin,blue] (-0.18,-.2) to (-0.18,.25);
     \node at (0.18,0) {$\red{\dot}$};
\end{tikzpicture}}
$
and
$\mathord{
\begin{tikzpicture}[baseline = -1mm]
 	\draw[->,thin,red] (0.18,-.2) to (0.18,.25);
	\draw[->,thin,blue] (-0.18,-.2) to (-0.18,.25);
     \node at (-0.18,0) {$\blue{\dot}$};
\end{tikzpicture}}\:
$
on an object of
$\mathcal{V}(f|g)$ have eigenvalues contained in the sets
$\left\{q^{2i} \lambda\:\big|\:i \in \Z, \lambda\text{ a root of }f(w)\right\}$
and
$\left\{q^{2j} \mu\:\big|\:j \in \Z, \mu\text{ a root of
  }g(w)\right\}$,
respectively.
By the genericity assumption, these sets are 
disjoint, hence, all eigenvalues of the endomorphism 
defined by
$\mathord{
\begin{tikzpicture}[baseline = -1mm]
 	\draw[->,thin,red] (0.18,-.2) to (0.18,.25);
	\draw[->,thin,blue] (-0.38,-.2) to (-0.38,.25);
	\draw[-,dotted] (-0.38,0.01) to (0.18,0.01);
     \node at (0.18,0) {$\dot$};
     \node at (-0.38,0.01) {$\dot$};
\end{tikzpicture}}=
\: \mathord{
\begin{tikzpicture}[baseline = -1mm]
 	\draw[->,thin,red] (0.18,-.2) to (0.18,.25);
	\draw[->,thin,blue] (-0.18,-.2) to (-0.18,.25);
     \node at (0.18,0) {$\red{\dot}$};
\end{tikzpicture}}
-\mathord{
\begin{tikzpicture}[baseline = -1mm]
 	\draw[->,thin,red] (0.18,-.2) to (0.18,.25);
	\draw[->,thin,blue] (-0.18,-.2) to (-0.18,.25);
     \node at (-0.18,0.01) {$\blue{\dot}$};
\end{tikzpicture}}
\:$ 
lie in $\k^\times$.
Consequently, this endomorphism is invertible.
\end{proof}

Lemma~\ref{starbucks} shows for suitably generic $f(w), g(w)$
that there is a strict $\k$-linear
monoidal functor
$\Psi_f \;\overline{\odot}\; \Psi^\vee_g:
\blue{\HEIS_{-l}(z,u)} \;\overline{\odot}\; \red{\HEIS_m(z,v^{-1})} \rightarrow
\mathcal{E}nd_\k(\mathcal{V}(f|g))$.
Composing this functor with the functor $\Delta_{\blue{-l}|\red{m}}$ from
Theorem~\ref{comult}, we obtain a strict $\k$-linear monoidal
functor 
\begin{equation}\label{psifg}
\Psi_{f|g}:=
\Psi_f \;\overline{\odot}\; \Psi^\vee_g\circ \Delta_{\blue{-l}|\red{m}}:
\HEIS_k(z,t) \rightarrow
\mathcal{E}nd_\k\left(\mathcal{V}(f|g)\right).
\end{equation}
Thus, 
we have made $\mathcal V(f|g)$
into a module category over $\HEIS_k(z,t)$.

\begin{theorem}\label{chemistry}
Assume that $f(w), g(w)$ satisfy the genericity
assumption from Lemma~\ref{starbucks} so that (\ref{psifg}) is defined.
Let $\operatorname{Ev}:
\mathcal{E}nd_\k\left(\mathcal{V}(f|g)\right)
\rightarrow \mathcal{V}(f|g)$ be the $\k$-linear
functor defined by evaluation on 
$S := (H_0^f, H_0^g) \in \mathcal{V}(f|g)$.
The composition 
$\operatorname{Ev}\circ \Psi_{f|g}$ factors through the generalized
cyclotomic quotient
$\H(f|g)$ to induce an equivalence of $\HEIS_k(z,t)$-module categories
$$
\psi_{f|g}:\Kar\left(\H(f|g)\right) \rightarrow \mathcal{V}(f|g).
$$
\end{theorem}

\begin{proof}
We first show that $\Psi_{f|g}
\left(\,\anticlockplus\,(w)\,\right)_S\in w^k \End(S)\llbracket
w^{-1}\rrbracket$ equals $\A^+(w) 1_S$.
Recalling that $\A^+(w)$ is the expansion at $w=\infty$ of the rational
function $g(w) / f(w)$,
this follows because
$$
\Psi_{f|g}
\left(\,\anticlockplus\,(w)\,\right)_S= \Psi_f\left(\,{\color{blue}\anticlockplus}\,(w)\,\right)_{H_0^f}
\otimes
\Psi_g^\vee\left(\,{\color{red}\anticlockplus}\,(w)\,\right)_{H_0^g}
$$
thanks to (\ref{dt1}),
and also 
$\Psi_f\left(\,{\color{blue}\anticlockplus}\,(w)\,\right)_{H_0^f}=1 / f(w)$
and
$\Psi_g^\vee\left(\,{\color{red}\anticlockplus}\,(w)\,\right)_{H_0^g}= g(w)$.
To see the last two assertions, we first apply Lemma~\ref{clever} 
to see that $\mathcal I(f|1)$, the left tensor ideal of $\HEIS_{-l}(z,u)$ generated
by $f(x)$,
contains all coefficients of the series
${\color{blue}\anticlockplus}\,(w) - 1 / f(w) 1_\unit$; all elements of this ideal act as
zero on $H_0^f$ since its generator $f(x)$ acts as zero.
Then we apply Lemma~\ref{clever} again to see that $\mathcal I(1|g)$,
the left tensor ideal of $\HEIS_m(z,v^{-1})$ generated by $g(y)$,
contains all coefficients of 
${\color{red}\anticlockplus}\,(w) - g(w) 1_\unit$; all
elements of this act as zero on $H_0^g$.

The previous paragraph shows that $\anticlockplus{\scriptstyle n} - \A_n^+
1_\unit$ acts as zero on $S$ for all $n \in \Z$. Also it is obvious that $f(x)$ acts as
zero on $S$. So the left tensor ideal $\mathcal{I}(f|g)$
acts as zero on $S$, which proves that
$\operatorname{Ev}\circ \Psi_{f|g}$ factors through the quotient
$\H(f|g) = \HEIS_k(z,t)\big / \mathcal I(f|g)$ to induce a $\k$-linear functor
$\H(f|g) \rightarrow \mathcal{V}(f)\boxtimes \mathcal{V}(g)^\vee$.
Since $\mathcal{V}(f|g)$ is additive Karoubian, this extends to the Karoubi envelope to induce 
the functor $\psi_{f|g}$ from the statement of the
theorem.
Moreover, it is automatic from the definition that $\psi_{f|g}$ is a
morphism of $\HEIS_k(z,t)$-module categories.
It just remains to show that $\psi_{f|g}$ is an equivalence, which we
do by showing that it is full, faithful and dense.

First we show that $\psi_{f|g}$ is full and faithful.
It suffices to check this on objects
$X = X_r\otimes\cdots\otimes X_1$
and $Y = Y_s\otimes\cdots\otimes Y_1$ that are words in $\up$ and $\down$.
We assume moreover that $k \geq 0$; a similar
argument with the roles of $\up$ and $\down$ interchanged does the job
when $k \leq 0$ too.
Let $X^* = X_1^* \otimes\cdots\otimes X_r^*$ be the dual object
(here, $\up^* = \down,
\down^* = \up$).
By rigidity, we have a canonical isomorphism
$\Hom_{\mathcal H(f|g)}(X, Y) \cong \Hom_{\mathcal H(f|g)}(\unit, X^* \otimes Y)$,
  from which we get a commuting diagram
\begin{equation*}
\begin{CD}
\Hom_{\mathcal H(f|g)}\left(X,Y\right) &@>\sim>>&\Hom_{\mathcal H(f|g)}\left(\unit, X^* \otimes Y\right)\\
@V\psi_{f|g} VV&&@VV\psi_{f|g} V\\
\Hom_{\mathcal{V}(f|g)}\left(X \otimes S, Y \otimes S\right)
&@>\sim>>&\Hom_{\mathcal{V}(f|g)}\left(S, X^* \otimes Y
\otimes S\right).
\end{CD}
\end{equation*}
The left-hand vertical map in this diagram is an isomorphism if and
only if the right-hand vertical map is one.
We claim that the left-hand vertical map is an isomorphism 
when $X = Y = \up^{\otimes n}$. To prove this, the usual straightening algorithm (see the beginning of the proof of Theorem~\ref{basisthm} for details) shows that $\End_{\HEIS_k(z,t)}\left(\up^{\otimes n}\right)$ is spanned by diagrams in the image of the canonical homomorphism
$AH_n \rightarrow 
\End_{\HEIS_k(z,t)}\left(\up^{\otimes n}\right)$, with some number of bubbles added to the right-hand edge.  Thus we have an induced homomorphism $H_n^f
\rightarrow \End_{\mathcal H(f|g)}\left(\up^{\otimes n}\right)$
which is surjective since bubbles on the right-hand edge are
scalars in the generalized cyclotomic quotient.
On the other hand, $\End_{\mathcal{V}(f|g)}\left(\up^{\otimes
  n} \otimes S\right)
= \End_{H_n^f}\left(H_n^f\right) = H_n^f$. The claim follows.
Hence, the right-hand vertical map is an isomorphism
when $X^* \otimes Y = \down^{\otimes n} \otimes \up^{\otimes n}$.
Using this, we can show that the right hand vertical map is an
isomorphism in general. All of the morphism spaces are zero unless $X^*
\otimes Y$ has the same number of $\up$'s as $\down$'s. If all
$\down$'s are to the left of all $\up$'s, we are done already, so we
may assume that $X^* \otimes Y$ involves
$\up\otimes \down$ as a subword. Let $U$ be $X^* \otimes Y$ with the
two letters in this subword interchanged and $V$ be $X^* \otimes Y$
with these two letters deleted.
Using the isomorphism $\up\otimes \down \:\cong\: \down\otimes \up \oplus
\,\unit^{\oplus k}$ from (\ref{invrel}), we get a commuting diagram
\begin{equation*}
\begin{CD}
\Hom_{\mathcal H(f|g)}\left(\unit, X^* \otimes Y\right)&@>\sim>>&\Hom_{\mathcal H(f|g)}\left(\unit, U \oplus
V^{\oplus k}\right)\\
@V\psi_{f|g}VV&&@VV\psi_{f|g} V\\
\Hom_{\mathcal{V}(f|g)}\left(S, X^*\otimes Y \otimes
S\right)&@>\sim>>&\Hom_{\mathcal{V}(f|g)}\left(S, U \otimes S
\oplus V \otimes S^{\oplus k}\right).
\end{CD}
\end{equation*}
By induction, the right-hand vertical map is an isomorphism, hence, so
too is the left-hand one.

Finally, we explain why $\psi_{f|g}$ is dense.
Let $Q$ be an indecomposable object in $\mathcal{V}(f|g)$.
We have that $\down^{\otimes m} \otimes 
\up^{\otimes n} \otimes\, S= \down^{\otimes m} \otimes\, (H_n^f, H_0^g)
= (H_n^f, H_m^g) \oplus M$
where $M$ is a direct sum of summands of
$(H_{n'}^f, H_{m'}^g)$ with $n' < n$ and $m' < m$.
It follows that $Q$
is isomorphic to the image of some
idempotent in
$\End_{\mathcal{V}(f|g)}\left(\down^{\otimes m} \otimes 
\up^{\otimes n} \otimes S\right)$ for some $m,n\geq 0$.
Since we have shown already that $\psi_{f|g}$ is full and faithful,
there is a corresponding idempotent in
$\End_{\mathcal H(f|g)}\left(\down^{\otimes m} \otimes \up^{\otimes n}\right)$. The latter
idempotent defines an object $P$ of $\Kar\left(\mathcal H(f|g)\right)$ such that
$\psi_{f|g}(P) \cong Q$.
\end{proof}

\begin{remark}
If $g(w) = 1$ the genericity assumption is vacuous, so Theorem~\ref{chemistry}
gives us an equivalence of categories
$\psi_{f|1}: \Kar\left(\H(f|1) \right) \rightarrow \mathcal{V}(f)$.
In other words, the generalized cyclotomic quotient
$\H(f|1)$
is Morita
equivalent to the ``usual'' cyclotomic quotient defined by the
cyclotomic Hecke algebras $H_n^f$ for all $n \geq 0$.
This statement is the quantum analog of \cite[Theorem 1.7]{Bheis}; see
also \cite[Theorem 4.25]{Rou2} for the analogous result in the setting
of Kac-Moody 2-categories.
\end{remark}

\begin{remark}
More generally, suppose that there are factorizations $f(w)=f_1(w)f_2(w)$ and $g(w)=g_1(w)g_2(w)$ such that the genericity assumption 
$\lambda
\mu^{-1} \notin \left\{q^{2i}\:\big|\:i \in \Z\right\}$ 
holds for $\lambda$ a root of $f_1(w)$ or $g_1(w)$, and $\mu$ a root of $f_2(w)$ or $g_2(w)$.  Then a similar argument to the proof of Theorem~\ref{chemistry} can be used to show that the categories $\Kar\left(\H(f|g)\right)$ and $\Kar\left(\H(f_1|g_1)\boxtimes \H(f_2|g_2)\right)$ are equivalent.
In particular, applying this to $\Kar\left(\H(f|1)\right)$ 
and using the previous remark, it follows that the cyclotomic Hecke algebra $H^{f}_n$ is Morita equivalent to
$\bigoplus_{n_1+n_2=n} H^{f_1}_{n_1}\otimes H_{n_2}^{f_2}$, thereby recovering
 a result of Dipper and Mathas \cite{DM}.
\end{remark}

\section{Basis theorem}\label{sbasis}

Finally, we 
prove a basis theorem for the morphism spaces in $\HEIS_k(z,t)$.
Our proof of this is very similar to the argument in the degenerate case from \cite[Theorem 6.4]{BSW1}.
Let $X = X_r\otimes \cdots\otimes X_1$ and $Y = Y_s \otimes\cdots
\otimes Y_1$ be objects of
$\HEIS_k(z,t)$ for $X_i, Y_j \in \{\up,\down\}$.
An {\em $(X,Y)$-matching} is a bijection between 
$\{i\:|\:X_i = \up\}\sqcup\{j\:|\:Y_j = \down\}$
and $\{i\:|\:X_i = \down\}\sqcup\{j\:|\:Y_j = \up\}$.
A {\em reduced lift} of an $(X,Y)$-matching means a diagram
representing a morphism $X \rightarrow Y$ such that
\begin{itemize}
\item
the endpoints of each string are points which
correspond under the given matching;
\item
there are no floating bubbles and no dots on any string;
\item there are no self-intersections of strings and
no two strings cross each other more than once.
\end{itemize}
Fix a set $B(X,Y)$ consisting of a choice of reduced lift for each of the
$(X,Y)$-matchings.
Let $B_{\circ}(X,Y)$ be the set of all morphisms that can be obtained 
from the elements of $B(X,Y)$ by adding dots labelled with integer multiplicities 
near to the terminus of each string.
Also recall the homomorphism $\beta:\Sym\otimes \Sym \rightarrow
\End_{\HEIS_k(z,t)}(\unit)$ from (\ref{beta}). Using it, we can make the
morphism
space $\Hom_{\HEIS_k(z,t)}(X,Y)$ into a right $\Sym\otimes
\Sym$-module: 
$\phi \theta := \phi \otimes \beta(\theta)$.

\begin{theorem} \label{basisthm}
For any ground ring $\k$, 
parameters $z, t \in \k^\times$, and objects $X, Y \in \HEIS_k(z,t)$,
the morphism space $\Hom_{\HEIS_k(z,t)}(X,Y)$ is a free right
$\Sym\otimes \Sym$-module with basis $B_{\circ}(X,Y)$.
\end{theorem}

\begin{proof}
We just prove this when $k \leq 0$; the result for $k \geq 0$ then follows
by applying $\Omega_k$.
Let $X = X_r\otimes \cdots \otimes X_1$ and $Y = Y_s\otimes
\cdots\otimes Y_1$ be two objects.

We first observe that $B_{\circ}(X,Y)$ spans
$\Hom_{\HEIS_k(z,t)}(X,Y)$ as a right $\Sym\otimes \Sym$-module.
The defining relations and the
additional relations derived in sections \ref{first}, \ref{second} and
\ref{third}
give Reidemeister-type relations
modulo terms with fewer crossings, plus a skein relation
and bubble and dot sliding relations. These relations allow
diagrams for morphisms in $\HEIS_k(z,t)$ to be transformed in a similar way to
the way oriented tangles are simplified in skein categories, 
modulo diagrams with fewer crossings.
Hence, there is a straightening algorithm to rewrite
any diagram representing a morphism $X \rightarrow Y$ as a linear
combination of the ones in $B_{\circ}(X,Y)$.

It remains to prove the linear independence. 
We say $\phi \in B_{\circ}(X,Y)$ is {\em positive} if it only involves
non-negative powers of dots.
It suffices to show just
that the positive morphisms in $B_{\circ}(X,Y)$ are linearly independent.
Indeed, given any linear relation of the form $\sum_{i=1}^N \phi_i \otimes
\beta(\theta_i) = 0$
for morphisms $\phi_i \in B_{\circ}(X,Y)$ and
coefficients
$\theta_i \in \Sym\otimes \Sym$, we can ``clear denominators'' by
multiplying the termini of the strings 
by sufficiently large
positive powers of dots to reduce to the positive case.

The main step now is to prove the linear independence 
in the special case that $X = Y = \up^{\otimes n}$.
To do this, we need to allow the ground ring $\k$ to change, so we
will add a subscript to our notation, denoting $\HEIS_k(z,t), \mathcal{V}(f|g), \Sym
\otimes \Sym, \dots$ by
${_\k}\HEIS_k(z,t), {_\k\!}\mathcal{V}(f|g), {_\k\!}\Sym \otimes_\k\, {_\k\!}\Sym, \dots$ to avoid any confusion.
It suffices to prove the linear independence
of positive elements of $B_\circ(X,Y)$ in the special case that $\k = \Z[z^{\pm 1}, t^{\pm
  1}]$; one can then use
the canonical $\k$-linear monoidal
functor ${_\k}\HEIS_k(z,t) \rightarrow \k \otimes_{\Z[z^{\pm 1}, t^{\pm 1}]}  {_{\Z[z^{\pm
    1}, t^{\pm 1}]}}\HEIS_k(z,t)$ to deduce the linear
independence over an arbitrary ground ring $\k$ and 
for arbitrary parameters.

So assume now that $\k = \Z[z^{\pm 1}, t^{\pm 1}]$
and take a linear relation $\sum_{i=1}^N \phi_i \otimes \beta(\theta_i)=0$
for positive $\phi_i \in B_\circ(X,Y)$.
Choose $a$ so that the multiplicities of dots in all $\phi_i$ arising in
this linear relation
are $\leq a$.
Also choose 
$b,c \geq 0$ so that all of the symmetric functions $\theta_i \in
{_\k\!}\Sym\otimes_\k\,{_\k\!}\Sym$ 
are polynomials in 
the elementary symmetric functions 
$\e_1 \otimes 1,\dots,\e_b\otimes 1$ and $1 \otimes
\e_1,\dots,1\otimes \e_c$.
Then choose $l, m$ so that 
$a < l$, $b+c < m$ and $k = m-l$. Note that $l \geq m$ due to our
standing assumption that $k \leq 0$.
Let $u_1,\dots,u_b$ and $v_1,\dots,v_c$ be indeterminates
and $\K$ be the algebraic closure of the field
$\Q(z,t,u_1,\dots,u_b,v_1,\dots,v_c)$. 
Pick $q \in \K^\times$ so that $z = q-q^{-1}$ and consider the
cyclotomic Hecke algebras ${_\K}H_n^f$ and ${_\K}H_n^g$ over $\K$
associated to the polynomials
\begin{align*}
f(w) &:= w^l + t^2,
&
g(w) = w^m+u_1 w^{m-1}+\cdots+u_b w^{m-b} + v_c w^c + \cdots + v_1 w +
  1.
\end{align*}
Note the formula for $g(w)$ makes sense because $b+c<m$.
Consider the
${_\K}\HEIS_k(z,t)$-module category
${_\K}\mathcal{V}(f|g)$ from (\ref{psifg}) (taking $u := t$ and $v := 1$).
Since $\k\hookrightarrow \K$, there is a canonical $\k$-linear
monoidal
functor ${_\k}\HEIS_k(z,t) \rightarrow {_\K}\HEIS_k(z,t)$, allowing us to
view
${_\K}\mathcal{V}(f|g)$ also as a module
category over ${_\k}\HEIS_k(z,t)$.
Then we evaluate the relation 
$\sum \phi_i \otimes \beta(\theta_i)=0$ on
${_\K}S:=({_\K}H^f_0,{_\K}H^g_0)$ 
to obtain a relation in 
${_\K}H_n^f$. By the basis theorem for ${_\K}H_n^f$ from
(\ref{akbasis}) and the assumption that $a < l$,  
the images of
 $\phi_1,\dots,\phi_N$ in ${_\K}H_n^f$ are linearly independent over $\K$, so we deduce
that
the image of $\beta(\theta_i)$ in $\K$ is zero for each $i$.
To deduce from this that $\theta_i = 0$, recall
that $\theta_i$ is a polynomial in
$\e_1\otimes 1,\dots,\e_b \otimes 1, 1 \otimes \e_1,\dots,1\otimes
\e_c$. So we need to show that the images of $\beta(\e_1\otimes
1),
\dots, \beta(\e_b \otimes 1), \beta(1 \otimes
\e_1),\dots,\beta(1\otimes \e_c)$ in $\K$
are algebraically independent.
In fact, we claim that these images are the
indeterminates
$u_1,\dots,u_b,v_1,\dots,v_c$, respectively.
To prove this, note that the low degree terms of $\A^{\pm}(w)$ are
\begin{align*}
\A^+(w) &= g(w)/ f(w) = w^k+u_1 w^{k-1}+\cdots+u_b
w^{k-b}+\cdots \in w^k \K\llbracket w^{-1}\rrbracket,\\
\A^-(w) &= t^2 g(w) / f(w) = 1+v_1 w+\cdots+v_c
w^{c}+\cdots \in \K\llbracket w \rrbracket.
\end{align*}
By (\ref{new2}), (\ref{am})--(\ref{amm})
and Lemma~\ref{clever}, the images of
$\beta(\e_n \otimes 1)$ and $\beta(1 \otimes \e_n)$ are
the $w^{k-n}$- and $w^n$-coefficients of $\A^+(u)$ and
 $\A^-(u)$, respectively, and the claim follows.
 
We have now proved the linear independence when $X = Y = \up^{\otimes
  n}$.
Returning to the general case, we can use the canonical isomorphism
$\Hom_{\HEIS_k(z,t)}(X,Y) \cong \Hom_{\HEIS_k(z,t)}(\unit, X^* \otimes Y)$
arising from the rigidity to see that the $\Sym\otimes\Sym$-linear independence of the
positive
morphisms in $B_\circ(X,Y)$ is equivalent to the
$\Sym\otimes\Sym$-linear independence of the positive morphisms in
$B_\circ(\unit, X^*\otimes Y)$. 
Thus, we are reduced to the case that $X = \unit$. Assume this from
now on.
The set $B_\circ(\unit, Y)$ is empty
unless $Y$ has the same number $n$ of $\up$'s as $\down$'s. Also we have already proved the linear
independence in the case $Y = \down^{\otimes n} \otimes
\up^{\otimes n}$. So we may assume that $Y$ has
a subword $\up \otimes \down$.
Let $Z$ be $Y$ with the two letters in the subword
interchanged.
By induction, we may assume the linear independence has already been
established for $B_\circ(\unit,Z)$.
Consider a linear relation $\sum_{i=1}^N \phi_i \otimes \beta(\theta_i)$ for positive $\phi_i \in
B_\circ(\unit, Y)$.
Recalling the isomorphism
$\up\otimes\down\oplus \:\unit^{\oplus (-k)}
\stackrel{\sim}{\rightarrow} \:
\down\otimes \up$ from
(\ref{invrel1}), multiplying the subword $\up\otimes\down$ on top by
the sideways crossing $\mathord{
\begin{tikzpicture}[baseline = -.5mm]
	\draw[<-,thin,darkblue] (0.2,-.2) to (-0.2,.3);
	\draw[line width=4pt,white,-] (-0.2,-.2) to (0.2,.3);
	\draw[thin,darkblue,->] (-0.2,-.2) to (0.2,.3);
\end{tikzpicture}
}$
defines a $\Sym\otimes\Sym$-linear map
$s:\Hom_{\HEIS_k(z,t)}(\unit, Y) \hookrightarrow \Hom_{\HEIS_k(z,t)}(\unit, Z).$
Unfortunately, $s$ does not send $B_\circ(\unit, Y)$ into
$B_\circ(\unit, Z)$. However, the image of $B_\circ(\unit,Y)$ is related to
$B_\circ(\unit,Z)$ in a triangular way, which is good enough to complete
the argument. The full explanation of this 
is almost exactly the same as in the degenerate case, so we refer the reader to the last paragraph of the proof of \cite[Theorem 6.4]{BSW1} for the details.
\end{proof}

\begin{corollary}\label{cor:beta-iso}
$\End_{\HEIS_k(z,t)}(\unit) \cong \Sym\otimes \Sym$.
\end{corollary}


\begin{thebibliography}{CLLSS}
\bibitem[AK]{AK}
S. Ariki and K. Koike,
A Hecke algebra of $({Z}/r{Z})\wr{S}\sb n$ and construction of its irreducible representations, 
{\em Advances Math.} {\bf 106} (1994), 216--243. 

\bibitem[BE]{BE}
R. Bezrukavnikov and P. Etingof, Parabolic induction and restriction functors for rational Cherednik algebras, {\em Selecta Math.} {\bf 14} (2009), 397--425.

\bibitem[B1]{Bheis}
J. Brundan,
On the definition of Heisenberg category, 
{\em Alg. Comb.} {\bf 1} (2018), 523--544.

\bibitem[B2]{Bskein}
\bysame,
Representations of the oriented skein category; {\tt arXiv:1712.08953}.

\bibitem[BCNR]{BCNR}
J. Brundan, J. Comes, D. Nash and A. Reynolds,
A basis theorem for the affine oriented Brauer category and its
cyclotomic quotients,
{\em Quantum Topology} {\bf 8} (2017), 75--112.

\bibitem[BD]{BD}
J. Brundan and N. Davidson,
Categorical actions and crystals,
{\em Contemp. Math.} {\bf 683} (2017), 105--147.

\bibitem[BKM]{BKM}
J. Brundan, A. Kleshchev and P. McNamara,
Homological properties of finite type Khovanov-Lauda-Rouquier
algebras,
{\em Duke Math. J.} 
{\bf 163} (2014), 1353--1404.

\bibitem[BSW1]{BSW1}
J. Brundan, A. Savage and B. Webster,
The degenerate Heisenberg category and its Grothendieck ring;
{\tt arXiv:1812.03255}.

\bibitem[BSW2]{BSW2}
\bysame,
Heisenberg and Kac-Moody categorification;
{\tt arXiv:1907.11988}.

\bibitem[BSW3]{BSW3}
\bysame,
Quantum Frobenius Heisenberg categorification, in preparation.

\bibitem[CLLSS]{CLLSS}
S. Cautis, A. Lauda, A. Licata, P. Samuelson and J. Sussan, 
The elliptic Hall algebra and the deformed Khovanov Heisenberg category, {\em Selecta Math.} {\bf 24} (2018), 4041--4103.

\bibitem[DM]{DM}
R. Dipper and A. Mathas,
Morita equivalences of Ariki-Koike algebras, \emph{Math. Z.} \textbf{240} (2002), 579--610.

\bibitem[GGOR]{GGOR}
V. Ginzburg, N. Guay, E. Opdam and R. Rouquier, On
  the category {$\mathcal O$} for rational {C}herednik algebras, \emph{Invent. Math.}
  \textbf{154} (2003), 617--651.

\bibitem[GZB]{GZB}
M. Gould, R. B. Zhang and A. Bracken, 
Generalized Gelfand invariants and characteristic identities for
quantum groups,
{\em J. Math. Phys.} {\bf 32} (1991), 2298--2303.

\bibitem[J]{Jantzen}
J. C. Jantzen,
{\em Lectures on Quantum Groups}, 
American Mathematical Society, Providence, RI, 1996.

\bibitem[K]{K}
M. Khovanov,
Heisenberg algebra and a graphical calculus,
{\em Fund. Math.} {\bf 225} (2014), 169--210.

\bibitem[Li]{Li}
J. Li,
The quantum Casimir operators of $\Uq(\mathfrak{gl}_n)$ and their
eigenvalues,
{\em J. Phys. A} {\bf 43} (2010), 345202, 9 pp..

\bibitem[LS]{LS}
A. Licata and A. Savage,
Hecke algebras, finite general linear groups, and Heisenberg
categorification,
{\em Quantum Topology} {\bf 4} (2013), 125--185.

\bibitem[Lu]{Lubook}
G. Lusztig, {\em Introduction to Quantum Groups}, Birkh\"auser, 1993.

\bibitem[M]{Mac}
I. G. Macdonald, {\em Symmetric Functions and Hall Polynomials},
Oxford Mathematical Monographs, second edition, OUP, 1995.

\bibitem[MS]{MS}
M. Mackaay and A. Savage,
Degenerate cyclotomic Hecke algebras and higher level Heisenberg
categorification,
{\em J. Algebra} {\bf 505} (2018), 150--193.

\bibitem[R]{Rou2}
R. Rouquier,
Quiver Hecke algebras and $2$-Lie algebras,
{\em Algebra Colloq.}
{\bf 19} (2012),
359--410.

\bibitem[Sa]{Savage}
A. Savage,
Frobenius Heisenberg categorification,
{\em. Alg. Comb.}
{\bf 2}:5 (2019),
937--967.

\bibitem[Sh]{Shan}
P. Shan, Crystals of {F}ock spaces and cyclotomic rational double
  affine {H}ecke algebras, {\em Ann. Sci. \'Ec. Norm. Sup\'er.} \textbf{44}
  (2011), 147--182. 

\bibitem[T]{Turaev1}
V. Turaev,
Operator invariants of tangles, and $R$-matrices,
{\em Math. USSR Izvestiya} {\bf 35} (1990), 
411--444.

\bibitem[V]{V} R. Virk, Derived equivalences and
  $\mathfrak{sl}_2$-categorifications for $\Uq(\mathfrak{gl}_n)$, {\em
    J. Algebra} {\bf 346} (2011), 82--100.

\bibitem[W1]{Wcanonical}
B. Webster, 
Canonical bases and higher representation theory, 
{\em Compositio Math.} {\bf 151} (2015), 121--166.

\bibitem[W2]{Wunfurling}
\bysame,
Unfurling Khovanov--Lauda--Rouquier algebras;
\arxiv{1603.06311}.
\end{thebibliography}
\end{document}